%% file: main.tex
\documentclass[%preprint, draft, 
10pt]{article}

\input{input/packages}

\input{input/befehle}

\bibliography{main}%

\begin{document}

%
%	TITLE PAGE
%
	
%{\centering
\begin{center}	
	\LARGE{Numerical analysis for a Cahn--Hilliard system modelling tumour growth with chemotaxis and active transport}
\end{center}
%\par}
%\vfill
\bigskip

\begin{center}	
	\normalsize{Harald Garcke}\\[1mm]
	\textit{Fakult\"at f\"ur Mathematik, Universit\"at Regensburg, 93053 Regensburg, Germany}\\[1mm]
	\texttt{Harald.Garcke@ur.de}
\end{center}
%\vfill
%\bigskip

\begin{center}	
	\normalsize{Dennis Trautwein}\\[1mm]
	\textit{Fakult\"at f\"ur Mathematik, Universit\"at Regensburg, 93053 Regensburg, Germany}\\[1mm]
	\texttt{Dennis.Trautwein@ur.de}
	%%%%%%%%%%%%%%%ORCID-ID%%%%%%%%%%%%%%%%%%%%%%%%
%	\\[-3mm]
%	\begin{minipage}[h]{0.225\textwidth}
%		\begin{flushright}
%			\vspace{-2pt}
%			\includegraphics[scale=0.05]{orcid.pdf} 
%		\end{flushright}
%	\end{minipage}
%	\begin{minipage}[h]{0.5\textwidth}
%		\hspace{-12pt}
%		\href{https://orcid.org/0000-0003-4115-4885}{orcid.org/0000-0003-4115-4885}
%	\end{minipage}
\end{center}

%\vfill
\bigskip

\begin{abstract}
In this work, we consider a diffuse interface model for tumour growth in the presence of a nutrient which is consumed by the tumour. The system of equations consists of a Cahn--Hilliard equation with source terms for the tumour cells and a reaction-diffusion equation for the nutrient. 
We introduce a fully-discrete finite element approximation of the model and prove stability bounds for the discrete scheme. Moreover, we show that discrete solutions exist and depend continuously on the initial and boundary data. We then pass to the limit in the discretization parameters and prove convergence to a global-in-time weak solution to the model. Under additional assumptions, this weak solution is unique.
Finally, we present some numerical results including numerical error investigation in one spatial dimension and some long time simulations in two and three spatial dimensions.
\pskip

\pskip
	
\end{abstract}

%\vfill

% Kapitel
%\input{test}
\input{results/introduction}

\input{results/fem_1_setting}
\input{results/fem_2_stability}

\input{results/fem_3_existence}
\input{results/fem_4_higher_order}

\input{results/fem_5_convergence}

\input{results/fem_6_numeric}

%\newpage
%\cleardoublepage
\printbibliography
%\addcontentsline{toc}{section}{References}

%\footnotesize
%\nocite{*} % Show all bib entries - both cited and uncited; comment this line to view only cited bib entries;

\end{document}

%% file: input/packages.tex
\usepackage[left=1in, right=1in, bottom=1in, top=1in, footskip=0.3in, %hoffset=20pt
]{geometry}

\usepackage[svgnames]{xcolor}
\definecolor{LinkColor}{rgb}{0,0,1}
\definecolor{LinkColor2}{rgb}{1,0,0}
\definecolor{lbcolor}{rgb}{0.85,0.85,0.85}
\definecolor{FrameColor}{rgb}{0.85,0.85,0.85}

\usepackage[ngerman, english]{babel}
\usepackage[utf8]{inputenc}
\usepackage{enumerate}
\usepackage[normalem]{ulem}
\usepackage{setspace}

\def\pskip{\\[-3mm]}

\usepackage{multicol}
\usepackage[babel,french=guillemets,german=swiss]{csquotes}

\usepackage{array}
\usepackage{booktabs}
\usepackage{framed}
\usepackage{rotating}
\usepackage{longtable}
\usepackage{multirow}
\usepackage{tabularx}
\newcolumntype{L}[1]{>{\raggedright\arraybackslash}p{#1}} % linksb?ndig mit Breitenangabe
\newcolumntype{C}[1]{>{\centering\arraybackslash}p{#1}} % zentriert mit Breitenangabe
\newcolumntype{R}[1]{>{\raggedleft\arraybackslash}p{#1}} % rechtsb?ndig mit Breitenangabe

\usepackage{graphicx}
\usepackage{caption}
\usepackage{amssymb}
\usepackage{dsfont}
\usepackage{nicefrac}
\usepackage{empheq}
\allowdisplaybreaks

\usepackage[%
pdftitle={Titel},% 
pdfauthor={Autor},% 
pdfcreator={LaTeX, LaTeX with hyperref and KOMA-Script},% Genutzte Programme
pdfsubject={Betreff}, % Betreff
pdfkeywords={Keywords}
]{hyperref} 

\hypersetup{%
	colorlinks=true,% Definition der Links im PDF File
	linkcolor=LinkColor,%
	anchorcolor=LinkColor,%
	citecolor=LinkColor2,%
	filecolor=LinkColor,%
	menucolor=LinkColor,%
	pagecolor=LinkColor,%
	urlcolor=LinkColor,%
}

\usepackage[savemem]{listings}
\usepackage{paralist}
{\begin{list}{$\diamondsuit$}{}}%
	{\end{list}}

\usepackage{amsthm}
\usepackage{upgreek}

\newtheoremstyle{tstyle}% name
{15pt}	% Space above
{5pt}	% Space below 
{\itshape}	% Body font
{}	% Indent amount: Indent amount: empty = no indent, \parindent = normal paragraph indent
{\bfseries}	% Theorem head font
{.}	% Punctuation after theorem head
{0.5em}	% Space after theorem head: Space after theorem head: { } = normal interword space; \newline = linebreak
{}	% Theorem head spec (can be left empty, meaning `normal')

\theoremstyle{tstyle}

\newtheorem{theorem}{Theorem}[section]
\newtheorem{lemma}[theorem]{Lemma}

\newtheorem{remark}[theorem]{Remark}

\newtheoremstyle{cstyle}% name
{15pt}	% Space above
{5pt}	% Space below 
{}	% Body font
{}	% Indent amount: Indent amount: empty = no indent, \parindent = normal paragraph indent
{\bfseries}	% Theorem head font
{}	% Punctuation after theorem head
{0.2222em}	% Space after theorem head: Space after theorem head: { } = normal interword space; \newline = linebreak
{}	% Theorem head spec (can be left empty, meaning `normal')

\theoremstyle{cstyle}

\makeatletter
\g@addto@macro{\thm@space@setup}{\thm@headpunct{}}
\renewenvironment{proof}[1][\proofname]{\par
	\pushQED{\qed}%
	\normalfont \topsep6\p@\@plus6\p@\relax
	\trivlist
	\item[\hskip\labelsep
	\bfseries
	#1\@addpunct{\,}]\ignorespaces
}{%
	\popQED\endtrivlist\@endpefalse
}
\makeatother

\makeatletter
\g@addto@macro{\thm@space@setup}{\thm@headpunct{}}
\newenvironment{sketch-proof}[1][Sketch of the proof]{\par
	\pushQED{\qed}%
	\normalfont \topsep6\p@\@plus6\p@\relax
	\trivlist
	\item[\hskip\labelsep
	\bfseries
	#1\@addpunct{\,}]\ignorespaces
}{%
	\popQED\endtrivlist\@endpefalse
}
\makeatother

\usepackage{mathtools}  % -> eigene Befehle mit \DeclarePairedDelimiterX
\usepackage{amsmath}
\usepackage{fancyhdr}
\usepackage{tikz}
\usepackage{pdflscape}
\usepackage{subcaption}
\usepackage{float}
\usepackage{hyperref}
\usepackage{tikz}
\usepackage{pgfplots}
\pgfplotsset{compat=1.16} 

\usepackage{verbatim}
\usepackage[section]{placeins}         % figures stay in same section

\usepackage[backend=biber, style=numeric, maxbibnames=99]{biblatex} % für Literaturverzeichnis

\numberwithin{equation}{section}

%% file: input/befehle.tex
\makeatletter
\newcommand{\mylabel}[2]{%
   \protected@write \@auxout {}{\string \newlabel {#1}{{#2}{\thepage}{#2}{#1}{}} }%
   \hypertarget{#1}{#2}
}
\newcommand{\mylabelHIDE}[2]{%
   \protected@write \@auxout {}{\string \newlabel {#1}{{#2}{\thepage}{#2}{#1}{}} }%
   \hypertarget{#1}{}
}
\makeatother

\newcommand{\pmbn}{\pmb{\mathrm{n}}}

\renewcommand{\rho}{\varrho}
\renewcommand{\phi}{\varphi}

\newcommand{\N}{\ensuremath{\mathbb{N}}}   
\newcommand{\R}{\ensuremath{\mathbb{R}}}

\newcommand{\mycal}[1]{\ensuremath{\mathcal{#1}}}

\newcommand{\calF}{\ensuremath{\mathcal{F}}}

\newcommand{\calH}{\ensuremath{\mathcal{H}}}
\newcommand{\calI}{\mycal{I}}

\newcommand{\calO}{\ensuremath{\mathcal{O}}}

\newcommand{\calP}{\ensuremath{\mathcal{P}}}
\newcommand{\calS}{\ensuremath{\mathcal{S}}}
\newcommand{\calT}{\ensuremath{\mathcal{T}}}

\newcommand{\divergenzz}[1]{\operatorname{div} \big({#1}\big)}

\DeclarePairedDelimiterX{\abs}[1]{\lvert}{\rvert}{#1}
\DeclarePairedDelimiterX{\aabs}[1]{\big\lvert}{\big\rvert}{#1}
\DeclarePairedDelimiterX{\aaabs}[1]{\Big\lvert}{\Big\rvert}{#1}

\DeclarePairedDelimiterX{\norm}[1]{\lVert}{\rVert}{#1}
\DeclarePairedDelimiterX{\nnorm}[1]{\big\lVert}{\big\rVert}{#1}
\DeclarePairedDelimiterX{\nnnorm}[1]{\Big\lVert}{\Big\rVert}{#1}

\DeclarePairedDelimiterX{\seminorm}[1]{\lvert}{\rvert}{#1}
\DeclarePairedDelimiterX{\sseminorm}[1]{\big\lvert}{\big\rvert}{#1}
\DeclarePairedDelimiterX{\ssseminorm}[1]{\Big\lvert}{\Big\rvert}{#1}

\DeclarePairedDelimiterX{\skp}[2]{(}{)}{#1, #2}
\DeclarePairedDelimiterX{\sskp}[2]{\big(}{\big)}{#1, #2}
\DeclarePairedDelimiterX{\ssskp}[2]{\Big(}{\Big)}{#1, #2}
\DeclarePairedDelimiterX{\dualp}[2]{\big<}{\big>}{#1, #2}

\newcommand{\dv}[1]{\,{\mathrm d}#1}
\newcommand{\dx}{\dv{x}}

\newcommand{\dt}{\dv{t}}

\newcommand{\dH}{\ \mathrm d \calH}

\newcommand{\ddv}[2]{ \frac{{\mathrm d}#1}{{\mathrm d}#2} }

\newcommand{\ddt}{\ddv{}{t}}

\newcommand{\red}{}%{\color{red}}
\newcommand{\blk}{}%{\color{black}}

%% file: results/introduction.tex
\section{Introduction}

Mathematical models which are based on continuum modelling go back to the work of Greenspan \cite{greenspan} who used a description based on free boundary problems to model tumour growth. Such modelling approaches have been later further developed by many authors and we refer to \cite{ambrosi, byrne_chaplain_1997} and the reviews \cite{bellomo_tumour, friedman_tumour, roose_tumour}. However, in recent years also diffuse interface descriptions have been used to model tumour growth. In these models the interface between tumour tissue and healthy tissue is modelled with the help of a phase field function which changes its value in a narrow transition layer and attains the value $+1$ in the tumour ``phase'' and $-1$ in the healthy ``phase''. These models go back to Cristini, Lowengrub and co-authors, see \cite{cristini_lowengrub_2010, frieboes_lowengrub_2007, wise_lowengrub_2008} and have been further developed by many authors, see, e.g., \cite{colli_2017_optimal, DFRSS_2017_analysis_multispecies, ebenbeck_garcke_nurnberg_2020, eyles_2019,  frigeri_2017_tumour_degenerate, garcke_lam_2017_dirichlet, garcke_lam_2017, GLNS_2018_multiphase_tumour_necrosis, garcke_lam_2016_optimal, garcke_lam_signori_2021, GarckeLSS_2016, hawkins_2012, krejci_2021, hawkins_2010}.

\red 
Here, the Cahn--Hilliard equation, which is a partial differential equation of fourth order, often plays an important role. The numerical approximation of Cahn--Hilliard systems is often made with finite element methods that are based on the works \cite{elliott_kinetics, elliott_second_order} and have found many applications in physics \cite{barrett_nurnberg_styles_2004, gruen_2003}.
In the past years, also numerical methods for tumour growth models have been proposed \cite{cristini_li_lowengrub_wise_2009, du_feng_2020} but the numerical analysis such as stability or convergence analysis is often missing.

The biological effect of our main interest is the process of chemotaxis which describes the movement of tumour cells towards regions with a higher concentration of an extracellular chemical species, which can favour an instable tumour growth and lead to invasion \cite{chemotaxis_in_cancer_2011}.
For chemotaxis models, often the Patlak--Keller--Segel system is proposed in the literature which is composed of a system of partial differential equations of second order, see \cite{hillen_painter_2009}. 
Here, numerical approximations are often based on finite volume or finite element methods which have been studied in, e.g., \cite{chertock_kurganov_2008, epshteyn2009fully, filbet_2006, gurusamy2018finite, marrocco_2003, saito_2007, saito_2012, strehl2013positivity, zhang2016characteristic}.
For a good overview about the most relavant works about Keller--Segel systems, we also refer the reader to the review \cite{arumugam_tyagi_2021}.
\blk

% In our work, we present a well-posed, stable and convergent fully-discrete scheme for a phase-field system for tumour growth.
% For $d=1$, we verify the convergence property of our scheme by computing experimental orders of convergence. We illustrate the practicability with a long-time simulation for $d=2$ and provide numerical examples for $d=3$ which, up to our knowledge, have been missing in the literature so far.

In this work, we study a phase-field model for tumour growth with chemotaxis and active transport from the numerical point of view. 
In particular, we introduce and study a fully-discrete finite element approach, for which we establish stability, existence, continuous dependence and convergence results and present numerical examples to illustrate the practicability of the method.
%In this paper we study a finite element approximation of a system of partial differential equations modelling tumour growth with chemotaxis and active transport. 
The mathematical model of our interest was originally introduced in \cite{GarckeLSS_2016} and analyzed in \cite{garcke_lam_2017} and consists of a Cahn--Hilliard
equation with source terms and a parabolic reaction-diffusion equation for a nutrient species, given by 
\begin{subequations}
\label{eq:all}
\begin{align}
    \label{eq:phi}
	\partial_t \phi &= \divergenzz{m(\phi)\nabla\mu} 
	+ \Gamma_\phi(\phi,\mu,\sigma) 
	&\text{ in } \Omega\times (0,T),
	\\
	\label{eq:mu}
	\mu &= A \psi^{\prime}(\phi) 
	- B\triangle\phi 
	- \chi_\phi\sigma 
	&\text{ in } \Omega\times (0,T),
	\\
	\label{eq:sigma}
	\partial_t \sigma 
	&= \divergenzz{n(\phi)(\chi_\sigma\nabla\sigma - \chi_\phi\nabla\phi)} 
	- \Gamma_\sigma(\phi,\mu,\sigma) 
	&\text{ in } \Omega\times (0,T),
	\\
	\label{eq:bc_neumann}
	0 &= \nabla\phi\cdot\pmbn = \nabla\mu\cdot\pmbn &\text{ on } {\partial\Omega}\times (0,T),
	\\
	\label{eq:bc_robin}
	n(\phi)\chi_\sigma\nabla\sigma\cdot\pmbn  &= K(\sigma_\infty - \sigma) &\text{ on } {\partial\Omega}\times (0,T),
\end{align}
\end{subequations}
where $\Omega\subset\R^d$, $d\in\N$, is a bounded domain with Lipschitz boundary $\partial\Omega$ and outer unit normal $\pmbn$. 
%%%
Here, the phase field variable $\phi\in[-1,1]$ denotes the difference in volume fractions, with $\{\phi=1\}$ describing unmixed tumour tissue, and $\{\phi=-1\}$ representing the surrounding healthy tissue. By $\mu$, we denote the chemical potential for $\phi$. Furthermore, $\sigma\geq0$ denotes the concentration of an unspecified chemical species (like oxygen or glucose) that serves as a nutrient for the tumour. Moreover, $\sigma_\infty\geq0$ and $K\geq 0$ denote a given nutrient supply on the boundary and a permeability constant, respectively.
%%%
The source and sink term $\Gamma_\phi$ in the phase field equation \eqref{eq:phi} may describe proliferation or apoptosis of the tumour cells, while $\Gamma_\sigma$ in the nutrient equation \eqref{eq:sigma} models effects like nutrient consumption or a nutrient supply from an existing vasculature.
%%%
The positive parameter $\chi_\sigma>0$ models the diffusivity of the nutrient and the non-negative parameter $\chi_\phi\geq 0$ refers to transport mechanisms such as chemotaxis and active uptake.
%%%
In the system \eqref{eq:all}, $m(\phi)$ and $n(\phi)$ denote positive mobilities for $\phi$ and $\sigma$, respectively. 
$\psi(\cdot)$ is a non-negative potential with two equal minima at $\pm1$. 
The positive parameters $A$ and $B$ are constant with the typical choice $A = \frac{\beta}{\epsilon}$, $B = \beta\epsilon$,
where $\beta>0$ denotes the surface tension and $\epsilon>0$ is a small parameter related to the interfacial thickness.

%%%%%%%%%%%%%%

The above system is based on the well-known Ginzburg--Landau energy density
\begin{align}
    f(\phi,\nabla\phi) 
    = A \psi(\phi) 
    + \frac{B}{2} \abs{\nabla\phi}^2,
\end{align}
which relates to interfacial energy and unmixing tendencies, and a nutrient energy density
\begin{align}
    N(\phi,\sigma) 
    = \frac{\chi_\sigma}{2} \abs{\sigma}^2
    + \chi_\phi \sigma(1-\phi),
\end{align}
where the second term describes an interaction between the nutrient and the cells \cite{GarckeLSS_2016}.
Analogously to \cite{garcke_lam_2017}, the following formal energy identity is satisfied:
\begin{align}
\label{eq:energy_identity}
\begin{split}
    &\ddt \int_\Omega \Big( f(\phi,\nabla\phi)
    + N(\phi,\sigma) \Big) \dx
    + \int_\Omega 
    m(\phi) \abs{\nabla\mu}^2 
    + n(\phi) \abs{\nabla N_{,\sigma}(\phi,\sigma)}^2 \dx
    \\
    &\quad
    + \int_\Omega N_{,\sigma}(\phi,\sigma) \Gamma_\sigma(\phi,\mu,\sigma)  
    - \mu \Gamma_\phi(\phi,\mu,\sigma) \dx
    + \int_{\partial\Omega} K N_{,\sigma}(\phi,\sigma) (\sigma - \sigma_\infty) \dH^{d-1}
    = 0,
\end{split}
\end{align}
where $N_{,\sigma}(\phi,\sigma)$ denotes the partial derivative of $N(\phi,\sigma)$ with respect to $\sigma$.

The main obstacles, which the authors of \cite{garcke_lam_2017} had to face in order to derive useful \textit{a priori} estimates based on the energy identity \eqref{eq:energy_identity}, arise from the nutrient energy density, where the term $\sigma (1-\phi)$ may become negative, and the presence of source terms $N_{,\sigma}(\phi,\sigma) \Gamma_\sigma(\phi,\mu,\sigma)  
- \mu \Gamma_\phi(\phi,\mu,\sigma)$, which may contain nonlinearities like, e.g., triple products. 
However, under some appropriate assumptions, the authors of \cite{garcke_lam_2017} established to prove well-posedness of weak solutions of the system \eqref{eq:all}. 

In this work, we use the \textit{a priori} estimates of \cite{garcke_lam_2017} to analyze a discretization of the system \eqref{eq:all}. 
The paper is organized as follows. At first, we discretize the system \eqref{eq:all} with a practical fully-discrete finite element approximation, where all terms except of the nonlinear mobility functions in \eqref{eq:phi} and \eqref{eq:sigma} are treated implicitly. Moreover, we make use of a convex-concave splitting of the double-well potential $\psi=\psi_1 + \psi_2$ with $\psi_1$ convex and $\psi_2$ concave, where $\psi_1'$ is treated implicitly and $\psi_2'$ explicitly.
After that, we derive stability estimates and prove existence of discrete solutions supposed that the time step size satisfies a minor constraint.
Moreover, the discrete solutions depend continuously on the initial and boundary data if the mobility functions are constant and if the time step size is small enough.
In comparison to the classical Cahn--Hilliard equation \cite{barrett_blowey_garcke_2000, elliott_kinetics} where e.g.~mass conservation of the order parameter $\phi$ is used, we have to use different techniques as we have to face the difficulties that arise from the non-positive nutrient energy, the additional source terms and the coupling to the reaction-diffusion equation.
Then, we establish higher order bounds for the discrete solutions before passing with the discretization parameters to zero. %Hence, we prove existence and continuous dependence of weak solutions of the system \eqref{eq:all}.

In particular, we successfully prove that subsequences of the discrete solutions convergence to a weak solution of the system \eqref{eq:all} which is unique under additional assumptions. 
Finally, we present some numerical results including numerical error investigation in one spatial dimension and some long time simulations in two and three spatial dimensions which highlight the practicability of our discrete scheme.

%% file: results/fem_1_setting.tex
\section{Fully discrete finite element approximation}

We split the time interval $[0,T)$ into intervals $[t^{n-1},t^n)$
%with $t^n = n \Delta t$ and $t^{N_T}=T$, where $\Delta t>0$ and $n=0,...,N_T$.
with $\Delta t_n = t^n - t^{n-1}$, $n=1,...,N_T$. For simplicity we assume that $\Delta t_n = \Delta t$ for a $\Delta t>0$ and all $n=1,...,N_T$.
Moreover, we assume that $\Omega \subset\R^d$, $d\in\{1,2,3\}$, is a convex, polygonal domain with boundary $\partial\Omega$.
%$\Gamma\coloneqq \partial\Omega$. 
We require $\{\calT_h\}_{h>0}$ to be a regular family of conform quasiuniform triangulations with mesh parameter $h>0$. We also require that the family of meshes $\{\calT_h\}_{h>0}$ consists only of non-obtuse simplices. 
For a given partitioning of meshes $\calT_h$, we denote the simplices by $K_k$ with $k=1,...,N_K$ and its vertices $\{P_i^k\}_{i=0}^d$. 
%The set of internal edges of triangles ($d=2$) in the mesh $\calT_h$ or facets of tetrahedra ($d=3$) is denoted by $\partial\calT_h = \{E_j\}_{j=1}^{N_E}$. 
The set of all the vertices of $\calT_h$ is denoted by $\{P_p\}_{p=1}^{N_p}$.
For more details on finite elements, we refer to \cite{bartels_2016}.

In this work, we use the standard notation from, e.g., \cite{alt_2016,evans_2010}.
We denote the Euclidean norm by $\abs{\cdot}$. For a Banach space $X$, we denote the dual space by $X'$. 
For $p\in[1,\infty]$ and an integer $m\geq 0$, we write $L^p\coloneqq L^p(\Omega)$, $W^{m,p}\coloneqq W^{m,p}(\Omega)$ and $H^m\coloneqq H^m(\Omega) \coloneqq W^{m,2}(\Omega)$, where $W^{0,p}\coloneqq L^p$ in the case $m=0$. 
The norms and seminorms are denoted by $\norm{\cdot}_{W^{m,p}}$ and $\seminorm{\cdot}_{W^{m,p}}$, respectively, and similarly for the spaces $L^p$ and $H^m$.
We denote the inner product of the spaces $L^2$ and $L^2(\partial\Omega)$ by $\skp{\cdot}{\cdot}_{L^2}$ and $\skp{\cdot}{\cdot}_{L^2(\partial\Omega)}$, respectively.
For $\alpha\in[0,1]$, we write $C^{0,\alpha}(\overline\Omega)$ for the Hölder spaces.
For a Banach space $X$, $p\in[1,\infty]$ and an integer $m\geq 0$, we denote the Bochner spaces by $L^p(0,T;X)$ and $W^{m,p}(0,T;X)$ and they are equipped with the norms $\norm{\cdot}_{L^p(0,T;X)}$ and $\norm{\cdot}_{W^{m,p}(0,T;X)}$. For $p=2$, we will also write $H^m(0,T;X)\coloneqq W^{m,2}(0,T;X)$ and $\norm{\cdot}_{H^m(0,T;X)} \coloneqq \norm{\cdot}_{W^{m,2}(0,T;X)}$. Sometimes, $L^p(0,T;L^p)$ will be identified with $L^p(\Omega_T)$ if $X=L^p$.

%%%%%%%%%%%%%%%%%%%%%%%%%%%%%%

We denote the finite element space of continuous and piecewise linear functions by
\begin{alignat*}{2}
    \calS_h &\coloneqq \{ q_h \in C(\overline\Omega): \ q_h|_{K_k} \in \calP_1, \ k=1,...,N_k \} 
    &&\subset H^1(\Omega).
\end{alignat*}
Moreover, we denote the nodal interpolation operator by $\calI_h: C(\overline{\Omega})\to \calS_h$ such that $(\calI_h \eta)(P_p) = \eta(P_p)$ for all $p=1,...,N_p$. 
As we want to use \textit{mass lumping}, we introduce the following semi-inner products and the induced seminorms on $C(\overline\Omega)$ and $C(\partial\Omega)$, respectively, by
\begin{alignat}{2}
    \skp{\eta_1}{\eta_2}_h &\coloneqq 
    \int_\Omega \calI_h \big[ \eta_1 \eta_2 \big] \dx,
    \quad\quad
    &&\norm{\cdot}_h \coloneqq \sqrt{\skp{\cdot}{\cdot}_h},
    \\
    \skp{\eta_3}{\eta_4}_{h,\partial\Omega} &\coloneqq 
    \int_{\partial\Omega} \calI_h \big[ \eta_3 \eta_4 \big] \dH^{d-1},
    \quad\quad
    &&\norm{\cdot}_{h,\partial\Omega} \coloneqq \sqrt{\skp{\cdot}{\cdot}_{h,\partial\Omega}}.
\end{alignat}

Below, we recall some well-known properties concerning $\calS_h$ and the interpolant $\calI_h$. Let $q_h, \zeta_h \in \calS_h$, $K\in \calT_h$,  $m\in\{0,1\}$ and $1 \leq r \leq p \leq \infty$. Then,
\begin{alignat}{2}
    %%%%%% norm equivalence
    \label{eq:norm_equiv}
    c \norm{q_h}_{L^2}^2 
    \leq \norm{q_h}_h^2 
    &\leq C \norm{q_h}_{L^2}^2,
    \\
    \label{eq:norm_equiv_Gamma}
    c \norm{q_h}^2_{L^2(\partial\Omega)} 
    \leq \norm{q_h}^2_{h,\partial\Omega} 
    &\leq C \norm{q_h}^2_{L^2(\partial\Omega)},
    \\
    %%%%%% inverse estimate
    \label{eq:inverse_estimate}
    \seminorm{q_h}_{W^{m,p}(K)} 
    &\leq 
    C h^{\frac{d}{p} - \frac{d}{r}} \seminorm{q_h}_{W^{m,r}(K)},
    \\
    %%%%%% Interpolation errors
    \label{eq:interp_H2}
    \norm{\eta - \calI_h \eta}_{L^2}
    + h \norm{\nabla(\eta-\calI_h\eta)}_{L^2} 
    &\leq C h^2 \seminorm{\eta}_{H^2}
    \quad\quad 
    &&\forall \eta\in H^2(\Omega),
    \\
    \label{eq:interp_continuous}
    \lim_{h\to 0} \norm{\eta - \calI_h\eta}_{L^\infty}
    &= 0
    \quad\quad
    &&\forall \eta \in C(\overline\Omega),
    \\
    %%%%%% mass lumping errors
    \label{eq:lump_Sh_Sh}
    \abs{ \skp{q_h}{\zeta_h}_h 
    - \skp{q_h}{\zeta_h}_{L^2} } 
    &\leq
    C h^2 \norm{q_h}_{H^1} \norm{\zeta_h}_{H^1},
    \\
    \label{eq:lump_Gamma_Sh_Sh}
    \abs{ \skp{q_h}{\zeta_h}_{h,\partial\Omega} 
    - \skp{q_h}{\zeta_h}_{L^2(\partial\Omega)} } 
    &\leq
    C h \norm{q_h}_{H^1} \norm{\zeta_h}_{H^1},
\end{alignat}
where $c,C>0$ we denote various constants that are independent of $h$. 
%Unless otherwise stated, $C>0$ will always denote a generic constant which is independent of $h,\Delta t$.

%%%%%%%%%%%%%

Furthermore, we recall the Clément operator $\calI_h^{Cl}: L^2(\Omega)\to \calS_h$ which is defined by local averages instead of nodal values, see \cite{clement_1975}. The following properties are taken from \cite[Chap.~3]{ciarlet}:
\begin{subequations}
\begin{alignat}{2}
    \label{eq:clement_error}
    \seminorm{\eta - \calI_h^{Cl} \eta }_{W^{k,2}}
    &\leq C h^{m-k} \seminorm{\eta}_{W^{m,2}}
    \quad
    && \forall \eta \in W^{m,2}(\Omega),\ 0 \leq k \leq m \leq 2,
    \\
    %%%%
    \label{eq:clement_conv}
    \lim\limits_{h\to 0} \norm{\eta-\calI_h^{Cl} \eta}_{W^{k,2}} &= 0
    \quad
    && \forall \eta \in W^{k,2}(\Omega), \ 0 \leq k \leq 1,
\end{alignat}
for a constant $C>0$ that is independent of $h$. Moreover, if only a finite number of patch shapes occur in the sequence of triangulations, then 
\begin{alignat}{2}
    \label{eq:clement_Gamma}
    \norm{\eta - \calI_h^{Cl} \eta }_{L^2(\partial\Omega)}
    &\leq C h^{1/2} \norm{\nabla\eta}_{L^2}
    \quad
    && \forall \eta \in H^1(\Omega),
\end{alignat}
\end{subequations}
see \cite[Thm.~4.2]{bartels_2016}. In practice, this assumption seems to be not that restrictive. Hence, we suppose it to hold.

%%%%%%%%%%%%%%%%%%%%
\subsubsection*{Approximation of the initial and boundary values}
Let the initial values fulfill
\begin{align*}
    &\phi_0 \in H^2(\Omega;[-1,1]) \text{ with } \nabla\phi_0\cdot\pmbn = 0 \text{ on } \partial\Omega,
    \\
    &\sigma_0 \in H^1(\Omega),
\end{align*}
where $\pmbn$ denotes the outer unit normal on $\partial\Omega$.
We approximate the intial data by 
\begin{align}
    \label{eq:def_initial}
    \phi_h^0 \coloneqq \calI_h \phi_0, 
    \quad
    \sigma_h^0 \coloneqq \calI_h^{Cl} \sigma_0.
\end{align}
Hence, it follows from \eqref{eq:interp_H2}, \eqref{eq:clement_error}, \cite[eq. (3.16)]{barrett_nurnberg_styles_2004} and the assumptions on $\phi_0$ and $\sigma_0$, that
\begin{align}
    \label{eq:bounds_initial}
    \int_\Omega \calI_h\big[ \psi(\phi_h^0)\big] \dx 
    + \norm{\phi_h^0}_{H^1}^2
    + \norm{\Delta_h \phi_h^0}_{L^2}^2
    + \norm{\sigma_h^0}_{H^1}^2
    &\leq C, 
\end{align}
where the discrete Neumann-Laplacian $\Delta_h: \calS_h \to \calS_h$ is defined by%\footnote{ Motivation of $\Delta_h$}
\begin{align}
    \label{eq:discr_laplace}
    \int_\Omega \calI_h\big[ \Delta_h q_h \zeta_h\big] \dx
    \coloneqq 
    - \int_\Omega \nabla q_h \cdot \nabla\zeta_h \dx
    \quad\quad \forall \zeta_h\in \calS_h,
\end{align}
where $q_h\in\calS_h$. 
We note for future reference, as $\{\calT_h\}_{h>0}$ is a quasi-uniform family of partitionings and
as the domain $\Omega$ is convex, that for $q_h\in \calS_h$
\begin{align}
    \label{eq:discr_laplace_bound}
    \seminorm{q_h}_{W^{1,s}} 
    \leq C \norm{\Delta_h q_h}_{L^2},
\end{align}
for $s\in[1,\infty]$ if $d=1$, $s\in[1,\infty)$ if $d=2$ and $s\in[1,6]$ if $d=3$, see \cite[Lemma 3.1]{barrett_langdon_nuernberg_2004}.
%for all $s\in[1,\infty]$ if $d=1$, and $s\in[1,\frac{2d}{d-2})$ if $d\in\{2,3\}$, see \cite[Lemma 3.1]{barrett_langdon_nuernberg_2004}.

Furthermore, let $\sigma_\infty \in L^2(0,T;H^1(\Omega))$. Then, for all $t\in[t^{n-1}, t^n)$ and $n=1,...,N_T$, we define  
the following approximation which is piecewise constant in time by
\begin{align}
    \label{eq:def_bc}
    \sigma_{\infty,h}^{\Delta t, +} (t,\cdot)
    = \sigma_{\infty,h}^n(\cdot) 
    \coloneqq 
    \frac{1}{\Delta t} \int_{t^{n-1}}^{t^n} \calI_h^{Cl} \big[ \sigma_{\infty}(t,\cdot) \big] \dt \ \in \calS_h,
\end{align}
which fulfills on noting \eqref{eq:clement_Gamma}, that
\begin{align}
    \label{eq:bounds_bc}
    \norm{\sigma_{\infty,h}^{\Delta t, +}}^2_{L^2(0,T;L^2(\partial\Omega))}
    =
    \Delta t \sum_{n=1}^{N_T} \norm{\sigma_{\infty,h}^n}^2_{L^2(\partial\Omega)}
    \leq C \norm{\sigma_{\infty}}^2_{L^2(0,T;H^1)}.
\end{align}

%We have from \eqref{eq:interp_H2}, \eqref{eq:clement_conv}, \eqref{eq:clement_Gamma} that
%\begin{align}
%    \lim_{h\to 0_+} \Big( 
%    \norm{\phi_h^0 - \phi_0}_{H^1}
%    + \norm{\Delta_h \phi_h^0 - \Delta \phi_0}_{L^2}
%    + \norm{\sigma_h^0 - \sigma_0}_{L^2} 
%    \Big) 
%    &= 0,
%    \\
%    \lim_{(h,\Delta t)\to 0_+} \norm{\sigma_{h,\infty}^{\Delta t,+} - \sigma_\infty }_{L^2(0,T;L^2(\partial\Omega))} 
%    &= 0.
%\end{align}

%%%%%
\subsubsection*{Assumptions on the model functions and parameters}
We make the following assumptions on the model parameters and functions.
%\mbox{}\vspace{-\topskip} % für Zeilenumbruch vor Itemize
\begin{itemize}
	    \item[$(A1)$] Let $\chi_\phi \geq0$ and $\chi_\sigma, A, B, K >0$ be constant.
	    
	    \item[$(A2)$] The functions $\Gamma_\phi, \Gamma_\sigma: \R^2\to\R$ only depend on $(\phi,\sigma)$ and they are continuous with linear growth, i.e.
        \begin{align}
            \abs{\Gamma_i(\phi,\sigma)} \leq R_0 (1+\abs{\phi}+\abs{\sigma}), \quad i\in\{\phi,\sigma\},
        \end{align}
        with a constant $R_0>0$. 
        
	    \item[$(A3)$]
        It holds $m, n \in C^0(\R)$ and there exist constants $m_0, m_1, n_0, n_1 >0$ such that for all $s\in\R$:
	    \begin{align*}
	        m_0 &\leq m(s) \leq m_1, \quad n_0 \leq n(s) \leq n_1.
	    \end{align*}
	    
	    \item[$(A4)$] The potential $\psi$ is nonnegative and belongs to $\psi\in C^{1,1}(\R)$ with
	    \begin{align}
	        \label{A4_1}
	        \psi(t) &\geq R_1 \abs{t}^2 - R_2, 
	   \end{align}
	   where $R_1,R_2>0$. Additionally, the potential can be decomposed as $\psi=\psi_1 + \psi_2$ with $\psi_1\in C^{1,1}(\R)$ convex and $\psi_2\in C^{1,1}(\R)$ concave such that
	   \begin{align} 
	        \label{A4_2}
	        \abs{\psi_i^\prime(t)} &\leq R_3 (1+\abs t),
	    \end{align}
	    where $i=1,2$ and $R_3 >0$. Moreover, we assume that
	    \begin{align}
	        \label{A4_3}
	        A > \frac{4\chi_\phi^2}{\chi_\sigma R_1}.
	    \end{align}

\end{itemize}
As $A=\frac{\beta}{\epsilon}$ and $\epsilon$ is small in applications, \eqref{A4_3} is not a severe constraint. %\footnote{The assumptions (A2)--(A4) are essential for the mathematical analysis, but in practice, we make other choices, see numerical examples in Section 7.}

%%%%%%%%%%%%%%%%%%%%
\subsubsection*{Fully discrete system}

Let us now introduce the numerical scheme approximating the system \eqref{eq:all}.

Let the discrete initial data $(\phi_h^{0},\sigma_{h}^{0}) \in (\calS_h)^2$ and, for $n=1,...,N_T$, let the discrete boundary values $\sigma_{\infty,h}^n\in \calS_h$ be given by \eqref{eq:def_initial} and \eqref{eq:def_bc}, respectively.
Then, for $n=1,...,N_T$, find the discrete solution triplet $(\phi_h^{n}, \mu_h^n, \sigma_{h}^{n}) \in (\calS_h)^3 $ which satisfies for any test function triplet $(\zeta_h, \rho_h, \xi_h) \in (\calS_h)^3 $:
\begin{subequations}
\begin{align}
    \label{eq:phi_FE}
    \int_\Omega \calI_h \Big[ \Big(\frac{\phi_h^n-\phi_h^{n-1}}{\Delta t}
    - \Gamma_{\phi,h}^n \Big) \zeta_h \Big]
    + \calI_h[m(\phi_h^{n-1})] \nabla\mu_h^n \cdot \nabla \zeta_h \dx
    = 0,
    \\
    \label{eq:mu_FE}
    \int_\Omega \calI_h \Big[ \Big( \mu_h^n  
    - A \psi_1'(\phi_h^n) - A \psi_2'(\phi_h^{n-1}) 
    + \chi_\phi \sigma_h^n \Big) \rho_h \Big] 
    - B \nabla\phi_h^n \cdot \nabla\rho_h \dx 
    = 0,
    \\
    \nonumber
    \label{eq:sigma_FE}
    \int_\Omega \calI_h \Big[ \Big(\frac{\sigma_h^n-\sigma_h^{n-1}}{\Delta t}
    + \Gamma_{\sigma,h}^n \Big) \xi_h \Big]
    + \calI_h[n(\phi_h^{n-1})] 
    \big(\chi_\sigma \nabla\sigma_h^n - \chi_\phi \nabla\phi_h^n \big)  \cdot \nabla \xi_h  \dx
    \quad\quad
    \\
    + \int_{\partial\Omega} \calI_h\Big[ K \big(\sigma_h^n - \sigma_{\infty,h}^n \big) \xi_h \Big] \dH^{d-1}
    = 0,
\end{align}
\end{subequations}
where $\Gamma_{\phi,h}^n \coloneqq \Gamma_\phi(\phi_h^n,\sigma_h^n)$ and $\Gamma_{\sigma,h}^n \coloneqq \Gamma_\sigma(\phi_h^n,\sigma_h^n)$.

In the scheme \eqref{eq:phi_FE}--\eqref{eq:sigma_FE}, we make use of numerical integration by \textit{mass lumping} which is often used for phase-field models because of computational reasons. 
\red The main advantage is that \textit{mass lumping} leads to simpler systems of equations as the \textit{mass matrices} are diagonal whereas the precision of the numerical solutions is not affected. 
In particular, the numerical errors resulting from \textit{mass lumping} are based on the interpolation error estimate \eqref{eq:interp_H2}, and here we refer to, e.g., \cite{blowey_elliott_1992, elliott_second_order}, where convergence rates for other phase-field systems have been studied. 
%Later, in Section \ref{sec:numeric}, we compute the experimental orders of convergence for solutions of the system \eqref{eq:phi_FE}--\eqref{eq:sigma_FE} for $d=1$. 
\blk
Let us briefly remark that other quadrature rules can also be used as long as the quadrature weights are non-negative.
% Otherwise the sign in some inequalities changes: Let $\int_T f(x) \dx \approx \sum_{i=0}^{i_{max}} w_i f(a_i)$. Then $w_i$ have to be non-negative, otherwise the inequality 
% \begin{align*}
%     w_i \Big(\psi_1'(\phi_h^n(a_i)) + \psi_2'(\phi_h^{n-1}(a_i)) \Big) (\phi_h^n(a_i) - \phi_h^{n-1}(a_i)) \geq \psi(\phi_h^n(a_i)) - \psi(\phi_h^{n-1}(a_i))
% \end{align*}
% does not hold...

Moreover, the time discretization of the scheme is chosen such that the nonlinear mobility functions $m(\cdot), n(\cdot)$ and the derivative of the concave part of the potential $\psi_2'(\cdot)$ are treated explicitly and all the other terms are treated implicitly. 
However, also other time discretizations of the mobility functions $m(\cdot), n(\cdot)$ and of the source terms $\Gamma_\phi(\cdot,\cdot), \Gamma_\sigma(\cdot,\cdot)$ are possible.
In \eqref{eq:mu_FE}, the terms $\psi_1'(\cdot), \psi_2'(\cdot)$ are discretized in time by a convex-concave splitting method which is often used in the context of phase-field systems, see e.g.~\cite{barrett_blowey_garcke_2000, gruen_2013}.
In particular, the convex-concave splitting allows the inequality \eqref{eq:convex_concave} which is essential for the analysis of the scheme.

\medskip

In the following two sections, we will discuss stability, existence and continuous dependence of solutions of the scheme \eqref{eq:phi_FE}-\eqref{eq:sigma_FE}. 

%% file: results/fem_2_stability.tex
\section{Stability of the discrete system}
Let us introduce the discrete free energy $\calF_h: \calS_h\times \calS_h \to \R$ of the system \eqref{eq:phi_FE}--\eqref{eq:sigma_FE} by 
\begin{align}
\label{eq:def_energy_FE}
\begin{split}
    \calF_h(\phi_h,\sigma_h) 
    &= 
    \int_\Omega 
    \frac{B}{2} \abs{\nabla\phi_h}^2 
    + \calI_h\Big[ A \psi(\phi_h) 
    + \frac{\chi_\sigma}{2} \abs{\sigma_h}^2
    + \chi_\phi \sigma_h (1 - \phi_h) \Big] \dx,
\end{split}
\end{align}
for all $\phi_h,\sigma_h \in \calS_h$. We note that the last term in \eqref{eq:def_energy_FE} can have a negative sign. This is one of the main obstacles we have to handle to derive useful \textit{a priori} estimates.

%%%%%%%%%%%%%%%%%%%%%%%%%%%%%%%%%%%%%%%%%%%%%%%%%%%

%We recall the following discrete Gronwall inequality from \cite{dahmen_reusken_numerik}. %\cite[Proof of Example 11.26]{dahmen_reusken_numerik}
We recall the following discrete version of Gronwall's inequality. For the proof, we refer to, e.g., \cite[pp.~401--402]{dahmen_reusken_numerik}.

\begin{samepage}
\begin{lemma}
\label{lemma:gronwall_discrete}
Assume that $e_n, a_n, b_n \geq 0$ for all $n\geq 0$. Then 
\begin{align}
\label{eq:gronwall_discrete}
\begin{split}
    e_n &\leq a_n + \sum\limits_{i=0}^{n-1}  b_i e_i 
    \quad \forall n\geq 0
    \quad \Longrightarrow\quad
    e_n \leq a_n \cdot \exp\Big( \sum\limits_{i=0}^{n-1} b_i \Big)
    \quad \forall n\geq 0.
    % \\
    % \Longrightarrow\quad
    % e_n  &\leq a_n \cdot \exp\Big( \sum\limits_{i=0}^{n-1} b_i \Big)
    % \quad\quad \forall n\geq 0.
\end{split}
\end{align}
\end{lemma}
\end{samepage}

%%%%%%%%%%%%%%%%%%%%%%%%%%%%%%%%%%%%%%%%%%%%%%%%%%
With the help of Lemma \ref{lemma:gronwall_discrete}, we can now derive stability estimates for the numerical scheme \eqref{eq:phi_FE}--\eqref{eq:sigma_FE}.

\begin{lemma}[Stability]
\label{lemma:energy_FE}
Assume that $\Delta t < \Delta t_*$, where $\Delta t_*$ is a constant that only depends on the model parameters. For the explicit form of $\Delta t_*$, see \eqref{eq:energy_FE_dt}. Then, for $n=1,...,N_T$, solutions $\big( \phi_h^n,\mu_h^n,\sigma_h^n \big)\in (\calS_h)^3$ of \eqref{eq:phi_FE}--\eqref{eq:sigma_FE}, if they exist, satisfy
\begin{align}
\begin{split}
\label{eq:energy_FE}
    %%% Zeitableitungen
    &  \max_{m=1,...,N_T} \Big( \norm{\phi_h^m}_h^2
    +  \norm{\nabla\phi_h^m}_{L^2}^2
    +  \norm{\sigma_h^m}_h^2 \Big)
    +  \sum_{n=1}^{N_T} \norm{\nabla\phi_h^n - \nabla\phi_h^{n-1}}_{L^2}^2 
    +  \sum_{n=1}^{N_T} \norm{\sigma_h^n - \sigma_h^{n-1}}_h^2
    \\
    &\quad
    + \Delta t \sum_{n=1}^{N_T} \Big( \norm{\mu_h^n}_h^2 
    + \norm{\nabla\mu_h^n}_{L^2}^2 
    + \norm{\nabla \sigma_h^n}_{L^2}^2
    + \norm{\sigma_h^n}_{h,{\partial\Omega }}^2 \Big)
   \\
    %%% rechte Seite
    &\leq 
    C \Big( T + \abs{\calF_h(\phi_h^0, \sigma_h^0)}
    + \Delta t \sum_{n=1}^{N_T} \norm{\sigma_{\infty,h}^n}_{h,{\partial\Omega }}^2 \Big) \cdot \exp( CT )
    \leq C.
\end{split}
\end{align}
\end{lemma}

\input{results/proof_energy}

\begin{remark}
We remark that the assumptions \eqref{A4_3} and \eqref{eq:energy_FE_dt} are by no means optimal assumptions. 
The assumption \eqref{A4_3} is necessary because the term
\begin{align*}
    \int_\Omega \calI_h\big[\sigma_h^m (1-\phi_h^m) \big] \dx
\end{align*}
in the discrete energy $\calF_h(\phi_h^m,\sigma_h^m)$ on the left-hand side of \eqref{eq:energy_FE_9} can be negative. Hence, we perform the steps \eqref{eq:energy_FE_10}--\eqref{eq:energy_FE_12} to absorb it. For more details, also see \cite[Remark 3.1]{garcke_lam_2017}. The condition \eqref{eq:energy_FE_dt} appears when applying a discrete Gronwall argument in \eqref{eq:energy_FE_13}.
\end{remark}

%% file: results/proof_energy.tex
\begin{proof} 
We now start the testing procedure.
In equation \eqref{eq:phi_FE}, we set $\zeta_h = \mu_h^n$ and use the lower bound of $m(\cdot)$ to obtain
\begin{align*}
    % Gleichung für phi
    \int_\Omega \calI_h \Big[ \Big(\frac{\phi_h^n-\phi_h^{n-1}}{\Delta t}
    - \Gamma_{\phi,h}^n \Big) \mu_h^n \Big]
    + m_0 \aabs{\nabla\mu_h^n}^2 \dx
    \leq 0.
\end{align*}

Testing \eqref{eq:mu_FE} with $\rho_h = \frac{1}{\Delta t}(\phi_h^n-\phi_h^{n-1})$ gives
\begin{align*}
    % Gleichung für mu
    \int_\Omega \calI_h \Big[ \frac{\phi_h^n-\phi_h^{n-1}}{\Delta t} \Big( - \mu_h^n  
    + A \psi_1'(\phi_h^n) + A \psi_2'(\phi_h^{n-1}) 
    - \chi_\phi \sigma_h^n \Big) \Big] 
    + \frac{B}{\Delta t} \nabla\phi_h^n \cdot \big( \nabla\phi_h^n - \nabla\phi_h^{n-1}\big) \dx
    = 0.
\end{align*}
As the potential can be decomposed into a convex and a concave part, i.e.~$\psi=\psi_1 + \psi_2$, we get the inequality
\begin{align}
\label{eq:convex_concave}
    \Big(\psi_1'(\phi_h^n) + \psi_2'(\phi_h^{n-1}) \Big) (\phi_h^n- \phi_h^{n-1}) \geq \psi(\phi_h^n) - \psi(\phi_h^{n-1}).
\end{align}
Using the elementary identity 
\begin{align}
    \label{eq:elementary_identity}
    2x(x-y) = x^2-y^2 + (x-y)^2 \quad\quad \forall x,y\in\R,
\end{align}
we obtain that
\begin{align*}
    % Gleichung für mu
    \int_\Omega \calI_h \Big[
    A \frac{\psi(\phi_h^n)-\psi(\phi_h^{n-1})}{\Delta t}
    - \frac{\phi_h^n-\phi_h^{n-1}}{\Delta t} (\mu_h^n + \chi_\phi \sigma_h^n) \Big] \dx
    \quad\quad
    \\
    +\int_\Omega \frac{B}{2} \frac{\abs{\nabla\phi_h^n}^2 - \abs{\nabla\phi_h^{n-1}}^2}{\Delta t} 
    + \frac{B}{2} \frac{\abs{\nabla\phi_h^n - \nabla\phi_h^{n-1}}^2}{\Delta t}  \dx
    \leq 0.
\end{align*}

Next, we test \eqref{eq:sigma_FE} with $\xi_h = \chi_\sigma \sigma_h^n + \chi_\phi (1-\phi_h^n)$ and use the lower bound of $n(\cdot)$. Then it holds on noting \eqref{eq:elementary_identity} that
\begin{align*}
    % Gleichung für sigma
    &\int_\Omega \calI_h \Big[ 
    \frac{\chi_\sigma}{2} \frac{\abs{\sigma_h^n}^2 - \abs{\sigma_h^{n-1}}^2}{\Delta t}
    + \frac{\chi_\sigma}{2} \frac{\abs{\sigma_h^n - \sigma_h^{n-1}}^2}{\Delta t}
    \Big] 
    + n_0  \aabs{\chi_\sigma \nabla \sigma_h^n - \chi_\phi \nabla\phi_h^n}^2  \dx
    \\
    &\quad 
    + \int_\Omega \calI_h \Big[ \Gamma_{\sigma,h}^n \big(\chi_\sigma \sigma_h^n + \chi_\phi (1-\phi_h^n)\big) 
    + \chi_\phi (1-\phi_h^n) \Big(\frac{\sigma_h^n - \sigma_h^{n-1}}{\Delta t}\Big)   \Big]\dx
    \\
    &\quad 
    + \int_{\partial\Omega } \calI_h\Big[ K \chi_\sigma \abs{\sigma_h^n}^2 
    + K \chi_\phi \sigma_h^n (1-\phi_h^n)
    - K \sigma_{\infty,h}^n \big(\chi_\sigma \sigma_h^n + \chi_\phi (1-\phi_h^n)\big)  \Big] \dH^{d-1}
    \\
    &\leq 0.
\end{align*}

%%%%%%%%%%%%%%%
So far we have that
\begin{align}
\begin{split}
\label{eq:energy_FE_1}
    %%% Zeitableitungen
    &  \frac{B}{2\Delta t} \Big( \norm{\nabla\phi_h^n}_{L^2}^2 - \norm{\nabla\phi_h^{n-1}}_{L^2}^2  
    +  \norm{\nabla\phi_h^n - \nabla\phi_h^{n-1}}_{L^2}^2 \Big) 
    + \int_\Omega \frac{A}{\Delta t} \calI_h\Big[  \psi(\phi_h^n)-\psi(\phi_h^{n-1}) \Big]\dx
    \\
    &\quad 
    + \frac{\chi_\sigma}{2\Delta t}  \Big( 
    \norm{\sigma_h^n}_h^2 - \norm{\sigma_h^{n-1}}_h^2
    + \norm{\sigma_h^n - \sigma_h^{n-1}}_h^2
    \Big)
    \\
    %%% nichtnegative Terme
    &\quad 
    + 
    m_0 \norm{\nabla\mu_h^n}_{L^2}^2 
    +  n_0  \nnorm{\chi_\sigma \nabla \sigma_h^n - \chi_\phi \nabla\phi_h^n}_{L^2}^2 
    + K \chi_\sigma \norm{\sigma_h^n}_{h,{\partial\Omega }}^2
    \\
    %%% andere Terme
    &\quad
    + \int_\Omega \calI_h \Big[ 
    \Gamma_{\sigma,h}^n \big(\chi_\sigma \sigma_h^n + \chi_\phi (1-\phi_h^n)\big) 
    - \mu_h^n \Gamma_{\phi,h}^n  
    \Big] \dx
    \\
    &\quad 
    + \int_\Omega \calI_h \Big[ 
    \chi_\phi (1-\phi_h^n) \Big(\frac{\sigma_h^n - \sigma_h^{n-1}}{\Delta t}\Big)  
    - \frac{\phi_h^n-\phi_h^{n-1}}{\Delta t} \chi_\phi \sigma_h^n \Big] \dx
    \\
    &\quad 
    + \int_{\partial\Omega } \calI_h\Big[
    K \chi_\phi \sigma_h^n (1-\phi_h^n)
    - K \sigma_{\infty,h}^n \big(\chi_\sigma \sigma_h^n + \chi_\phi (1-\phi_h^n)\big)  \Big] \dH^{d-1}
    \\
    &\leq 0.
\end{split}
\end{align}

%%%%%%%%%%%%%%%%%%%%%%%
Next, we derive an estimate for the chemical potential $\mu_h^n$.
On noting \eqref{A4_2} and Young's inequality, we receive by testing \eqref{eq:mu_FE} with $\rho_h = \mu_h^n$ that
\begin{align*}
    &\int_\Omega \calI_h \big[ \abs{\mu_h^n}^2 \big] \dx 
    =
    \int_\Omega \calI_h \Big[
    \Big(A\psi_1'(\phi_h^n) + A\psi_2'(\phi_h^{n-1}) 
    - \chi_\phi \sigma_h^n\Big) \mu_h^n \Big] 
    + B \nabla\phi_h^n \cdot \nabla\mu_h^n \dx
    \\
    &\leq \int_\Omega \calI_h \Big[
    A R_3 (2 + \abs{\phi_h^{n-1}} 
    + \abs{\phi_h^n} \big) \abs{\mu_h^n} 
    + \chi_\phi \abs{\sigma_h^n} \abs{\mu_h^n} \Big]
    + B \abs{\nabla\phi_h^n} \abs{\nabla\mu_h^n} \dx
    \\
    &\leq \int_\Omega \calI_h \Big[
    \frac{1}{2} \abs{\mu_h^n}^2 
    + 2 A^2 R_3^2 \abs{\phi_h^n}^2  
    + 2 A^2 R_3^2 \abs{\phi_h^{n-1}}^2 
    + 2 \chi_\phi^2 \abs{\sigma_h^n}^2 \Big] \dx
    \\
    &\quad + \int_\Omega \frac{B^2}{m_0} \abs{\nabla\phi_h^n}^2 
    + \frac{m_0}{4} \abs{\nabla\mu_h^n}^2 \dx 
    + C(A, R_3, \Omega),
\end{align*}
which yields
\begin{align}
\begin{split}
    \label{eq:energy_FE_2}
    \norm{\mu_h^n}_h^2 
    &\leq 
    4 A^2 R_3^2 \Big(\norm{\phi_h^n}_h^2 
    +  \norm{\phi_h^{n-1}}_h^2 \Big)
    + 4 \chi_\phi^2 \norm{\sigma_h^n}_h^2
    + \frac{2B^2}{m_0} \norm{\nabla\phi_h^n}_{L^2}^2 
    \\
    &\quad
    + \frac{m_0}{2} \norm{\nabla\mu_h^n}_{L^2}^2 \dx 
    + C(A,R_3,\Omega).
\end{split}
\end{align}

With Hölder's and Young's inequalities and \eqref{eq:energy_FE_2}, we can estimate the source terms in \eqref{eq:energy_FE_1} as follows:
\begin{align}
\begin{split}
\label{eq:energy_FE_3}
    &\int_\Omega \calI_h \Big[ 
    \Gamma_{\sigma,h}^n \big(\chi_\sigma \sigma_h^n + \chi_\phi (1-\phi_h^n)\big) 
    - \mu_h^n \Gamma_{\phi,h}^n  
    \Big] \dx
    \\
    &\leq
    \frac{1}{2} \norm{\mu_h}_h^2 
    + \frac{1}{2} \norm{\Gamma_{\phi,h}^n}_h^2
    + \frac{1}{2} \norm{\Gamma_{\sigma,h}^n}_h^2
    + \frac{3\chi_\sigma^2}{2} \norm{\sigma_h^n}_h^2
    + \frac{3\chi_\phi^2}{2} \norm{\phi_h^n}_h^2
    + \frac{3\chi_\phi^2}{2} \abs{\Omega}
    \\
    &\leq 
    \frac{1}{2} \norm{\mu_h}_h^2 
    + R_0^2 \nnorm{(1 + \abs{\phi_h^n} + \abs{\sigma_h^n}) }_h^2
    + \frac{3\chi_\sigma^2}{2} \norm{\sigma_h^n}_h^2
    + \frac{3\chi_\phi^2}{2} \norm{\phi_h^n}_h^2
    + \frac{3\chi_\phi^2}{2} \abs{\Omega}
    \\
    &\leq 
    \frac{1}{2} \norm{\mu_h}_h^2 
    + \big(3R_0^2+ \frac{3\chi_\sigma^2}{2}\big) \norm{\sigma_h^n}_h^2 
    + \big(3R_0^2+ \frac{3\chi_\phi^2}{2}\big)  \norm{\phi_h^n}_h^2
    + C(R_0, \chi_\phi, \Omega).
\end{split}
\end{align}

%%%%%%%%%%%%%%%
For the terms in \eqref{eq:energy_FE_1} involving the boundary integrals, we have by Hölder's and Young's inequalities, \eqref{eq:norm_equiv}, \eqref{eq:norm_equiv_Gamma} and the trace theorem, that
\begin{align}
\begin{split}
\label{eq:energy_FE_4}
    &\int_{\partial\Omega } \calI_h\Big[
    K \chi_\phi \sigma_h^n (1-\phi_h^n)
    - K \sigma_{\infty,h}^n \big(\chi_\sigma \sigma_h^n + \chi_\phi (1-\phi_h^n)\big)  \Big] \dH^{d-1}
    \\
    &\leq \int_{\partial\Omega } \calI_h \Big[ 
    \frac{3 K\chi_\sigma}{4}  \abs{\sigma_h^n}^2  
    + K\big(\frac{\chi_\phi^2}{2\chi_\sigma} + 1 \big) \abs{\phi_h^n}^2
    + C(K, \chi_\phi, \chi_\sigma) \big(1 
    + \abs{\sigma_{\infty,h}^n}^2 \big) 
    \Big] \dH^{d-1}
    \\
    &\leq 
    K C_{tr}^2 \big(\frac{\chi_\phi^2}{2\chi_\sigma} + 1 \big)  
    \Big( \norm{\phi_h^n}_h^2
    + \norm{\nabla\phi_h^n}_{L^2}^2 \Big)
    + \frac{3 \chi_\sigma K}{4}  \norm{\sigma_h^n}_{h,{\partial\Omega }}^2  
    + C(K, \chi_\phi, \chi_\sigma) \big( \abs{{\partial\Omega }}
    + \norm{\sigma_{\infty,h}^n}_{h,{\partial\Omega }}^2 \big) .
\end{split}
\end{align}

Furthermore, we can calculate
\begin{subequations}
\label{eq:energy_FE_5}
\begin{align}
\begin{split}
    & \int_\Omega \calI_h \Big[ \chi_\phi (1-\phi_h^n) \big(\sigma_h^n - \sigma_h^{n-1}\big)  
    -  \chi_\phi \big(\phi_h^n-\phi_h^{n-1}\big) \sigma_h^n \Big] \dx
    \\
    &= 
    \int_\Omega \calI_h \Big[ \chi_\phi  \sigma_h^n (1-\phi_h^n) 
    - \chi_\phi \sigma_h^{n-1} (1-\phi_h^{n-1})
    -  \chi_\phi (\phi_h^n-\phi_h^{n-1}) (\sigma_h^n- \sigma_h^{n-1}) \Big] \dx,
\end{split}
\end{align}
and
\begin{align}
\begin{split}
    & \aaabs{\int_\Omega \calI_h \Big[ \chi_\phi (\phi_h^n-\phi_h^{n-1}) (\sigma_h^n- \sigma_h^{n-1}) \Big] \dx }
    \leq \frac{\chi_\phi^2}{\chi_\sigma} \norm{\phi_h^n - \phi_h^{n-1}}_h^2
    + \frac{\chi_\sigma}{4} \norm{\sigma_h^n - \sigma_h^{n-1}}_h^2
    \\
    &\leq \frac{2\chi_\phi^2}{\chi_\sigma} \Big( \norm{\phi_h^n}_h^2 + \norm{\phi_h^{n-1}}_h^2 \Big)
    + \frac{\chi_\sigma}{4} \norm{\sigma_h^n - \sigma_h^{n-1}}_h^2.
\end{split}
\end{align}
\end{subequations}

%%%%%%%%%%%%%%%%%%%%
Combining \eqref{eq:energy_FE_1}--\eqref{eq:energy_FE_5}, we obtain on noting \eqref{eq:def_energy_FE} that
\begin{align}
\begin{split}
\label{eq:energy_FE_6}
    %%% Zeitableitungen
    & \frac{1}{\Delta t} \Big( 
    \calF_h(\phi_h^n, \sigma_h^n) - \calF_h(\phi_h^{n-1}, \sigma_h^{n-1}) \Big)
    + \frac{B}{2\Delta t} \norm{\nabla\phi_h^n - \nabla\phi_h^{n-1}}_{L^2}^2 
    + \frac{\chi_\sigma}{4\Delta t} 
    \norm{\sigma_h^n - \sigma_h^{n-1}}_h^2
    \\
    %%% nichtnegative Terme
    &\quad 
    + \frac{1}{2} \norm{\mu_h^n}_h^2 
    + \frac{m_0}{2} \norm{\nabla\mu_h^n}_{L^2}^2 
    +  n_0  \nnorm{\chi_\sigma \nabla \sigma_h^n - \chi_\phi \nabla\phi_h^n}_{L^2}^2 \dx
    + \frac{K \chi_\sigma}{4} \norm{\sigma_h^n}_{h,{\partial\Omega }}^2
    \\
    %%% rechte Seite
    &\leq 
    C(A, R_0, R_3,K, \chi_\phi, \chi_\sigma, \Omega, {\partial\Omega }) \big( 1
    + \norm{\sigma_{\infty,h}^n}_{h,{\partial\Omega }}^2 \big)
    +  \frac{2\chi_\phi^2}{\chi_\sigma \Delta t} \Big( \norm{\phi_h^n}_h^2 + \norm{\phi_h^{n-1}}_h^2 \Big)
    \\
    %%% sigma und nabla phi
    &\quad 
    + \Big(3R_0^2+ \frac{3\chi_\sigma^2}{2} + 4 \chi_\phi^2 \Big) \norm{\sigma_h^n}_h^2 
    + \Big( \frac{2B^2}{m_0} + K C_{tr}^2 \big(\frac{\chi_\phi^2}{2\chi_\sigma} + 1 \big) \Big)
    \norm{\nabla\phi_h^n}_{L^2}^2 
    \\
    %%% phi
    &\quad
    + 4 A^2 R_3^2 \norm{\phi_h^{n-1}}_h^2
    + \Big( 4 A^2 R_3^2
    + 3R_0^2+ \frac{3\chi_\phi^2}{2}
    + K C_{tr}^2 \big(\frac{\chi_\phi^2}{2\chi_\sigma} + 1 \big) \Big)
    \norm{\phi_h^n}_h^2.
\end{split}
\end{align}

%%%%%%%%%%%%%%%%
Next, applying the triangle inequality and Young's inequality, we obtain
\begin{align*}
    \chi_\sigma^2 \norm{\nabla \sigma_h^n}_{L^2}^2
    &\leq
    \Big( \nnorm{\chi_\sigma \nabla \sigma_h^n - \chi_\phi \nabla\phi_h^n}_{L^2} 
    + \chi_\phi \norm{\nabla \phi_h^n}_{L^2} \Big)^2
    \\
    &\leq 2 \nnorm{\chi_\sigma \nabla \sigma_h^n - \chi_\phi \nabla\phi_h^n}_{L^2}^2
    + 2 \chi_\phi^2 \norm{\nabla \phi_h^n}_{L^2}^2,
\end{align*}
so that \eqref{eq:energy_FE_6} becomes
\begin{align}
\begin{split}
\label{eq:energy_FE_7}
    %%% Zeitableitungen
    & \frac{1}{\Delta t} \Big( 
    \calF_h(\phi_h^n, \sigma_h^n) - \calF_h(\phi_h^{n-1}, \sigma_h^{n-1}) \Big)
    + \frac{B}{2\Delta t} \norm{\nabla\phi_h^n - \nabla\phi_h^{n-1}}_{L^2}^2 
    + \frac{\chi_\sigma}{4\Delta t} 
    \norm{\sigma_h^n - \sigma_h^{n-1}}_h^2
    \\
    %%% nichtnegative Terme
    &\quad 
    + \frac{1}{2} \norm{\mu_h^n}_h^2 
    + \frac{m_0}{2} \norm{\nabla\mu_h^n}_{L^2}^2 
    +  \frac{n_0 \chi_\sigma^2}{2}  \nnorm{\nabla \sigma_h^n}_{L^2}^2
    + \frac{K \chi_\sigma}{4} \norm{\sigma_h^n}_{h,{\partial\Omega }}^2
    \\
    %%% rechte Seite
    &\leq 
    C(A, R_0, R_3,K, \chi_\phi, \chi_\sigma, \Omega, {\partial\Omega }) \big( 1
    + \norm{\sigma_{\infty,h}^n}_{h,{\partial\Omega }}^2 \big)
    +  \frac{2\chi_\phi^2}{\chi_\sigma \Delta t} \Big( \norm{\phi_h^n}_h^2 + \norm{\phi_h^{n-1}}_h^2 \Big)
    \\
    %%% sigma und nabla phi
    &\quad 
    + \Big(3R_0^2+ \frac{3\chi_\sigma^2}{2} + 4 \chi_\phi^2 \Big) \norm{\sigma_h^n}_h^2 
    + \Big( \frac{2B^2}{m_0} + K C_{tr}^2 \big(\frac{\chi_\phi^2}{2\chi_\sigma} + 1 \big) 
    + n_0 \chi_\phi^2 \Big)
    \norm{\nabla\phi_h^n}_{L^2}^2 
    \\
    %%% phi
    &\quad
    + 4 A^2 R_3^2 \norm{\phi_h^{n-1}}_h^2
    + \Big( 4 A^2 R_3^2
    + 3R_0^2+ \frac{3\chi_\phi^2}{2}
    + K C_{tr}^2 \big(\frac{\chi_\phi^2}{2\chi_\sigma} + 1 \big) \Big)
    \norm{\phi_h^n}_h^2.
\end{split}
\end{align}

%%%%%%%%%%%%%%
Now we define the constants
\begin{align}
\begin{split}
\label{eq:energy_FE_const}
    &c_1 \coloneqq \frac{m_0}{2}, \quad
    c_2 \coloneqq \frac{n_0 \chi_\sigma^2}{2}, \quad
    c_3 \coloneqq \frac{K \chi_\sigma}{4}, \quad
    c_4 \coloneqq 3R_0^2+ \frac{3\chi_\sigma^2}{2} + 4 \chi_\phi^2, 
    \\
    &c_5 \coloneqq \frac{2B^2}{m_0} + K C_{tr}^2 \big(\frac{\chi_\phi^2}{2\chi_\sigma} + 1 \big) 
    + n_0 \chi_\phi^2, \quad
    c_6 \coloneqq 4 A^2 R_3^2, 
    \\
    & c_7 \coloneqq 4 A^2 R_3^2
    + 3R_0^2+ \frac{3\chi_\phi^2}{2}
    + K C_{tr}^2 \big(\frac{\chi_\phi^2}{2\chi_\sigma} + 1 \big).
\end{split}
\end{align}
Then, \eqref{eq:energy_FE_7} becomes
\begin{align}
\begin{split}
\label{eq:energy_FE_8}
    %%% Zeitableitungen
    & \frac{1}{\Delta t} \Big( 
    \calF_h(\phi_h^n, \sigma_h^n) - \calF_h(\phi_h^{n-1}, \sigma_h^{n-1}) \Big)
    + \frac{B}{2\Delta t} \norm{\nabla\phi_h^n - \nabla\phi_h^{n-1}}_{L^2}^2 
    + \frac{\chi_\sigma}{4\Delta t} 
    \norm{\sigma_h^n - \sigma_h^{n-1}}_h^2
    \\
    %%% nichtnegative Terme
    &\quad 
    + \frac{1}{2} \norm{\mu_h^n}_h^2 
    + c_1 \norm{\nabla\mu_h^n}_{L^2}^2 
    +  c_2  \nnorm{\nabla \sigma_h^n}_{L^2}^2
    + c_3 \norm{\sigma_h^n}_{h,{\partial\Omega }}^2
    \\
    %%% rechte Seite
    &\leq 
    C(A, R_0, R_3,K, \chi_\phi, \chi_\sigma, \Omega, {\partial\Omega }) \big( 1
    + \norm{\sigma_{\infty,h}^n}_{h,{\partial\Omega }}^2 \big)
    + \frac{2\chi_\phi^2}{\chi_\sigma \Delta t} \Big( \norm{\phi_h^n}_h^2 + \norm{\phi_h^{n-1}}_h^2 \Big)
    \\
    %%% sigma und nabla phi
    &\quad 
    + c_4 \norm{\sigma_h^n}_h^2 
    + c_5  \norm{\nabla\phi_h^n}_{L^2}^2 
    + c_6 \norm{\phi_h^{n-1}}_h^2
    + c_7  \norm{\phi_h^n}_h^2.
\end{split}
\end{align}

Multiplying both sides of \eqref{eq:energy_FE_8} with $\Delta t$ and summing from $n=1,...,m$, where $m=1,...,N_T$, yields
\begin{align}
\begin{split}
\label{eq:energy_FE_9}
    %%% Zeitableitungen
    & \calF_h(\phi_h^m, \sigma_h^m) 
    + \frac{B}{2}  \sum_{n=1}^m \norm{\nabla\phi_h^n - \nabla\phi_h^{n-1}}_{L^2}^2 
    + \frac{\chi_\sigma}{4}  \sum_{n=1}^m \norm{\sigma_h^n - \sigma_h^{n-1}}_h^2
    \\
    %%% nichtnegative Terme
    &\quad 
    + \Delta t \sum_{n=1}^m \Big( \frac{1}{2} \norm{\mu_h^n}_h^2 
    + c_1 \norm{\nabla\mu_h^n}_{L^2}^2 
    +  c_2  \norm{\nabla \sigma_h^n}_{L^2}^2
    + c_3 \norm{\sigma_h^n}_{h,{\partial\Omega }}^2 \Big)
   \\
    %%% rechte Seite
    &\leq 
    \abs{\calF_h(\phi_h^0, \sigma_h^0)} + C \Big( T
    + \Delta t \sum_{n=1}^{N_T} \norm{\sigma_{\infty,h}^n}_{h,{\partial\Omega }}^2 \Big)
    + \sum_{n=1}^m \frac{2\chi_\phi^2}{\chi_\sigma } \Big( \norm{\phi_h^n}_h^2 + \norm{\phi_h^{n-1}}_h^2 \Big)
    \\
    %%% sigma und nabla phi
    &\quad 
    + \Delta t \sum_{n=1}^m \Big( c_4 \norm{\sigma_h^n}_h^2 
    + c_5  \norm{\nabla\phi_h^n}_{L^2}^2 
    + c_6 \norm{\phi_h^{n-1}}_h^2
    + c_7  \norm{\phi_h^n}_h^2 \Big).
\end{split}
\end{align}
%%%%%%%%%%%%%%%%
By Hölder's and Young's inequalities, we have
\begin{align}
\label{eq:energy_FE_10}
    &\int_\Omega \calI_h \Big[ \sigma_h^m (1-\phi_h^m) \Big] \dx
    \leq \frac{\chi_\sigma}{4} \norm{\sigma_h^m}_h^2
    + \frac{2 \chi_\phi^2}{\chi_\sigma } \norm{\phi_h^m}_h^2
    + C(\chi_\sigma, \chi_\phi, \Omega).
\end{align}
Moreover, we obtain from \eqref{A4_1} that
\begin{align}
\label{eq:energy_FE_11}
    & R_1 \norm{\phi_h^m}_h^2
    \leq  \int_\Omega \calI_h \big[ \psi(\phi_h^m) \big] \dx 
    + R_2 \abs{\Omega}.
\end{align}
%%%%%%%%%%%%%%%%%%
Hence, we can deduce from \eqref{eq:energy_FE_9}--\eqref{eq:energy_FE_11} that
\begin{align}
\begin{split}
\label{eq:energy_FE_12}
    %%% Zeitableitungen
    & \Big( A R_1 - \frac{2 \chi_\phi^2}{\chi_\sigma}\Big) \norm{\phi_h^m}_h^2
    + \frac{B}{2}  \norm{\nabla\phi_h^m}_{L^2}^2
    + \frac{\chi_\sigma}{4} \norm{\sigma_h^m}_h^2
    \\
    &\quad
    + \frac{B}{2}  \sum_{n=1}^m \norm{\nabla\phi_h^n - \nabla\phi_h^{n-1}}_{L^2}^2 
    + \frac{\chi_\sigma}{4}  \sum_{n=1}^m \norm{\sigma_h^n - \sigma_h^{n-1}}_h^2
    \\
    %%% nichtnegative Terme
    &\quad 
    + \Delta t \sum_{n=1}^m \Big( \frac{1}{2} \norm{\mu_h^n}_h^2 
    + c_1 \norm{\nabla\mu_h^n}_{L^2}^2 
    +  c_2  \norm{\nabla \sigma_h^n}_{L^2}^2
    + c_3 \norm{\sigma_h^n}_{h,{\partial\Omega }}^2 \Big)
   \\
    %%% rechte Seite
    &\leq 
    \abs{\calF_h(\phi_h^0, \sigma_h^0)} + C \Big( T
    + \Delta t \sum_{n=1}^{N_T} \norm{\sigma_{\infty,h}^n}_{h,{\partial\Omega }}^2 \Big)
    + \sum_{n=1}^m \frac{2\chi_\phi^2}{\chi_\sigma } \Big( \norm{\phi_h^n}_h^2 + \norm{\phi_h^{n-1}}_h^2 \Big)
    \\
    %%% sigma und nabla phi
    &\quad 
    + \Delta t \sum_{n=1}^m \Big( c_4 \norm{\sigma_h^n}_h^2 
    + c_5  \norm{\nabla\phi_h^n}_{L^2}^2 
    + c_6 \norm{\phi_h^{n-1}}_h^2
    + c_7  \norm{\phi_h^n}_h^2 \Big).
\end{split}
\end{align}

To apply a discrete Gronwall argument, i.e.~Lemma \ref{lemma:gronwall_discrete}, we have to absorb all terms on the right-hand side with index $n=m$. Hence, we obtain from \eqref{eq:energy_FE_12} that
\begin{align}
\begin{split}
\label{eq:energy_FE_13}
    %%% Zeitableitungen
    & \Big( A R_1 - \frac{4 \chi_\phi^2}{\chi_\sigma} - c_7 \Delta t \Big) \norm{\phi_h^m}_h^2
    + \Big( \frac{B}{2}  - c_5 \Delta t \Big)  \norm{\nabla\phi_h^m}_{L^2}^2
    + \Big( \frac{\chi_\sigma}{4} - c_4 \Delta t \Big) \norm{\sigma_h^m}_h^2
    \\
    &\quad
    + \frac{B}{2}  \sum_{n=1}^m \norm{\nabla\phi_h^n - \nabla\phi_h^{n-1}}_{L^2}^2 
    + \frac{\chi_\sigma}{4}  \sum_{n=1}^m \norm{\sigma_h^n - \sigma_h^{n-1}}_h^2
    \\
    %%% nichtnegative Terme
    &\quad 
    + \Delta t \sum_{n=1}^m \Big( \frac{1}{2} \norm{\mu_h^n}_h^2 
    + c_1 \norm{\nabla\mu_h^n}_{L^2}^2 
    +  c_2  \norm{\nabla \sigma_h^n}_{L^2}^2
    + c_3 \norm{\sigma_h^n}_{h,{\partial\Omega }}^2 \Big)
   \\
    %%% rechte Seite
    &\leq 
    \abs{\calF_h(\phi_h^0, \sigma_h^0)} + C \Big( T
    + \Delta t \sum_{n=1}^{N_T} \norm{\sigma_{\infty,h}^n}_{h,{\partial\Omega }}^2 \Big)
    + \sum_{n=0}^{m-1} \Big( \frac{4\chi_\phi^2}{\chi_\sigma } + (c_6 + c_7) \Delta t \Big) \norm{\phi_h^n}_h^2 
    \\
    %%% sigma und nabla phi
    &\quad 
    + \Delta t \sum_{n=1}^{m-1} \Big( c_4 \norm{\sigma_h^n}_h^2 
    + c_5  \norm{\nabla\phi_h^n}_{L^2}^2  \Big).
\end{split}
\end{align}

%%%%%%%%%%%%%%
We need to make sure that all coefficients on the left-hand side are positive which can be achieved with the following restriction to the time step size $\Delta t$:
\begin{align}
\label{eq:energy_FE_dt}
    \Delta t &< \Delta t_* \coloneqq 
    \min \bigg\{ \frac{B}{2 c_5}, \
    \frac{\chi_\sigma}{4 c_4}, \
    \frac{A R_1 - \frac{4 \chi_\phi^2}{\chi_\sigma}}{c_7}
    \bigg\},
\end{align}
where the constants $c_4, c_5, c_7>0$ are defined by \eqref{eq:energy_FE_const}. We remark that $A R_1 - \frac{4 \chi_\phi^2}{\chi_\sigma}>0$ by assumption \eqref{A4_3}. 
Hence, we obtain from Lemma \ref{lemma:gronwall_discrete}, \eqref{eq:bounds_initial} and \eqref{eq:bounds_bc} that
\begin{align}
\begin{split}
\label{eq:energy_FE_14}
    %%% Zeitableitungen
    &  \norm{\phi_h^m}_h^2
    +  \norm{\nabla\phi_h^m}_{L^2}^2
    +  \norm{\sigma_h^m}_h^2
    +  \sum_{n=1}^m \norm{\nabla\phi_h^n - \nabla\phi_h^{n-1}}_{L^2}^2 
    +  \sum_{n=1}^m \norm{\sigma_h^n - \sigma_h^{n-1}}_h^2
    \\
    &\quad
    + \Delta t \sum_{n=1}^m \Big( \norm{\mu_h^n}_h^2 
    + \norm{\nabla\mu_h^n}_{L^2}^2 
    + \norm{\nabla \sigma_h^n}_{L^2}^2
    + \norm{\sigma_h^n}_{h,{\partial\Omega }}^2 \Big)
   \\
    %%% rechte Seite
    &\leq 
    C \Big( T + \abs{\calF_h(\phi_h^0, \sigma_h^0)}
    + \Delta t \sum_{n=1}^{N_T} \norm{\sigma_{\infty,h}^n}_{h,{\partial\Omega }}^2 \Big) \cdot \exp( CT )
    \leq C,
\end{split}
\end{align}
for some constants $C>0$ that are independent of $h$ and $\Delta t$.
Taking the maximum over $m=1,...,N_T$ on the left-hand side yields the desired result.
\end{proof}

%% file: results/fem_3_existence.tex
%%%%%%%%%%%%%%%%%%%%%%%%%%%%%%%%%%%%%%%%%%%%
%%%%%%%%%%%%%%%%%%%%%%%%%%%%%%%%%%%%%%%%%%%%
\section{Existence and continuous dependence of discrete solutions}

We recall the following lemma from \cite[Chap. 9.1]{evans_2010} which is a direct consequence of Brouwer's fixed point theorem.
\begin{lemma}[Zeros of a vector field]
\label{lemma:zeros_vector_field}
For $n\in\N$, assume that the continuous function $\pmb v\colon \R^n\to\R^n$ satisfies
\begin{align*}
    \pmb v(\pmb x)\cdot \pmb x \geq 0 \ \ \textit{ if } \abs{\pmb x}=R,
\end{align*}
for some $R>0$. Then there exists a point $\pmb x\in B_R(0)$ such that $\pmb v(\pmb x)=0$.
\end{lemma}

Now we can establish the following existence result.

\begin{theorem}[Existence]
\label{theorem:existence_FE}
Let $\phi_h^0,\sigma_h^0 \in \calS_h$ and for $n=1,...,N_T$, let $\sigma_{\infty,h}^n\in \calS_h$ be given by \eqref{eq:def_initial} and \eqref{eq:def_bc}, respectively. Furthermore, assume that $\Delta t < \Delta t_*$, where $\Delta t_*$ is given by \eqref{eq:energy_FE_dt}. Then, for all $n=1,...,N_T$, there exists a solution triplet $\big( \phi_h^n,\mu_h^n,\sigma_h^n \big)\in (\calS_h)^3$ of \eqref{eq:phi_FE}--\eqref{eq:sigma_FE} which fulfills \eqref{eq:energy_FE}.
\end{theorem}

\input{results/proof_existence}

%%%%%%%%%%%%%%%%%%%%%%%%%

In the next theorem, we assume that the source terms $\Gamma_\phi(\cdot,\cdot)$ and $\Gamma_\sigma(\cdot,\cdot)$ are Lipschitz continuous in both arguments and that the mobility functions $m(\cdot)$ and $n(\cdot)$ are constant. This makes it possible to show that solutions of \eqref{eq:phi_FE}-\eqref{eq:sigma_FE} depend continuously on the inital and boundary data if the time step size $\Delta t>0$ is small enough. In particular, discrete solutions are unique.

%%%%%%%%%%%%%%%%%%%%%%%%

\begin{theorem}[Continuous dependence]
\label{theorem:FE_cont_dep}

Let $\Gamma_\phi,\Gamma_\sigma \in C^{0,1}(\R^2)$ with Lipschitz constants $L_{\Gamma_\phi}, L_{\Gamma_\sigma}>0$. Moreover, suppose that $m(\cdot)=n(\cdot)=1$.
For $i=1,2$ and $n=1,...,N_T$, let $(\phi_{h,i}^n, \mu_{h,i}^n, \sigma_{h,i}^n) \in(\calS_h)^3$ be solutions of \eqref{eq:phi_FE}--\eqref{eq:sigma_FE} with corresponding initial data $\phi_{h,i}^0, \sigma_{h,i}^0\in \calS_h$ and boundary data $\sigma_{\infty,h,i}^n \in \calS_h$, $n=1,...,N_T$. 
Let
\begin{align}
\label{eq:FE_cont_dep_dt}
    \Delta t  <
    \frac{B}{2 A^2L_{\psi_1'}^2 
    + 4 \chi_\phi^2 
    + 3 B ( L_{\Gamma_\phi} + L_{\Gamma_\sigma})}.
\end{align}
Then, there exist constants $C>0$ that are independent of $h, \Delta t$ such that
\begin{align}
\begin{split}
\label{eq:FE_cont_dep}
    &\max_{m=1,...,N_T}
    \Big( \norm{\phi_{h,1}^m-\phi_{h,2}^m}_h^2 
    + \norm{\sigma_{h,1}^m-\sigma_{h,2}^m}_h^2 \Big)
    \\
    &\quad
    + \sum_{n=1}^{N_T} \Big(
    \norm{(\phi_{h,1}^n - \phi_{h,2}^n) 
    - (\phi_{h,1}^{n-1} - \phi_{h,2}^{n-1})}_h^2 
    + \norm{(\sigma_{h,1}^n-\sigma_{h,2}^n) 
    - (\sigma_{h,1}^{n-1}-\sigma_{h,2}^{n-1})}_h^2 \Big)
    \\
    &\quad
    + \Delta t \sum_{n=1}^{N_T} \Big( 
    \norm{\mu_{h,1}^n-\mu_{h,2}^n}_h^2 
    + \norm{\nabla ( \sigma_{h,1}^n-\sigma_{h,2}^n) }_{L^2}^2
    + \norm{\sigma_{h,1}^n-\sigma_{h,2}^n}_{h,{\partial\Omega }}^2 \Big)
    \\
    %%%%%%%%%%
    &\leq
    C \Big( \norm{\phi_{h,1}^0-\phi_{h,2}^0}_h^2 
    + \norm{\sigma_{h,1}^0-\sigma_{h,2}^0}_h^2 
    + \Delta t \sum_{n=1}^{N_T} \norm{\sigma_{\infty,h,1}^n 
    - \sigma_{\infty,h,2}^n}_{h,{\partial\Omega }}^2 \Big)
    \cdot \exp(CT).
\end{split}
\end{align}
\end{theorem}

%%%%%%%%%%%%%

\input{results/proof_continuous_dependence}

%%%%%%%%%%%%%%%

\begin{remark}~
\begin{enumerate}
\item[1.]
In practice, the constants $A$ and $B$ are usually defined as $A=\frac{\beta}{\epsilon}$ and $B=\beta\epsilon$, respectively, where $\epsilon$ is a small constant. From Lemma \ref{lemma:energy_FE} and Theorem \ref{theorem:existence_FE}, we have the condition $\Delta t = \calO(\epsilon)$ for the time step size in order to obtain stability and existence of solutions of \eqref{eq:phi_FE}--\eqref{eq:sigma_FE}.
In contrast to this, and with additional assumptions on the mobility functions and source terms, we can deduce from Theorem \ref{theorem:FE_cont_dep} that the time step size must fulfill the condition $\Delta t = \calO(\epsilon^3)$ to obtain continuous dependence and, in particular, uniqueness of discrete solutions.

\item[2.]
Suppose that the source terms have the specific form
\begin{align}
\label{eq:Gamma_phi_sigma}
    \Gamma_\phi(\phi,\sigma) 
    = \big( \lambda_p \sigma - \lambda_a \big) h(\phi) ,
    \quad\quad
    \Gamma_\sigma(\phi,\sigma) 
    = \lambda_c \sigma h(\phi),
\end{align}
for all $\phi,\sigma\in\R$, where $\lambda_p,\lambda_a,\lambda_c$ are nonnegative constants referring to proliferation, apoptosis and consumption rate. Moreover, $h:\R\to\R$ is a nonnegative, bounded and Lipschitz continuous function with $h(-1)=0$ and $h(1)=1$. This specific choice of the source terms is motivated by linear kinetics and is a common choice for numerical simulations of tumour growth models \cite{ebenbeck_garcke_nurnberg_2020, GarckeLSS_2016}.%\footnote{This sentence is new.}

With the choice \eqref{eq:Gamma_phi_sigma}, continuous dependence of solutions of \eqref{eq:phi_FE}--\eqref{eq:sigma_FE} on the initial and boundary data can be shown analogously to Theorem \ref{theorem:FE_cont_dep} if the time step size is small enough. The main difference is that
\begin{align*}
    \abs{\sigma_{h,1}^n h(\phi_{h,1}^n)
    -\sigma_{h,2}^n h(\phi_{h,2}^n)} 
    &\leq \abs{\sigma_{h,1}^n} \abs{h(\phi_{h,1}^n)-h(\phi_{h,2}^n)} 
    + \abs{\sigma_{h,1}^n - \sigma_{h,2}^n} \abs{h(\phi_{h,2}^n)}
    \\
    &\leq L_h \abs{\sigma_{h,1}^n} \abs{\phi_{h,1}^n - \phi_{h,2}^n} 
    + h_\infty \abs{\sigma_{h,1}^n - \sigma_{h,2}^n},
\end{align*}
where $L_h$ denotes the Lipschitz constant of $h(\cdot)$ and $h_\infty = \norm{h(\cdot)}_{L^\infty(\R)} $. Following the proof, one then has to handle triple products of the form
\begin{align*}
    \int_\Omega \calI_h\Big[ 
    \abs{\sigma_{h,1}^n} 
    \abs{\phi_{h,1}^n - \phi_{h,2}^n}^2 \Big] \dx
    +
    \int_\Omega \calI_h\Big[ \abs{\sigma_{h,1}^n} 
    \abs{\phi_{h,1}^n - \phi_{h,2}^n} 
    \abs{\sigma_{h,1}^n - \sigma_{h,2}^n} \Big] \dx.
\end{align*}
With further calculations, these terms can be bounded by
\begin{align*}
    C_1 \norm{\phi_{h,1}^n - \phi_{h,2}^n}_{H^1}^2
    + C_2 \norm{\sigma_{h,1}^n - \sigma_{h,2}^n}_h^2
    + C_3 \norm{\nabla(\sigma_{h,1}^n - \sigma_{h,2}^n)}_{L^2}^2
\end{align*}
where the constant $C_1$ depends on $\max\limits_{n=1,...,N_T} \norm{\sigma_{h,1}^n}_h^2$, which can be bounded uniformly in $(h,\Delta t)$ if the time step size is small enough, see Lemma \ref{lemma:energy_FE}. The constants $C_2, C_3$ arise from Young's inequality and only depend on the model parameters.
The third term can be absorbed whereas the first two terms can be handled with a discrete Gronwall argument, i.e.~Lemma \ref{lemma:gronwall_discrete}, if the time step size satisfies an additional constraint.
\end{enumerate}
\end{remark}

%% file: results/proof_existence.tex
\begin{proof}

%%%%%%%%%%%%%%%%%%%%%%%
%%%%%%%%%%%%%%%%%%%%%%%

Let us define a vector field $\pmb v: \R^{3N_h}\to\R^{3N_h}$ that maps the coefficient vector $\pmb x \in \R^{3 N_h}$ of 
\begin{align*}
    \Big(\mu_h^n, \ 
    -2 \mu_h^n + \frac{1}{\Delta t}\phi_h^n, \ \chi_\sigma \sigma_h^n - \chi_\phi  \phi_h^n \Big) \in (\calS_h)^3,
\end{align*}
to the left-hand side of \eqref{eq:phi_FE}--\eqref{eq:sigma_FE}.
Then, a zero of $\pmb v$ corresponds to a solution of \eqref{eq:phi_FE}--\eqref{eq:sigma_FE}.

The aim is to show $\pmb v(\pmb x) \cdot \pmb x \geq c_1 \abs{\pmb x}^2 - c_2$ with $\abs{\pmb x}=R>0$ for some $R>0$ and some constants $c_1, c_2>0$ that are independent of $\phi_h^n, \mu_h^n, \sigma_h^n$. 
We obtain similarly to the proof of Lemma \ref{lemma:energy_FE} that
\begin{align*}
    \pmb v(\pmb x)\cdot \pmb x 
    &=
    \int_\Omega \calI_h \Big[ \Big(\frac{\phi_h^n-\phi_h^{n-1}}{\Delta t}
    - \Gamma_{\phi,h}^n \Big) \mu_h^n \Big]
    + \calI_h[m(\phi_h^{n-1})] \nabla\mu_h^n \cdot \nabla \mu_h^n \dx
    \\
    &\quad
    + \int_\Omega \calI_h \Big[ \Big( \mu_h^n  
    - A \psi_1'(\phi_h^n) - A \psi_2'(\phi_h^{n-1}) 
    + \chi_\phi \sigma_h^n \Big) 
    \big( 2 \mu_h^n - \frac{1}{\Delta t}\phi_h^n \big) \Big] \dx
    \\
    &\quad
    - \int_\Omega B \nabla\phi_h^n \cdot \nabla\big( 2 \mu_h^n - \frac{1}{\Delta t}\phi_h^n \big) \dx 
    \\
    &\quad
    + \int_\Omega \calI_h \Big[ \Big(\frac{\sigma_h^n-\sigma_h^{n-1}}{\Delta t}
    + \Gamma_{\sigma,h}^n \Big) 
    \big( \chi_\sigma \sigma_h^n - \chi_\phi \phi_h^n \big) \Big] \dx
    \\
    &\quad
    + \int_\Omega  \calI_h[n(\phi_h^{n-1})] 
    \big(\chi_\sigma \nabla\sigma_h^n - \chi_\phi \nabla\phi_h^n \big)  \cdot \nabla \big( \chi_\sigma \sigma_h^n - \chi_\phi \phi_h^n \big)  \dx
    \\
    &\quad
    + \int_{\partial\Omega } \calI_h\Big[ K \big(\sigma_h^n - \sigma_{\infty,h}^n \big) \big( \chi_\sigma \sigma_h^n - \chi_\phi \phi_h^n \big) \Big] \dH^{d-1}
    \\
    %%%%%%%%%%%%%%%%%%%
    %&\geq 
    %\Big( \frac{A R_1}{\Delta t} - \frac{4 \chi_\phi^2}{\chi_\sigma \Delta t} - c_7  \Big) \norm{\phi_h^n}_h^2
    %+ \Big( \frac{B}{2\Delta t}  - c_5  \Big)  \norm{\nabla\phi_h^n}_{L^2}^2
    %+ \Big( \frac{\chi_\sigma}{4\Delta t} - c_4  \Big) \norm{\sigma_h^n}_h^2
    %\\
    %&\quad
    %+ \frac{B}{2 \Delta t}  \norm{\nabla\phi_h^n - \nabla\phi_h^{n-1}}_{L^2}^2 
    %+ \frac{\chi_\sigma}{4 \Delta t}  \norm{\sigma_h^n - \sigma_h^{n-1}}_h^2
    %\\
    %%% nichtnegative Terme
    %&\quad 
    %+ \frac{1}{2} \norm{\mu_h^n}_h^2 
    %+ c_1 \norm{\nabla\mu_h^n}_{L^2}^2 
    %+ c_2  \norm{\nabla \sigma_h^n}_{L^2}^2
    %+ c_3 \norm{\sigma_h^n}_{h,\Gamma}^2 
    %\\
    %%% rechte Seite
    %&\quad
    %- \calF_h(\phi_h^{n-1}, \sigma_h^{n-1}) 
    %- C \Big( 1
    %+ \norm{\sigma_{\infty,h}^n}_{h,\Gamma}^2 \Big)
    %- \Big( \frac{4\chi_\phi^2}{\chi_\sigma \Delta t } + c_6 \Big) \norm{\phi_h^{n-1}}_h^2
    %\\
    %%%%%%%%%%
    &\geq 
    C \Big( \frac{1}{\Delta t} - \frac{1}{\Delta t_*} \Big) 
    \Big(\norm{\phi_h^n}_h^2
    + \norm{\nabla\phi_h^n}_{L^2}^2
    + \norm{\sigma_h^n}_h^2 \Big)
    \\
    &\quad
    + C \Big(\norm{\mu_h^n}_h^2 
    + \norm{\nabla\mu_h^n}_{L^2}^2 
    + \norm{\nabla \sigma_h^n}_{L^2}^2
    + \norm{\sigma_h^n}_{h,{\partial\Omega }}^2 \Big)
    - C(\phi_h^{n-1}, \sigma_h^{n-1}, \sigma_{\infty,h}^n)
    \\
    %%%%%%%%%%%%%%%
    &\geq 
    C \big( \abs{\pmb\phi}^2 
    + \abs{\pmb\mu}^2 
    + \abs{\pmb\sigma}^2 \big) 
    - C(\phi_h^{n-1}, \sigma_h^{n-1}, \sigma_{\infty,h}^n),
\end{align*}
for various constants $C> 0$ that are independent of $\phi_h^n, \mu_h^n, \sigma_h^n$, where $\Delta t_*$ is given by \eqref{eq:energy_FE_dt} and the coefficient vectors of $\phi_h^n, \mu_h^n, \sigma_h^n \in\calS_h$ are denoted by $\pmb\phi, \pmb\mu, \pmb\sigma \in \R^{N_h}$.

Let us remark that 
\begin{align*}
    ||| \cdot |||: \ (\pmb\phi, \pmb\mu, \pmb\sigma)
    \mapsto 
    \Big( \abs{\pmb\mu}^2
    + \aabs{-2 \pmb\mu + \frac{1}{\Delta t} \pmb\phi}^2
    + \aabs{\chi_\sigma \pmb\sigma - \chi_\phi  \pmb\phi}^2 \Big)^{1/2}
\end{align*}
defines a norm on the finite dimensional space $\R^{3 N_h}$. The definiteness of $|||\cdot|||$ can be shown as follows. Assuming $|||(\pmb\phi, \pmb\mu, \pmb\sigma)|||=0$, it follows from the first term that $\pmb\mu = \pmb 0$. The second term yields $\pmb\phi= \pmb 0$. Then, from the third term, we have $\pmb\sigma= \pmb 0$.

Hence, on noting norm equivalence in finite dimensions, 
we obtain with $R=\abs{\pmb x}=|||(\pmb\phi, \pmb\mu, \pmb\sigma)|||$ large enough that
\begin{align*}
    \pmb v(\pmb x) \cdot \pmb x
    &\geq 
    C \big( \abs{\pmb\phi}^2 
    + \abs{\pmb\mu}^2 
    + \abs{\pmb\sigma}^2 \big) 
    - C(\phi_h^{n-1}, \sigma_h^{n-1}, \sigma_{\infty,h}^n)
    \\
    &\geq 
    C \abs{\pmb x}^2 - C(\phi_h^{n-1}, \sigma_h^{n-1}, \sigma_{\infty,h}^n)
    \\
    &>0.
\end{align*}
It follows from Lemma \ref{lemma:zeros_vector_field} that there exists a zero of $\pmb v$ which corresponds to a solution of \eqref{eq:phi_FE}--\eqref{eq:sigma_FE}.
\end{proof}

%% file: results/proof_continuous_dependence.tex
\begin{proof}
For $n=1,...,N_T$, suppose there are two solutions of \eqref{eq:phi_FE}--\eqref{eq:sigma_FE} denoted by $(\phi_{h,i}^n, \mu_{h,i}^n, \sigma_{h,i}^n)$, $i=1,2$, with corresponding initial data $\phi_{h,i}^0,\sigma_{h,i}^0$ and boundary data $\sigma_{\infty,h,i}^0$. Let us denote the differences by
\begin{align*}
    \phi_h^n = \phi_{h,1}^n-\phi_{h,2}^n, \quad 
    \mu_h^n = \mu_{h,1}^n-\mu_{h,2}^n, \quad 
    \sigma_h^n = \sigma_{h,1}^n-\sigma_{h,2}^n, \quad
    \sigma_{\infty,h}^n = \sigma_{\infty,h,1}^n - \sigma_{\infty,h,2}^n.
\end{align*} 
It holds that
\begin{subequations}
\begin{align}
    \label{eq:phi_FE_cont_dep}
    \int_\Omega \calI_h \Big[ 
    (\phi_h^n - \phi_h^{n-1}) \zeta_h
    - \Delta t \Big(\Gamma_{\phi}(\phi_{h,1}^n,\sigma_{h,1}^n) 
    - \Gamma_{\phi}(\phi_{h,2}^n,\sigma_{h,2}^n) \Big) \zeta_h  \Big]
    +  \Delta t \nabla\mu_h^n \cdot \nabla \zeta_h \dx
    = 0,
    \\
    \nonumber
    \label{eq:mu_FE_cont_dep}
    \int_\Omega \calI_h \Big[ \Big( \mu_h^n  
    + \chi_\phi \sigma_h^n \Big) \rho_h 
    - A \Big( \psi_1'(\phi_{h,1}^n) - \psi_1'(\phi_{h,2}^n) 
    + \psi_2'(\phi_{h,1}^{n-1}) - \psi_2'(\phi_{h,2}^{n-1})\Big) \rho_h \Big] \dx
    \quad\quad
    \\
    - \int_\Omega B \nabla\phi_h^n \cdot \nabla\rho_h \dx 
    = 0,
    \\
    \nonumber
    \label{eq:sigma_FE_cont_dep}
    \int_\Omega \calI_h \Big[  
    (\sigma_h^n - \sigma_h^{n-1}) \xi_h
    + \Delta t \Big( \Gamma_\sigma(\phi_{h,1}^n,\sigma_{h,1}^n) 
    - \Gamma_\sigma(\phi_{h,2}^n,\sigma_{h,2}^n) \Big) \xi_h \Big] \dx
    \quad\quad
    \\
    + \int_\Omega \Delta t \big(\chi_\sigma \nabla\sigma_h^n - \chi_\phi \nabla\phi_h^n \big)  \cdot \nabla \xi_h  \dx
    + \int_{\partial\Omega } \Delta t K \calI_h\Big[ 
    (\sigma_h^n - \sigma_{\infty,h}^n) \xi_h \Big] \dH^{d-1}
    = 0,
\end{align}
\end{subequations}
for all $\zeta_h,\rho_h,\xi_h\in S^h$. 
Setting $\zeta_h= B \phi_h^n$, $\rho_h=\Delta t(\mu_h^n-\chi_\phi\sigma_h^n)$, $\xi_h= B \sigma_h^n$ in \eqref{eq:phi_FE_cont_dep}--\eqref{eq:sigma_FE_cont_dep} and adding the resulting equations yields, on noting \eqref{eq:elementary_identity}, that
\begin{align}
\begin{split}
\label{eq:FE_cont_dep_1}
    &\frac{B}{2} \Big(\norm{\phi_h^n}_h^2 
    - \norm{\phi_h^{n-1}}_h^2 
    + \norm{\phi_h^n - \phi_h^{n-1}}_h^2 
    + \norm{\sigma_h^n}_h^2 
    - \norm{\sigma_h^{n-1}}_h^2 
    + \norm{\sigma_h^n - \sigma_h^{n-1}}_h^2 \Big)
    \\
    &\quad
    + \Delta t  \norm{\mu_h^n}_h^2 
    -\Delta t \chi_\phi^2 \norm{\sigma_h^n}_h^2 
    + \Delta t  B\chi_\sigma \norm{\nabla \sigma_h^n}_{L^2}^2
    + \Delta t  B K \norm{\sigma_h^n}_{h,{\partial\Omega }}^2 
    \\
    %%%%%%%%%
    &=
    \Delta t  A \int_\Omega \calI_h \Big[ \Big(\psi_1^\prime(\phi_{h,1}^n)
    - \psi_1^\prime(\phi_{h,2}^n) 
    + \psi_2'(\phi_{h,1}^{n-1}) - \psi_2'(\phi_{h,2}^{n-1}) \Big) 
    \big( \mu_h^n-\chi_\phi\sigma_h^n \big) \Big] \dx
    \\
    &\quad
    + \Delta t  B \int_\Omega \calI_h \Big[ \Big(\Gamma_\phi(\phi_{h,1}^n,\sigma_{h,1}^n) 
    - \Gamma_\phi(\phi_{h,2}^n,\sigma_{h,2}^n)\Big) \phi_h^n 
    \Big] \dx
    \\
    &\quad
    - \Delta t  B \int_\Omega \calI_h \Big[ \Big(\Gamma_\sigma(\phi_{h,1}^n,\sigma_{h,1}^n) 
    - \Gamma_\sigma(\phi_{h,2}^n,\sigma_{h,2}^n)\Big) \sigma_h^n \Big] \dx 
    \\
    &\quad 
    +\Delta t B K \int_{\partial\Omega } \calI_h \big[ \sigma_{\infty,h}^n \sigma_h^n \big] \dH^{d-1}.
\end{split}
\end{align}
Using the Lipschitz assumptions of $\psi_1',\psi_2',\Gamma_\phi, \Gamma_\sigma$, we have
\begin{align*}
    \abs{\psi_1^\prime(\phi_{h,1}^n)
    - \psi_1^\prime(\phi_{h,2}^n)} 
    &\leq L_{\psi_1^\prime} \abs{\phi_h^n},
    \\
    \abs{\psi_2^\prime(\phi_{h,1}^{n-1})
    - \psi_2^\prime(\phi_{h,2}^{n-1})} 
    &\leq L_{\psi_2^\prime} \abs{\phi_h^{n-1}},
    \\
    \aabs{ \Gamma_\phi(\phi_{h,1}^n,\sigma_{h,1}^n) 
    - \Gamma_\phi(\phi_{h,2}^n,\sigma_{h,2}^n) }
    &\leq L_{\Gamma_\phi}
    \big(\abs{\phi_h^n} + \abs{\sigma_h^n} \big),
    \\
    \aabs{ \Gamma_\sigma(\phi_{h,1}^n,\sigma_{h,1}^n) 
    - \Gamma_\sigma(\phi_{h,2}^n,\sigma_{h,2}^n) }
    &\leq L_{\Gamma_\sigma}
    \big(\abs{\phi_h^n} + \abs{\sigma_h^n} \big).
\end{align*}

On noting Young's inequality, we obtain together with \eqref{eq:FE_cont_dep_1}, that
\begin{align}
\begin{split}
\label{eq:FE_cont_dep_2}
    &\frac{B}{2} \Big(\norm{\phi_h^n}_h^2 
    - \norm{\phi_h^{n-1}}_h^2 
    + \norm{\phi_h^n - \phi_h^{n-1}}_h^2 
    + \norm{\sigma_h^n}_h^2 
    - \norm{\sigma_h^{n-1}}_h^2 
    + \norm{\sigma_h^n - \sigma_h^{n-1}}_h^2 \Big)
    \\
    &\quad
    + \Delta t  \norm{\mu_h^n}_h^2 
    -\Delta t \chi_\phi^2 \norm{\sigma_h^n}_h^2 
    + \Delta t  B\chi_\sigma \norm{\nabla \sigma_h^n}_{L^2}^2
    + \Delta t  B K \norm{\sigma_h^n}_{h,{\partial\Omega }}^2 
    \\
    %%%%%%%%%
    &\leq
    \Delta t  A  \int_\Omega \calI_h \Big[ 
    \big( L_{\psi_1^\prime} \abs{\phi_h^n} 
    + L_{\psi_2^\prime} \abs{\phi_h^{n-1}} \big) 
    \big( \abs{\mu_h^n} + \chi_\phi \abs{\sigma_h^n} \big) \Big] \dx
    \\
    &\quad
    +\Delta t  B  \int_\Omega \calI_h \Big[ 
    L_{\Gamma_\phi} \big(\abs{\phi_h^n} + \abs{\sigma_h^n} \big) \abs{\phi_h^n} 
    + L_{\Gamma_\sigma} \big(\abs{\phi_h^n} + \abs{\sigma_h^n} \big) \abs{\sigma_h^n} \Big] \dx 
    \\
    &\quad
    + \frac{1}{2} \Delta t B K \big( \norm{\sigma_{\infty,h}^n}_{h,{\partial\Omega }}^2 
    + \norm{\sigma_h^n}_{h,{\partial\Omega }}^2 \big)
    \\
    %%%%%%%%%%
    &\leq
    \frac{1}{2} \Delta t \norm{\mu_h^n}_h^2 
    + \Delta t  \big(  A^2 L_{\psi_1^\prime}^2 + \chi_\phi^2 \big) \big( \norm{\phi_h^n}_h^2 
    + \norm{\sigma_h^n}_h^2 \big)
    + C(A,\chi_\phi,L_{\psi_1'},L_{\psi_2'}) 
    \Delta t  \norm{\phi_h^{n-1}}_h^2 
    \\
    &\quad
    + \Delta t  B \big( \frac{3}{2} L_{\Gamma_\phi} 
    + \frac{1}{2} L_{\Gamma_\sigma} \big) \norm{\phi_h^n}_h^2 
    + \Delta t B \big( \frac{1}{2} L_{\Gamma_\phi} 
    + \frac{3}{2} L_{\Gamma_\sigma} \big)  \norm{\sigma_h^n}_h^2 
    \\
    &\quad
    + \frac{1}{2} \Delta t B K \big( \norm{\sigma_{\infty,h}^n}_{h,{\partial\Omega }}^2 
    + \norm{\sigma_h^n}_{h,{\partial\Omega }}^2 \big).
\end{split}
\end{align}

%%%%%%%%%%%%%%%%%%

Absorbing the terms on the right-hand side and summing from $n=1,...,m$, where $m=1,...,N_T$, leads to
\begin{align}
\begin{split}
\label{eq:FE_cont_dep_3}
    &\frac{B}{2} \big( \norm{\phi_h^m}_h^2 
    + \norm{\sigma_h^m}_h^2 \big)
    + \frac{B}{2} \sum_{n=1}^m \Big(
    \norm{\phi_h^n - \phi_h^{n-1}}_h^2 
    + \norm{\sigma_h^n - \sigma_h^{n-1}}_h^2 \Big)
    \\
    &\quad
    +  \Delta t \sum_{n=1}^m \Big( 
    \frac{1}{2} \norm{\mu_h^n}_h^2 
    +  B\chi_\sigma \norm{\nabla \sigma_h^n}_{L^2}^2
    + \frac{1}{2}  B K \norm{\sigma_h^n}_{h,{\partial\Omega }}^2 \Big)
    \\
    %%%%%%%%%%
    &\leq
    \frac{B}{2} \Big( \norm{\phi_h^0}_h^2 
    + \norm{\sigma_h^0}_h^2 
    + \Delta t \sum_{n=1}^{N_T} \norm{\sigma_{\infty,h}^n}_{h,{\partial\Omega }}^2 \Big)
    + \Delta t \sum_{n=0}^{m-1}  C(A,\chi_\phi,L_{\psi_1'},L_{\psi_2'})  \norm{\phi_h^n}_h^2 
    \\
    &\quad
    + \Delta t \sum_{n=1}^m \Big(
    2 \chi_\phi^2 
    +  A^2 L_{\psi_1^\prime}^2 
    +  \frac{3}{2} B ( L_{\Gamma_\phi} + L_{\Gamma_\sigma})
    \Big) \Big( \norm{\phi_h^n}_h^2  + \norm{\sigma_h^n}_h^2 \Big).
\end{split}
\end{align}

%%%%%%%%%%%%%%%%%%%%%

In order to apply a discrete Gronwall argument, we absorb the terms on the right-hand side of \eqref{eq:FE_cont_dep_3} with index $n=m$. Therefore, we receive 
\begin{align}
\begin{split}
\label{eq:FE_cont_dep_4}
    &\Big( \frac{B}{2} 
    - \Delta t \big( 2 \chi_\phi^2 
    +  A^2 L_{\psi_1^\prime}^2 
    +  \frac{3}{2} B ( L_{\Gamma_\phi} + L_{\Gamma_\sigma})
    \big) \Big)
    \Big( \norm{\phi_h^m}_h^2 
    + \norm{\sigma_h^m}_h^2 \Big)
    \\
    &\quad
    + \frac{B}{2} \sum_{n=1}^m \Big(
    \norm{\phi_h^n - \phi_h^{n-1}}_h^2 
    + \norm{\sigma_h^n - \sigma_h^{n-1}}_h^2 \Big)
    \\
    &\quad
    +  \Delta t \sum_{n=1}^m \Big( 
    \frac{1}{2} \norm{\mu_h^n}_h^2 
    +  B\chi_\sigma \norm{\nabla \sigma_h^n}_{L^2}^2
    + \frac{1}{2}  B K \norm{\sigma_h^n}_{h,{\partial\Omega }}^2 \Big)
    \\
    %%%%%%%%%%
    &\leq
    \frac{B}{2} \Big( \norm{\phi_h^0}_h^2 
    + \norm{\sigma_h^0}_h^2 
    + \Delta t \sum_{n=1}^{N_T} \norm{\sigma_{\infty,h}^n}_{h,{\partial\Omega }}^2 \Big)
    \\
    &\quad
    + \Delta t \sum_{n=0}^{m-1} 
    C(A,\chi_\phi,L_{\psi_1'},L_{\psi_2'},B, L_{\Gamma_\phi}, L_{\Gamma_\sigma}) \Big( \norm{\phi_h^n}_h^2  + \norm{\sigma_h^n}_h^2 \Big).
\end{split}
\end{align}

%%%%%%%%%%%%%
The terms on the left-hand side are nonnegative supposed that
\begin{align}
    \Delta t  <
    \frac{B}{2 A^2L_{\psi_1'}^2 
    + 4 \chi_\phi^2 
    + 3 B ( L_{\Gamma_\phi} + L_{\Gamma_\sigma})}.
\end{align}
Hence, we can deduce from Lemma \ref{lemma:gronwall_discrete} that there exist constants $C>0$ that are independent of $h, \Delta t$ such that
\begin{align}
\begin{split}
\label{eq:FE_cont_dep_5}
    &\norm{\phi_h^m}_h^2 
    + \norm{\sigma_h^m}_h^2
    + \sum_{n=1}^m \Big(
    \norm{\phi_h^n - \phi_h^{n-1}}_h^2 
    + \norm{\sigma_h^n - \sigma_h^{n-1}}_h^2 \Big)
    \\
    &\quad
    +  \Delta t \sum_{n=1}^m \Big( 
    \norm{\mu_h^n}_h^2 
    + \norm{\nabla \sigma_h^n}_{L^2}^2
    + \norm{\sigma_h^n}_{h,{\partial\Omega }}^2 \Big)
    \\
    %%%%%%%%%%
    &\leq
    C \Big( \norm{\phi_h^0}_h^2 
    + \norm{\sigma_h^0}_h^2 
    + \Delta t \sum_{n=1}^{N_T} \norm{\sigma_{\infty,h}^n}_{h,{\partial\Omega }}^2 \Big)
    \cdot \exp(CT).
\end{split}
\end{align}
Taking the maximum over $m=1,...,N_T$ on the left-hand side proves the result.
\end{proof}

%% file: results/fem_4_higher_order.tex
\section{Higher order estimates}

In this section we prove higher order estimates for solutions of \eqref{eq:phi_FE}--\eqref{eq:sigma_FE}.
This is needed in order to show more compactness properties for $\phi$ in space dimensions $d=2,3$ which is needed in presence of the nodal interpolation operator $\calI_h$. However, for $d=1$, the stability estimates \eqref{eq:energy_FE} give enough spatial regularity to pass to the limit in the scheme \eqref{eq:phi_FE}--\eqref{eq:sigma_FE}.

At first, we introduce the projection operator $\hat Q_h: L^2(\Omega) \to \calS_h$ defined by
\begin{align}
\label{def:projection_Q}
    \int_\Omega \calI_h \Big[ \hat Q_h \eta \zeta_h \Big] \dx
    = \int_\Omega \eta \zeta_h \dx, \quad \forall \zeta_h\in \calS_h.
\end{align}
It holds, see, e.g., \cite{barrett_blowey_1996, barrett_blowey_garcke_2000}:
\begin{align}
    \label{eq:projection_Q_hat}
    \norm{\eta - \hat Q_h \eta}_{L^2}
    + h \norm{\nabla \eta - \nabla \hat Q_h \eta}_{L^2}
    & \leq C h \norm{\nabla \eta}_{L^2} 
    \qquad \forall \eta \in H^1(\Omega),
\end{align}
with a constant $C>0$ which is independent of $h$.

%%%%%%%%%%%%%%%%%

\begin{lemma}
\label{lemma:bounds_higher_order}
Let the assumptions of Theorem \ref{theorem:existence_FE} hold. Then it holds
\begin{subequations}
\begin{align}
\begin{split}
    \label{eq:bounds_higher_order}
    & \Delta t \sum_{n=1}^{N_T} \Big( 
    %\norm{\phi_h^n - \phi_h^{n-1}}_{L^2}^2
    \norm{\Delta_h \phi_h^n}_{L^2}^2 
    + \nnnorm{\frac{\phi_h^n - \phi_h^{n-1}}{\Delta t}}_{(H^1)'}^2 
    + \nnnorm{\frac{\sigma_h^n - \sigma_h^{n-1}}{\Delta t}}_{(H^1)'}^2 \Big)
    \leq C,
\end{split}
\end{align}
and
\begin{align}
    \label{eq:bounds_phi_translation}
    \Delta t 
    \sum_{n=0}^{N_T - l}
    \norm{ \phi_h^{n+l}
    - \phi_h^n}_{L^2}^2 
    \leq C l \Delta t,
\end{align}
for any $l \in \{1,...,N_T\}$, where the constants $C>0$ are independent of $h,\Delta t$.
\end{subequations}
\end{lemma}

%%% proof
\input{results/proof_reg}

%% file: results/proof_reg.tex
\begin{proof}
On noting the definition of the discrete Laplacian \eqref{eq:discr_laplace}, we can rewrite \eqref{eq:mu_FE} as
\begin{align*}
    \int_\Omega \calI_h \Big[ \Big( - \mu_h^n  
    + A \psi_1'(\phi_h^n) + A \psi_2'(\phi_h^{n-1}) 
    - \chi_\phi \sigma_h^n
    - B \Delta_h \phi_h^n \Big) \rho_h \Big] 
    = 0, \quad\quad \forall \zeta_h\in\calS_h.
\end{align*}
Choosing $\rho_h = \Delta_h \phi_h^n$ and dividing both sides by $B$ yields
\begin{align*}
    \norm{\Delta_h \phi_h^n}_h^2
    %\int_\Omega \calI_h\big[ \abs{\Delta_h \phi_h^n}^2 \big] \dx
    &= \frac{1}{B} \int_\Omega \calI_h \Big[
    \big(- \mu_h^n + A \psi_1'(\phi_h^n) + A \psi_2'(\phi_h^{n-1})
    - \chi_\phi \sigma_h^n \big) \Delta_h \phi_h^n \Big].
\end{align*}
Together with Young's inequality and the growth assumptions on $\psi_1', \psi_2'$, we obtain
\begin{align*}
    \norm{\Delta_h \phi_h^n}_h^2
    &\leq 
    \frac{1}{2} \norm{\Delta_h \phi_h^n}_h^2 
    + \frac{2}{B^2} \norm{\mu_h^n}_h^2 
    + \frac{2A^2}{B^2} \norm{\psi_1'(\phi_h^n)}_h^2
    + \frac{2A^2}{B^2} \norm{\psi_2'(\phi_h^{n-1})}_h^2 
    + \frac{2\chi_\phi^2}{B^2}  \norm{\sigma_h^n}_h^2 
    \\
    &\leq 
    \frac{1}{2} \norm{\Delta_h \phi_h^n}_h^2 
    + C\Big( 1 
    + \norm{\mu_h^n}_h^2 
    + \norm{\phi_h^n}_h^2
    + \norm{\phi_h^{n-1}}_h^2 
    + \norm{\sigma_h^n}_h^2 \Big).
\end{align*}

Absorbing $\frac{1}{2} \norm{\Delta_h \phi_h^n}_h^2 $ to the left-hand side, noting \eqref{eq:energy_FE}, multiplying both sides with $\Delta t$ and summing from $n=1,...,N_T$ yields 
\begin{align}
\begin{split}
    \Delta t \sum\limits_{n=1}^{N_T} 
    \norm{\Delta_h \phi_h^n}_h^2 
    &\leq 
    \Delta t \sum\limits_{n=1}^{N_T} C\Big( 1 
    + \norm{\mu_h^n}_h^2 
    + \norm{\phi_h^n}_h^2
    + \norm{\phi_h^{n-1}}_h^2 
    + \norm{\sigma_h^n}_h^2 \Big)
    \leq C.
\end{split}
\end{align}
Applying \eqref{eq:norm_equiv} leads to the first bound in \eqref{eq:bounds_higher_order}.

%%%%%%%%%%%%
Let $\zeta\in H^1(\Omega)$. On noting \eqref{def:projection_Q}, \eqref{eq:projection_Q_hat}, \eqref{eq:phi_FE}, Hölder's inequality, the linear growth of $\Gamma_\phi$ and \eqref{eq:norm_equiv}, we obtain
\begin{align*}
    &\int_\Omega \Big(\frac{\phi_h^n - \phi_h^{n-1}}{\Delta t}\Big) \zeta \dx
    =
    \int_\Omega \calI_h \Big[ \Big(\frac{\phi_h^n - \phi_h^{n-1}}{\Delta t}\Big) \hat Q_h \zeta \Big] \dx
    \\
    &= \int_\Omega \calI_h\big[ \Gamma_{\phi,h}^n \hat Q_h\zeta \big] 
    - m(\phi_h^{n-1})\nabla\mu_h^n \cdot \nabla\hat Q_h\zeta \dx
    \\
    %%%
    &\leq C \big(1 + \norm{\phi_h^n}_h + \norm{\sigma_h^n}_h \big)
    \norm{\hat Q_h \zeta}_h
    + m_1 \norm{\nabla \mu_h^n}_{L^2} \norm{\nabla \hat Q_h\zeta}_{L^2}
    \\
    &\leq C( 1 
    + \norm{\phi_h^n}_h + \norm{\sigma_h^n}_h 
    + \norm{\nabla \mu_h^n}_{L^2} \Big)
    \norm{\zeta}_{H^1},
\end{align*}
which gives
\begin{align}
\label{eq:bounds_higher_order_1}
    \nnnorm{\frac{\phi_h^n - \phi_h^{n-1}}{\Delta t}}_{(H^1)'} 
    \leq C\Big( 1 
    + \norm{\phi_h^n}_h + \norm{\sigma_h^n}_h 
    + \norm{\nabla \mu_h^n}_{L^2} \Big).
\end{align}

%%%%%%%%%%%%%%%%%%

Similarly, we receive from \eqref{def:projection_Q}, \eqref{eq:projection_Q_hat}, \eqref{eq:sigma_FE}, Hölder's inequality, the linear growth of $\Gamma_\sigma$, \eqref{eq:norm_equiv}, \eqref{eq:norm_equiv_Gamma} and the trace theorem, that
\begin{align*}
    &\int_\Omega \Big(\frac{\sigma_h^n - \sigma_h^{n-1}}{\Delta t} \Big) \zeta \dx
    =
    \int_\Omega \calI_h \Big[ \Big(\frac{\sigma_h^n - \sigma_h^{n-1}}{\Delta t}\Big) \hat Q_h \zeta \Big] \dx
    \\
    &= \int_\Omega - \calI_h\big[ \Gamma_{\sigma,h}^n \hat Q_h\zeta \big] 
    - n(\phi_h^{n-1})(\chi_\sigma \nabla\sigma_h^n - \chi_\phi \nabla\phi) \cdot \nabla\hat Q_h\zeta \dx
    \\
    &\quad
    + \int_{\partial\Omega } K \calI_h\Big[ \big( \sigma_h^n - \sigma_{\infty,h}^n \big) \hat Q_h\zeta \Big] \dH^{d-1}
    \\
    %%%
    &\leq C \big(1 + \norm{\phi_h^n}_h + \norm{\sigma_h^n}_h \big)
    \norm{\hat Q_h \zeta}_h
    + n_1 \Big( \norm{\chi_\sigma \nabla\sigma_h^n}_{L^2} 
    +\norm{\chi_\phi \nabla\phi}_{L^2} \Big)
    \norm{\nabla \hat Q_h\zeta}_{L^2}
    \\
    &\quad
    + K \Big( \norm{\sigma_h^n}_{h,{\partial\Omega }} 
    + \norm{\sigma_{\infty,h}^n}_{h,{\partial\Omega }} \Big)
    \norm{\hat Q \zeta}_{h,{\partial\Omega }} 
    \\
    &\leq C \Big( 1 
    + \norm{\phi_h^n}_h + \norm{\sigma_h^n}_h 
    + \norm{\nabla \phi_h^n}_{L^2}
    + \norm{\nabla \sigma_h^n}_{L^2} 
    + \norm{\sigma_{\infty,h}^n}_{h,{\partial\Omega }} 
    + \norm{\sigma_h^n}_{h,{\partial\Omega }}\Big)
    \norm{\zeta}_{H^1},
\end{align*}
which yields
\begin{align}
\label{eq:bounds_higher_order_2}
\begin{split}
    \nnnorm{\frac{\sigma_h^n - \sigma_h^{n-1}}{\Delta t}}_{(H^1)'} 
    &\leq C \Big( 1 
    + \norm{\phi_h^n}_h + \norm{\sigma_h^n}_h 
    + \norm{\nabla \phi_h^n}_{L^2}
    + \norm{\nabla \sigma_h^n}_{L^2} 
    \\
    &\quad
    + \norm{\sigma_h^n}_{h,{\partial\Omega }} 
    + \norm{\sigma_{\infty,h}^n}_{h,{\partial\Omega }}\Big).
\end{split}
\end{align}

%%%%%%%%%%%%%%
Hence, we obtain from \eqref{eq:bounds_higher_order_1}, \eqref{eq:bounds_higher_order_2} and \eqref{eq:energy_FE}, that
\begin{align*}
    &\Delta t \sum_{n=1}^{N_T} \Big(
    \nnnorm{\frac{\phi_h^n - \phi_h^{n-1}}{\Delta t}}_{(H^1)'}^2
    + \nnnorm{\frac{\sigma_h^n - \sigma_h^{n-1}}{\Delta t}}_{(H^1)'}^2
    \Big)
    \\
    %%%
    &\leq C\Delta t \sum_{n=1}^{N_T} \Big( 1 
    + \norm{\phi_h^n}_h^2 
    + \norm{\sigma_h^n}_h^2 
    + \norm{\nabla \mu_h^n}_{L^2}^2 
    + \norm{\nabla \phi_h^n}_{L^2}^2
    + \norm{\nabla \sigma_h^n}_{L^2}^2
    \\
    &\quad
    + \norm{\sigma_h^n}_{h,{\partial\Omega }}^2 
    + \norm{\sigma_{\infty,h}^n}_{h,{\partial\Omega }}^2 \Big)
    \\
    %%%
    &\leq C.
\end{align*}
This proves the second and third bounds in \eqref{eq:bounds_higher_order}.

%On noting Young's inequality, we have
%\begin{align*}
%    \norm{\phi_h^n - \phi_h^{n-1}}_{L^2}^2
%    &\leq \Delta t \norm{\phi_h^n - \phi_h^{n-1}}_{H^1}
%    \nnnorm{\frac{\phi_h^n - \phi_h^{n-1}}{\Delta t}}_{(H^1)'}
%    \\
%    &\leq 
%    \frac{1}{2} \norm{\phi_h^n - \phi_h^{n-1}}_{H^1}^2
%    + \frac{(\Delta t)^2}{2} \nnnorm{\frac{\phi_h^n - \phi_h^{n-1}}{\Delta t}}_{(H^1)'}^2,
%\end{align*}
%which yields
%\begin{align*}
%    \norm{\phi_h^n - \phi_h^{n-1}}_{L^2}^2
%    \leq \norm{\nabla(\phi_h^n - \phi_h^{n-1})}_{L^2}^2
%    + (\Delta t)^2 \nnnorm{\frac{\phi_h^n - \phi_h^{n-1}}{\Delta t}}_{(H^1)'}^2.
%\end{align*}
%Summing from $n=1,...,N_T$ together with \eqref{eq:energy_FE} and the third bound in \eqref{eq:bounds_higher_order} prove the first bound in \eqref{eq:bounds_higher_order}.

%%%%%%%%%%%%%%%%%%%%%%%%

Next, we test \eqref{eq:phi_FE} with $\zeta_h = \Delta t (\phi_h^{m+l} - \phi_h^m)$, where $m=0,...,N_T - l$ and $l=1,...,N_T$, to obtain
\begin{align*}
    0 &= \int_\Omega \calI_h \Big[ \Big(\phi_h^n-\phi_h^{n-1}
    - \Delta t \Gamma_{\phi,h}^n \Big)  (\phi_h^{m+l} - \phi_h^m) \Big]
    + \Delta t \calI_h[m(\phi_h^{n-1})] \nabla\mu_h^n \cdot \nabla  (\phi_h^{m+l} - \phi_h^m) \dx.
\end{align*}
Summing from $n=m+1, ... , m+l$ gives
\begin{align*}
    0 &= \int_\Omega \calI_h \Big[ \abs{\phi_h^{m+l} - \phi_h^m}^2 \Big] \dx
    - \Delta t \sum_{n=m+1}^{m+l} \int_\Omega  
    \calI_h\Big[ \Gamma_{\phi}(\phi_h^n, \sigma_h^n) (\phi_h^{m+l} - \phi_h^m)  \Big] \dx
    \\
    &\quad
    + \Delta t \sum_{n=m+1}^{m+l} \int_\Omega
    \calI_h[m(\phi_h^{n-1})] \nabla\mu_h^n \cdot \nabla  (\phi_h^{m+l} - \phi_h^m) \dx,
\end{align*}
which yields on noting \eqref{eq:norm_equiv}, \eqref{eq:energy_FE}, Hölder's inequality and the assumptions on $\Gamma_\phi(\cdot,\cdot)$ and $m(\cdot)$, that
\begin{align*}
    \norm{\phi_h^{m+l} - \phi_h^m}_{L^2}^2
    &\leq
    C \Delta t \sum_{n=m+1}^{m+l} 
    \Big( \norm{\Gamma_\phi(\phi_h^n, \sigma_h^n)}_{L^2} 
    + \norm{\nabla\mu_h^n}_{L^2} 
    \Big) \norm{\phi_h^{m+l} - \phi_h^m}_{H^1}
    \\
    &\leq
    C \Delta t \sum_{n=m+1}^{m+l} 
    \Big( 1 + \norm{\phi_h^n}_{L^2} 
    + \norm{\sigma_h^n}_{L^2} 
    + \norm{\nabla\mu_h^n}_{L^2}
    \Big) \norm{\phi_h^{m+l} - \phi_h^m}_{H^1}
    \\
    &\leq 
    C \Delta t \sum_{k=1}^{l} 
    \Big( 1 + \norm{\nabla\mu_h^{m+k}}_{L^2}
    \Big) \norm{\phi_h^{m+l} - \phi_h^m}_{H^1}.
\end{align*}
Multiplying both sides by $\Delta t$, summing from $m=0,...,N_T - l$ and applying a Hölder's inequality and the bounds \eqref{eq:energy_FE} leads to
\begin{align*}
    \Delta t \sum_{m=0}^{N_T - l}
    \norm{\phi_h^{m+l} - \phi_h^m}_{L^2}^2
    &\leq 
    C (\Delta t)^2 \sum_{k=1}^l \sum_{m=0}^{N_T - l} 
    \Big( 1 + \norm{\nabla\mu_h^{m+k}}_{L^2}
    \Big) \norm{\phi_h^{m+l} - \phi_h^m}_{H^1}
    \\
    &\leq C \Delta t  \sum_{k=1}^l
    \bigg( 1 + \Big( \Delta t \sum_{m=0}^{N_T - l} \norm{\nabla\mu_h^{m+k}}_{L^2}^2 \Big)^\frac{1}{2} \bigg)
    \cdot \Big( \Delta t \sum_{m=0}^{N_T - l} \norm{\phi_h^{m+l} - \phi_h^m}_{H^1}^2
    \Big)^\frac{1}{2}
    \\
    &\leq 
    C l \Delta t.
\end{align*}
%
% Hier haben die $\Delta t$ auf der rechten Seite vor den Summen gefehlt!
%
This proves \eqref{eq:bounds_phi_translation}.
\end{proof}

%% file: results/fem_5_convergence.tex
\section{Convergence to a weak solution}

In this section, we will use compactness arguments and the bounds \eqref{eq:energy_FE}, \eqref{eq:bounds_higher_order} to show that solutions of the discrete scheme \eqref{eq:phi_FE}--\eqref{eq:sigma_FE} converge to a weak solution of \eqref{eq:phi}--\eqref{eq:sigma} when we pass to the limit $(h,\Delta t) \to (0,0)$.

%%%%%%%%%%%%%%%%%%%%%%%%%%%%%%%%%%%%%%%%%%%%

For future reference, we recall the following compactness results from \cite[Sect.~8, Cor.~4 and Thm.~5]{simon_1986}.
Let $X, Y, Z$ be Banach spaces with a compact embedding $X \hookrightarrow \hookrightarrow Y$ and a continuous embedding $Y \hookrightarrow Z$. Let $1\leq p < \infty$ and $r>1$. Then we have the following compact embeddings:
\begin{subequations}
\begin{alignat}{3}
    \label{eq:compact_Lp}
    &\{\eta \in L^p(0,T;X): \ 
    &&\partial_t\eta \in L^1(0,T;Z) \}
    &&\hookrightarrow \hookrightarrow L^p(0,T;Y),
    \\
    \label{eq:compact_C}
    &\{\eta \in L^\infty(0,T;X) : \ 
    &&\partial_t\eta \in L^r(0,T;Z) \}
    &&\hookrightarrow \hookrightarrow  C([0,T];Y).
\end{alignat}
Moreover, let $F$ be a bounded subset in $L^p(0,T;X)$ with
\begin{align}
    \label{eq:compact_translation}
    \lim_{\theta\to 0} \norm{\eta(\cdot,\cdot+\theta) - \eta(\cdot,\cdot)}_{L^p(0,T-\theta;Z)} = 0 
    \quad
    \text{ uniformly for } \eta\in F.
\end{align}
Then $F$ is relatively compact in $L^p(0,T;Y)$ if $1\leq p < \infty$ and in $C([0,T];Y)$ if $p=\infty$, respectively.
\end{subequations}

%%%%%%%%%%%%%%
\bigskip

Let us introduce the following notation %\footnote{$a^{\Delta t(,\pm)}$ if it holds for all three.} 
for affine-linear and piecewise constant extensions of time-discrete functions $a^n(\cdot)$, $n=0,...,N_T$:
\begin{alignat}{2}
    \label{def:fun_Delta_t}
    a^{\Delta t}(\cdot, t) 
    &\coloneqq 
    \frac{t - t^{n-1}}{\Delta t} a^n(\cdot)
    + \frac{t^n - t}{\Delta t} a^{n-1}(\cdot)
    \quad\quad 
    && t\in [t^{n-1},t^n], \ n\in \{1,...,N_T\},
    \\
    \label{def:fun_Delta_t_pm}
    a^{\Delta t,+}(\cdot, t) 
    &\coloneqq a^n(\cdot),
    \quad\quad 
    a^{\Delta t,-}(\cdot, t) 
    \coloneqq a^{n-1}(\cdot)
    \quad\quad 
    && t\in (t^{n-1},t^n], \ n\in \{1,...,N_T\}.
\end{alignat}

Using this notation, we can reformulate the system \eqref{eq:phi_FE}--\eqref{eq:sigma_FE} continuously in time. Multiplying each equation in \eqref{eq:phi_FE}--\eqref{eq:sigma_FE} by $\Delta t$ and summing from $n=1,...,N_T$, we obtain for all test functions $(\zeta_h$, $\rho_h$, $ \xi_h) \in (L^2(0,T;\calS_h))^3$ that
\begin{subequations}
\begingroup
\allowdisplaybreaks
\begin{align}
    % Gleichung für phi
    \label{eq:phi_FE_time}
    \int_0^T \int_\Omega \calI_h \Big[ \big(\partial_t \phi_h^{\Delta t}
    - \Gamma_{\phi,h}^{\Delta t,+} \big) \zeta_h \Big]
    + \calI_h[m(\phi_h^{\Delta t,-})] \nabla\mu_h^{\Delta t,+} \cdot \nabla \zeta_h \dx \dt
    = 0,
    \\
    % Gleichung für mu
    \nonumber
    \label{eq:mu_FE_time}
    \int_0^T \int_\Omega \calI_h \Big[ \Big( - \mu_h^{\Delta t,+}  
    + A \psi_1'(\phi_h^{\Delta t,+}) 
    + A \psi_2'(\phi_h^{\Delta t,-}) 
    - \chi_\phi \sigma_h^{\Delta t,+} \Big) \rho_h \Big] \dx\dt
    \quad\quad
    \\
    + \int_0^T \int_\Omega B \nabla\phi_h^{\Delta t,+} \cdot \nabla\rho_h\dx \dt
    = 0,
    \\
    % Gleichung für sigma
    \nonumber
    \label{eq:sigma_FE_time}
    \int_0^T \int_\Omega \calI_h \Big[ \big( \partial_t \sigma_h^{\Delta t}
    + \Gamma_{\sigma,h}^{\Delta t,+} \big) \xi_h \Big]
    + \calI_h[n(\phi_h^{\Delta t,-})] \big(
    \chi_\sigma \nabla\sigma_h^{\Delta t,+} 
    - \chi_\phi \nabla\phi_h^{\Delta t,+} \big)
    \cdot \nabla \xi_h  \dx \dt
    \quad\quad
    \\
    + \int_0^T \int_{\partial\Omega } \calI_h\Big[ K \big(\sigma_h^{\Delta t,+} - \sigma_{\infty,h}^{\Delta t,+} \big) \xi_h \Big] \dH^{d-1} \dt
    = 0,
\end{align}
\endgroup
\end{subequations}
subject to the initial conditions $\phi_h^{\Delta t}(0) = \phi_h^0$, $\sigma_h^{\Delta t}(0) = \sigma_h^0$.

\begin{subequations}
Under the assumptions of Theorem \ref{theorem:existence_FE}, we can deduce from \eqref{eq:energy_FE}, \eqref{eq:bounds_higher_order}, \eqref{eq:bounds_phi_translation}, \eqref{eq:norm_equiv}, \eqref{eq:bounds_initial}, \eqref{eq:norm_equiv_Gamma}, \eqref{def:fun_Delta_t} and 
\eqref{def:fun_Delta_t_pm} that
\begin{align}
\label{eq:bounds_time}
\begin{split}
    %%%%% phi
    &\quad \norm{\phi_h^{\Delta t(,\pm)}}_{L^\infty(0,T;H^1)}^2 
    + \norm{\Delta_h \phi_h^{\Delta t(,\pm)}}_{L^2(0,T;L^2)}^2 
    + \norm{\partial_t \phi_h^{\Delta t} }_{L^2(0,T;(H^1)')}^2
    \\
    %%%%% sigma
    &\quad + \norm{\sigma_h^{\Delta t(,\pm)}}_{L^\infty(0,T;L^2)}^2 
    + \norm{\sigma_h^{\Delta t(,\pm)}}_{L^2(0,T;H^1)}^2 
    + \norm{\sigma_h^{\Delta t(,\pm)}}_{L^2(0,T;L^2({\partial\Omega }))}^2 
    + \norm{\partial_t \sigma_h^{\Delta t} }_{L^2(0,T;(H^1)')}^2
    \\
    %%%%% mu
    &\quad + \norm{\mu_h^{\Delta t,+}}_{L^2(0,T;H^1)}^2
    + \frac{1}{\Delta t} \norm{\phi_h^{\Delta t} - \phi_h^{\Delta t,\pm}}_{L^2(0,T;H^1)}^2
    + \frac{1}{\Delta t} \norm{\sigma_h^{\Delta t} - \sigma_h^{\Delta t,\pm}}_{L^2(0,T;L^2)}^2
    \\
    %%%
    &\leq C,
\end{split}
\end{align}
and for any $l\in\{1,...,N_T\}$,
\begin{align}
    \label{eq:bounds_time_phi_translation}
    \int_0^{T- l \Delta t} 
    \nnorm{\phi_h^{\Delta t(,\pm)}(t+ l \Delta t) 
    - \phi_h^{\Delta t(,\pm)}(t) }_{L^2}^2 \dt
    &\leq C  l\Delta t,
\end{align}
with constants $C>0$ that are independent of $h,\Delta t$.
\end{subequations}

In the following step, we show that there exists a subsequence of $\big(\phi_h^{\Delta t(,\pm)}$, $\mu_h^{\Delta t(,\pm)}$, $\sigma_h^{\Delta t(,\pm)}\big)_{h,\Delta t>0}$ that converges to some limit functions $(\phi,\mu,\sigma)$ as $(h,\Delta t)\to (0,0)$.

\begin{lemma}
\label{lemma:convergence}
Let the assumptions of Theorem \ref{theorem:existence_FE} hold. Then there exist a subsequence of $\big(\phi_h^{\Delta t(,\pm)}$, $\mu_h^{\Delta t(,\pm)}$, $\sigma_h^{\Delta t(,\pm)}\big)_{h,\Delta t>0}$ and functions $\phi, \mu, \sigma$ satisfying
\begin{subequations}
\label{eq:conv_phi_mu_sigma}
\begin{align}
    \phi &\in L^\infty(0,T; H^1) 
    \cap L^2(0,T;H^2)
    \cap H^1(0,T; (H^1)'),
    \\
    \mu &\in L^2(0,T; H^1),
    \\
    \sigma &\in L^\infty(0,T; L^2) \cap L^2(0,T; H^1) \cap H^1(0,T; (H^1)'),
\end{align}
\end{subequations}
with $\phi(0) = \phi_0$ and $\sigma(0) = \sigma_0$ in $L^2(\Omega)$, such that, as $(h,\Delta t)\to (0,0)$,
\begin{subequations}
\begin{alignat}{3}
    \label{eq:conv_phi_Linf_H1}
    \phi_h^{\Delta t(,\pm)} &\to \phi \quad &&\text{ weakly-$*$ } \quad && \text{ in } L^\infty(0,T;H^1),
    \\
    \label{eq:conv_dtphi_L2_H1'}
    \partial_t \phi_h^{\Delta t} &\to \partial_t\phi \quad
    &&\text{ weakly } \quad && \text{ in } L^2(0,T;(H^1)'),
    \\
    \label{eq:conv_phi_L2_H2}
    \Delta_h \phi_h^{\Delta t(,\pm)} &\to \Delta \phi \quad &&\text{ weakly } \quad && \text{ in } L^2(0,T;L^2),
    \\
    \label{eq:conv_phi_L2_W1s}
    \phi_h^{\Delta t(,\pm)} &\to \phi \quad &&\text{ weakly } \quad && \text{ in } L^2(0,T;W^{1,s}),
    \\
    \label{eq:conv_phi_strong}
    \phi_h^{\Delta t(,\pm)} &\to \phi \quad &&\text{ strongly } \quad && \text{ in } L^2(0,T;C^{0,\alpha}(\overline\Omega)),
    \\
    %%%%%%%%%%
    \label{eq:conv_mu_L2_H1}
    \mu_h^{\Delta t,+} &\to \mu \quad &&\text{ weakly } \quad && \text{ in } L^2(0,T;H^1),
    \\
    %%%%%%%%%%
    \label{eq:conv_sigma_Linf_L2}
    \sigma_h^{\Delta t(,\pm)} &\to \sigma \quad &&\text{ weakly-$*$ } \quad && \text{ in } L^\infty(0,T; L^2),
    \\
    \label{eq:conv_sigma_L2_H1}
    \sigma_h^{\Delta t(,\pm)} &\to \sigma \quad &&\text{ weakly } \quad && \text{ in }  L^2(0,T;H^1),
    \\
    \label{eq:conv_dtsigma_L2_H1'}
    \partial_t \sigma_h^{\Delta t} &\to \partial_t\sigma \quad
    &&\text{ weakly } \quad && \text{ in } L^2(0,T;(H^1)'),
    \\
    \label{eq:conv_sigma_strong}
    \sigma_h^{\Delta t(,\pm)} &\to \sigma \quad &&\text{ strongly } \quad && 
    \text{ in } L^2(0,T; L^p),
    %\text{ in } L^2(0,T;C^{0,r}(\overline\Omega)) \text{ if } d=1 \text{ and } L^2(0,T; L^p) \text{ if } d\in\{2,3\}, 
\end{alignat}
where $s\in[2,\infty)$, $\alpha \in [0,1)$, $p\in[1,\infty)$ if $d\in\{1,2\}$, and $s\in[2,6)$, $\alpha\in[0,\frac{1}{2})$, $p\in [1,6)$ if $d=3$, respectively. 
%
%
% NOTE: For $d=1$ we only get $\sigma_h^{\Delta t} \to \sigma$ strongly in $L^2(0,T; C^{0,r}(\overline\Omega)$ with $r\in[0,\frac{1}{2})$. For subsequences of $\sigma_h^{\Delta t,\pm}$, we do not get this as Gagliardo-Nirenberg has no use in this case.
%
%
% where $s\in[2,\frac{2d}{d-2})$, $\alpha\in[0,2-\frac{d}{2})$ and $p\in [1,\frac{2d}{d-2})$, respectively.
\end{subequations}
\end{lemma}

%%%%%%%%%%%%

\begin{proof}
%Zeigen, dass Grenzwertfunktionen von $a^{\Delta t}, a^+, a^-$ übereinstimmen $(a=\phi, a=\sigma)$. Siehe [BNS04].

It follows from \eqref{eq:bounds_time} that
\begin{align}
\label{eq:convergence_1}
    \norm{\phi_h^{\Delta t} - \phi_h^{\Delta t,\pm} }_{L^2(0,T;H^1)}^2
    + \norm{\sigma_h^{\Delta t} - \sigma_h^{\Delta t, \pm}}_{L^2(0,T;L^2)}^2
    \leq C \Delta t \to 0,
\end{align}
as $\Delta t\to 0$. 
%Hence, we have that subsequences of  $\phi_h^{\Delta t}$, $\phi_h^{\Delta t,+}$, $\phi_h^{\Delta t,-}$ and subsequences of $\sigma_h^{\Delta t}$, $\sigma_h^{\Delta t,+}$, $\sigma_h^{\Delta t,-}$ converge to the same limits $\phi$ and $\sigma$, respectively, as $(h,\Delta t)\to (0,0)$. 
Therefore, on noting \eqref{eq:bounds_time} and \eqref{eq:convergence_1}, we can choose a subsequence of $\big(\phi_h^{\Delta t(,\pm)}$, $\mu_h^{\Delta t(,\pm)}$, $\sigma_h^{\Delta t(,\pm)}\big)_{h,\Delta t>0}$ such that there exist limit functions 
\begin{align*}
    \phi &\in L^\infty(0,T; H^1) 
    \cap H^1(0,T; (H^1)'),
    \\
    \mu &\in L^2(0,T; H^1),
    \\
    \sigma &\in L^\infty(0,T; L^2) \cap L^2(0,T; H^1) \cap H^1(0,T; (H^1)'),
\end{align*}
such that the convergence results \eqref{eq:conv_phi_Linf_H1}, \eqref{eq:conv_dtphi_L2_H1'}, \eqref{eq:conv_mu_L2_H1}, \eqref{eq:conv_sigma_Linf_L2}, \eqref{eq:conv_sigma_L2_H1} and \eqref{eq:conv_dtsigma_L2_H1'} hold.

Let $\zeta\in C^1([0,T];H^1)$ with $\zeta(T) = 0$. By integration by parts in time, we obtain
\begin{align*}
    \sskp{\phi_h^0}{\zeta(0)}_{L^2}
    &= \sskp{\phi_h^{\Delta t}(0)}{\zeta(0)}_{L^2}
    =- \int_0^T \sskp{\partial_t \phi_h^{\Delta t}}{\zeta}_{L^2} \dt 
    - \int_0^T \sskp{\phi_h^{\Delta t}}{\partial_t \zeta}_{L^2} \dt.
\end{align*}
On noting \eqref{eq:def_initial} and \eqref{eq:interp_H2}, the term on the left-hand side converges to $\skp{\phi_0}{\zeta(0)}_{L^2}$ as $h \to 0$. It follows from \eqref{eq:conv_dtphi_L2_H1'} and \eqref{eq:conv_phi_Linf_H1} that the terms on the right-hand side converge to 
\begin{align*}
- \int_0^T \dualp{\partial_t\phi}{\zeta} \dt 
-  \int_0^T \sskp{\phi}{\partial_t\zeta}_{L^2} \dt 
= \sskp{\phi(0)}{\zeta(0)}_{L^2},
\end{align*}
as $(h,\Delta t)\to (0,0)$, where $\dualp{\cdot}{\cdot}$ denotes the duality pairing between $H^1(\Omega)$ and its dual space.
The last equality is a consequence of the continuous embedding 
\begin{align*}
    L^2(0,T;H^1) \cap H^1(0,T; (H^1)') \hookrightarrow C([0,T];L^2),
\end{align*}
from, e.g., \cite[Thm. 25.5]{wloka_1987}, and integration by parts in time. 
Hence, the initial conditions for $\phi$ are satisfied. The result for $\sigma$ can be established analogously.

We can deduce from \eqref{eq:bounds_initial}, \eqref{eq:bounds_time} and \eqref{eq:norm_equiv} that
\begin{align}
\label{eq:convergence_2}
    \norm{\Delta_h \phi_h^{\Delta t(,\pm)}}_{L^2(0,T;L^2)} 
    \leq C.
\end{align}
The result \eqref{eq:conv_phi_L2_H2} follows analogously to \cite[Lemma 3.1]{barrett_nurnberg_styles_2004}. Together with elliptic regularity, as $\Omega$ is a convex, polygonal domain, we obtain additionally that $\phi\in L^2(0,T;H^2)$.

We can establish \eqref{eq:conv_sigma_strong} for a subsequence of $\sigma_h^{\Delta t}$ on noting \eqref{eq:conv_dtsigma_L2_H1'}, \eqref{eq:conv_sigma_L2_H1} and \eqref{eq:compact_Lp}, as the embedding $H^1(\Omega) \hookrightarrow \hookrightarrow L^p(\Omega)$ is compact.
Combining this with \eqref{eq:convergence_1}, \eqref{eq:bounds_time} and a Gagliardo-Nirenberg inequality yields the result \eqref{eq:conv_sigma_strong} for a subsequence of $\sigma_h^{\Delta t,\pm}$.
%\footnote{Gilt auch $\sigma^{\Delta t, \pm} \in L^2(C^{0,r})$? Gagliardo-Nirenberg nützt hier ja nichts.}

On extracting a further subsequence, it follows from \eqref{eq:conv_phi_L2_H2} and \eqref{eq:discr_laplace_bound} that \eqref{eq:conv_phi_L2_W1s} holds.
The strong convergence of a subsequence of $\phi_h^{\Delta t}$ to $\phi$ in $L^2(0,T; C^{0,\alpha}(\overline\Omega))$, as $(h,\Delta t)\to (0,0)$, is a consequence of \eqref{eq:conv_dtphi_L2_H1'}, \eqref{eq:conv_phi_L2_W1s} and \eqref{eq:compact_C}, as the embedding $W^{1,s}(\Omega) \hookrightarrow \hookrightarrow C^{0,\alpha}(\overline\Omega)$ is compact. 
%It follows similarly to e.g. \cite[Lemma 3.3]{barrett_nurnberg_styles_2004}, or  \cite[Lemma 7.1]{barrett_nurnberg_2004} that
%\begin{align}
%    \int_0^{T-\theta} 
%    \norm{ \phi_h^{\Delta t,\pm}(t+\theta)
%    - \phi_h^{\Delta t, \pm}(t) }_{L^2}^2 \dt
%    \leq C \theta,
%\end{align}
%for any $\theta\in(0,T)$. 
Moreover, we obtain from \eqref{eq:bounds_time_phi_translation}, \eqref{eq:conv_phi_L2_W1s} and \eqref{eq:compact_translation} that $\phi_h^{\Delta t,\pm} \to \phi$ strongly in $L^2(0,T;C^{0,\alpha}(\overline\Omega))$,
as $(h,\Delta t)\to (0,0)$, as the embedding $W^{1,s}(\Omega) \hookrightarrow \hookrightarrow C^{0,\alpha}(\overline\Omega)$ is compact. This yields \eqref{eq:conv_phi_strong}.
\end{proof}

%%%%%%%%%%%%%%%%%%

For our main result, we will need the following lemma which is a slightly modified version of \cite[Lemma 6.8]{barrett_2018_fene-p}.

\begin{lemma}
Assume that $g\in C^{0,1}(\R^n,\R)$ with Lipschitz constant $L_g$ and $n\in\N$. Then it holds for all $K \in \calT_h$ and $\pmb q_h\in (\calS_h)^n$, that
\begin{align}
\label{eq:interp_Lipschitz}
    \nnorm{\calI_h\big[ g(\pmb q_h) \big] - g(\pmb q_h) }_{L^2(K)}^2
    &\leq C L_g^2 h^2 \Big( \norm{\nabla \pmb q_h}_{L^2(K)}^2 \Big).
\end{align}
\end{lemma}

%%%%%%%%%%%%%%
Now we pass to the limit in the system \eqref{eq:phi_FE}--\eqref{eq:sigma_FE}.

\begin{theorem}[Convergence]
\label{theorem:convergence} 
Let the assumptions of Lemma \ref{lemma:convergence} hold. Additionally, assume that $\Gamma_\phi, \Gamma_\sigma \in C^{0,1}(\R^2,\R)$ with Lipschitz constants $L_{\Gamma_\phi}$ and $L_{\Gamma_\sigma}$, respectively. Then, the functions $\phi,\mu,\sigma$ from Lemma \ref{lemma:convergence} satisfy for all $\zeta,\rho,\xi\in L^2(0,T;H^1)$
\begin{subequations}
\begin{align}
    \label{eq:phi_weak}
    \int_0^T \dualp{\partial_t\phi}{\zeta} \dt
    &= \int_0^T \int_\Omega -m(\phi)\nabla\mu\cdot\nabla\zeta 
    + \Gamma_\phi(\phi,\sigma) \zeta \dx\dt, 
    \\
    \label{eq:mu_weak}
    \int_0^T \int_\Omega \mu\rho \dx \dt
    &= \int_0^T \int_\Omega 
    A \psi^\prime(\phi)\rho 
    + B\nabla\phi\cdot\nabla\rho 
    - \chi_\phi\sigma\rho \dx \dt, 
    \\
    \label{eq:sigma_weak}
    \begin{split}
    \int_0^T \dualp{\partial_t\sigma}{\xi} \dt
    &= \int_0^T \int_\Omega -n(\phi)(\chi_\sigma\nabla\sigma 
    - \chi_\phi\nabla\phi) \cdot\nabla\xi 
    - \Gamma_\sigma(\phi,\sigma)\xi \dx \dt
    \\
    &\quad
    + \int_0^T \int_{\partial\Omega } K(\sigma_\infty - \sigma)\xi \dH^{d-1} \dt,
    \end{split}
\end{align}
\end{subequations}
and $\phi(0)=\phi_0$, $\sigma(0)=\sigma_0$ in $L^2(\Omega)$, where $\dualp{\cdot}{\cdot}$ denotes the duality pairing between $H^1(\Omega)$ and its dual space.
\end{theorem}

%%%%%%%%%%%%%%%

\begin{proof}
Let $\zeta\in C_0^\infty(0,T;H^1)$. We then define $\zeta_h \coloneqq \calI_h^{Cl} \zeta \in C_0^\infty(0,T;\calS_h)$. It holds for the first term in \eqref{eq:phi_FE_time} that
\begin{align*}
    \aaabs{\int_0^T  \sskp{\partial_t \phi_h^{\Delta t}}{\zeta_h}_h 
    - \dualp{\partial_t \phi}{\zeta} \dt }
    &\leq
    \aaabs{\int_0^T \sskp{\phi_h^{\Delta t}}{ \partial_t \zeta_h}_{h} 
    - \sskp{\phi_h^{\Delta t}}{\partial_t \zeta_h}_{L^2} \dt}
    \\
    &\quad
    + \aaabs{\int_0^T \sskp{\partial_t \phi_h^{\Delta t}}{ \zeta_h - \zeta}_{L^2}  \dt}
    \\
    &\quad
    + \aaabs{\int_0^T \dualp{\partial_t \phi_h^{\Delta t} 
    - \partial_t \phi}{\zeta} \dt}.
\end{align*}
The first and the second terms on the right-hand side vanish as $(h,\Delta t)\to (0,0)$ by using \eqref{eq:lump_Sh_Sh}, \eqref{eq:clement_error}, \eqref{eq:clement_conv} and \eqref{eq:bounds_time}. The last term converges to $0$ as $(h,\Delta t)\to (0,0)$ on noting \eqref{eq:conv_dtphi_L2_H1'}.

For the boundary integrals in \eqref{eq:sigma_FE_time}, we have
\begin{align*}
    \aaabs{\int_0^T  \sskp{\sigma_h^{\Delta t,+}}{\zeta_h}_{h,{\partial\Omega }} - \sskp{\sigma}{\zeta}_{L^2(\partial\Omega )}\dt }
    &\leq
    \aaabs{\int_0^T  \sskp{\sigma_h^{\Delta t,+}}{\zeta_h}_{h,{\partial\Omega }}
    -   \sskp{\sigma_h^{\Delta t,+}}{\zeta_h}_{L^2(\partial\Omega )} \dt}
    \\
    &\quad
    + \aaabs{\int_0^T  \sskp{\sigma_h^{\Delta t,+}}{\zeta_h - \zeta}_{L^2(\partial\Omega )}\dt }
    \\
    &\quad
    + \aaabs{ \int_0^T  \sskp{\sigma_h^{\Delta t,+} - \sigma}{\zeta}_{L^2(\partial\Omega )} \dt }
\end{align*}
On noting \eqref{eq:lump_Gamma_Sh_Sh}, \eqref{eq:clement_error} and \eqref{eq:bounds_time}, we obtain that the first term on the right-hand side can be bounded by $C h \norm{\sigma_h^{\Delta t,+}}_{L^2(0,T;H^1)} \norm{\zeta}_{L^2(0,T;H^1)} \to 0$ in the limit $(h,\Delta t)\to (0,0)$. The second and the third term on the right-hand side vanish as $(h,\Delta t)\to (0,0)$ on noting \eqref{eq:clement_Gamma}, \eqref{eq:bounds_time}, \eqref{eq:conv_sigma_L2_H1} and the continuity of the trace operator.

The passage to the limit in the remaining linear terms in \eqref{eq:phi_FE_time}--\eqref{eq:sigma_FE_time} can be established similarly.

%%%%%%%%%%%%%%%%%%%

\bigskip

Now we show convergence of the nonlinear terms. 
By the continuity of $m(\cdot)$, \eqref{eq:conv_phi_strong} and \eqref{eq:interp_continuous}, we have for almost all $t\in(0,T)$ that
\begin{align*}
    \nnorm{ \calI_h\big[ m(\phi_h^{\Delta t,-}) \big] - m(\phi)}_{L^\infty}
    &\leq 
    \nnorm{ \calI_h\big[ m(\phi_h^{\Delta t,-}) - m(\phi)\big]}_{L^\infty}
    + \nnorm{ \calI_h\big[ m(\phi) \big] - m(\phi)}_{L^\infty}
    \\
    &\leq 
    C \nnorm{ m(\phi_h^{\Delta t,-}) - m(\phi)}_{L^\infty}
    + \nnorm{ \calI_h\big[ m(\phi) \big] - m(\phi)}_{L^\infty},
\end{align*}
where the right-hand side converges to zero, as $(h,\Delta t)\to (0,0)$.
Hence, by the boundedness of $m(\cdot)$ and \eqref{eq:clement_conv}, it holds
\begin{alignat*}{2}
    \calI_h\big[m(\phi_h^{\Delta t,-})\big]  \nabla\zeta_h 
    &\to m(\phi) \nabla\zeta
    &&\quad \text{ a.e. in } \Omega_T \text{ as } (h,\Delta t)\to (0,0) ,
    \\
    \aaabs{ \calI_h\big[m(\phi_h^{\Delta t,-})\big]  \nabla\zeta_h} 
    &\leq m_1 \abs{\nabla\zeta_h} 
    &&\quad \text{ a.e. in } \Omega_T, \text{ for all } h,\Delta t>0,
    \\
    m_1 \abs{\nabla\zeta_h} 
    &\to m_1 \abs{\nabla\zeta} 
    &&\quad \text{ strongly in } L^2(\Omega_T) \text{ as } (h,\Delta t)\to (0,0).
\end{alignat*}
Applying the generalised Lebesgue dominated convergence theorem \cite[Chap. 3]{alt_2016} yields
\begin{align*}
    \norm{\calI_h\big[m(\phi_h^{\Delta t,-})\big]  \nabla\zeta_h - m(\phi) \nabla\zeta}_{L^2(\Omega_T)} \to 0,
\end{align*}
as $(h,\Delta t)\to (0,0)$. Together with the weak convergence of $\nabla\mu_h^{\Delta t,+}$ to $\nabla\mu$ in $L^2(0,T;L^2)$, as $(h,\Delta t)\to (0,0)$, we obtain by the product of weak-strong convergence \cite[Chap. 8]{alt_2016}, that
\begin{align*}
    \int_0^T \int_\Omega \calI_h\big[m(\phi_h^{\Delta t,-})\big]  \nabla\zeta_h \cdot \nabla\mu_h^{\Delta t,+} \dx\dt
    \to \int_0^T \int_\Omega m(\phi)  \nabla\zeta \cdot \nabla\mu \dx\dt,
\end{align*}
as $(h,\Delta t)\to (0,0)$.
The terms involving $n(\cdot)$ can be dealt with in a similar fashion.

\bigskip
By the assumption that the source term $\Gamma_\phi(\cdot,\cdot)$ is Lipschitz continuous, we can proceed as follows.
\begin{align}
\begin{split}
\label{eq:theorem_convergence_I_II_III_IV}
    &\aaabs{\int_0^T \ssskp{\Gamma_\phi(\phi_h^{\Delta t,+},\sigma_h^{\Delta t,+})}{\zeta_h}_h
    - \ssskp{\Gamma_\phi(\phi, \sigma)}{\zeta}_{L^2} \dt}
    \\
    &\leq
    \aaabs{\int_0^T \ssskp{\calI_h\big[\Gamma_\phi(\phi_h^{\Delta t,+}, \sigma_h^{\Delta t,+})\big]}{\zeta_h}_h
    - \ssskp{\calI_h\big[\Gamma_\phi(\phi_h^{\Delta t,+}, \sigma_h^{\Delta t,+})\big]}{\zeta_h}_{L^2} \dt}
    \\
    &\quad
    + \aaabs{\int_0^T \ssskp{\calI_h\big[\Gamma_\phi(\phi_h^{\Delta t,+}, \sigma_h^{\Delta t,+})\big] - \Gamma_\phi(\phi_h^{\Delta t,+}, \sigma_h^{\Delta t,+})}{\zeta_h}_{L^2} \dt}
    \\
    &\quad
    + \aaabs{\int_0^T \ssskp{\Gamma_\phi(\phi_h^{\Delta t,+}, \sigma_h^{\Delta t,+}) 
    - \Gamma_\phi(\phi, \sigma) }{\zeta_h}_{L^2}\dt}
    \\
    &\quad
    + \aaabs{\int_0^T \ssskp{\Gamma_\phi(\phi, \sigma) }{\zeta_h - \zeta}_{L^2}\dt}
    \\
    &\eqqcolon I + II + III + IV.
\end{split}
\end{align}

On noting Hölder's inequality, \eqref{eq:lump_Sh_Sh}, \eqref{eq:inverse_estimate}, \eqref{eq:clement_error}, \eqref{eq:norm_equiv} and the growth assumptions on $\Gamma_\phi$, it holds that
\begin{align}
\begin{split}
\label{eq:theorem_convergence_I}
    I &\leq 
    C h \nnorm{\calI_h\big[\Gamma_\phi(\phi_h^{\Delta t,+}, \sigma_h^{\Delta t,+}) \big]}_{L^2(0,T;L^2)} \norm{\zeta}_{L^2(0,T;H^1)}
    \\
    &\leq 
    C h \Big( 1 
    + \norm{ \phi_h^{\Delta t,+} }_{L^2(0,T;L^2)} 
    + \norm{ \sigma_h^{\Delta t,+} }_{L^2(0,T;L^2)} \Big)
    \norm{\zeta}_{L^2(0,T;H^1)}.
\end{split}
\end{align}
Moreover, we receive on noting Hölder's inequality, \eqref{eq:interp_Lipschitz} and \eqref{eq:clement_error} that
\begin{align}
\begin{split}
\label{eq:theorem_convergence_II}
    %%%%
    II &\leq 
    C \nnorm{\calI_h\big[\Gamma_\phi(\phi_h^{\Delta t,+}, \sigma_h^{\Delta t,+})\big] - \Gamma_\phi(\phi_h^{\Delta t,+}, \sigma_h^{\Delta t,+})}_{L^2(0,T;L^2)} \norm{\zeta}_{L^2(0,T;H^1)}
    \\
    %%%
    &\leq 
    C h \Big( 
    \norm{ \nabla\phi_h^{\Delta t,+} }_{L^2(0,T;L^2)} 
    + \norm{ \nabla\sigma_h^{\Delta t,+} }_{L^2(0,T;L^2)} \Big)
    \norm{\zeta}_{L^2(0,T;H^1)}.
\end{split}
\end{align}
Hence, on noting \eqref{eq:bounds_time} we obtain that $I, II\to 0$, as $(h,\Delta t)\to (0,0)$. 
%Moreover, by the growth assumptions on $\Gamma_\phi$, \eqref{eq:conv_phi_strong}, \eqref{eq:conv_sigma_strong}, \eqref{eq:clement_error} and the generalized Lebesgue dominated convergence theorem, we obtain that $III\to 0$, as $(h,\Delta t)\to (0,0)$.
Moreover, by the Lipschitz continuity of $\Gamma_\phi(\cdot,\cdot)$, \eqref{eq:conv_phi_strong}, \eqref{eq:conv_sigma_strong} and \eqref{eq:clement_error}, we obtain that $III\to 0$, as $(h,\Delta t)\to (0,0)$.
Further, it holds that $IV\to 0$, as $(h,\Delta t) \to (0,0)$ by noting \eqref{eq:clement_error}, \eqref{eq:conv_phi_mu_sigma} and the growth assumptions on $\Gamma_\phi$.
This leads to
\begin{align*}
    &\aaabs{\int_0^T \ssskp{\Gamma_\phi(\phi_h^{\Delta t,+},\sigma_h^{\Delta t,+})}{\zeta_h}_h
    - \ssskp{\Gamma_\phi(\phi, \sigma)}{\zeta}_{L^2} \dt}
    \to 0,
\end{align*}
as $(h,\Delta t)\to (0,0)$.

The terms containing $\psi_1'(\cdot)$, $\psi_2'(\cdot)$, $\Gamma_\sigma(\cdot,\cdot)$ can be treated similarly using the Lipschitz continuity and growth assumptions.

Finally, we obtain that $\phi,\mu,\sigma$ form a solution of the system \eqref{eq:phi_weak}--\eqref{eq:sigma_weak} in the required sense.
\end{proof}

%%%%%%%%%%%%%%%%%%%%%%%

\begin{remark}~
\begin{enumerate}
\item 
Assume that the source terms $\Gamma_{\phi}, \Gamma_{\sigma}:\R^2\to \R$ have the specific form \eqref{eq:Gamma_phi_sigma}. 
Then, the passage to the limit in the terms containing $\Gamma_\phi$ and $\Gamma_\sigma$ for $(h,\Delta t) \to (0,0)$ can be established with the following strategy.

First of all, similarly to \eqref{eq:interp_Lipschitz}, one can show the following result for $\Gamma_\phi$ (and similarly for $\Gamma_\sigma$):
\begin{align}
\label{eq:interp_specific_source}
    \nnorm{\calI_h\big[ \Gamma_\phi(\phi_h,\sigma_h) \big] - \Gamma_\phi(\phi_h,\sigma_h) }_{L^2(K)}^2
    &\leq C h^2 \Big( 
    L_h^2 \norm{\sigma_h}_{L^\infty(K)} ^2
    \norm{\nabla \phi_h}_{L^2(K)}^2 
    + \norm{h(\cdot)}_{L^\infty(\R)}^2 \norm{\nabla\sigma_h}_{L^2(K)}^2 \Big),
\end{align}
for all $\phi_h, \sigma_h \in \calS_h$ and all simplices $K\in\calT_h$.

For the terms $I$ and $IV$ in \eqref{eq:theorem_convergence_I_II_III_IV}, one can follow the proof of Theorem \ref{theorem:convergence} in order to show $I,IV \to 0$, as $(h,\Delta t)\to (0,0)$. Further, we can use the specific form of $\Gamma_\phi$ together with \eqref{eq:conv_phi_strong}, \eqref{eq:conv_sigma_strong} and \eqref{eq:clement_error} to obtain $III\to 0$, as $(h,\Delta t)\to (0,0).$
Instead of the calculation in \eqref{eq:theorem_convergence_II}, we proceed as follows to show $II\to 0$, as $(h,\Delta t)\to (0,0)$. On noting Hölder's inequality, \eqref{eq:clement_error}, \eqref{eq:interp_specific_source} and \eqref{eq:bounds_time}, it holds that
\begin{align*}
    II &=
    \aaabs{\int_0^T \ssskp{\calI_h\big[\Gamma_\phi(\phi_h^{\Delta t,+}, \sigma_h^{\Delta t,+})\big] - \Gamma_\phi(\phi_h^{\Delta t,+}, \sigma_h^{\Delta t,+})}{\zeta_h}_{L^2} \dt}
    \\
    %%%
    &\leq 
    C  \nnorm{\calI_h\big[\Gamma_\phi(\phi_h^{\Delta t,+}, \sigma_h^{\Delta t,+})\big] - \Gamma_\phi(\phi_h^{\Delta t,+}, \sigma_h^{\Delta t,+})}_{L^2(0,T;L^2)}
    \norm{\zeta}_{L^2(0,T;H^1)} 
    \\
    %%%
    &\leq C \Big( 
    h \norm{\sigma_h^{\Delta t,+}}_{L^2(0,T;L^\infty)}
    \norm{\nabla\phi_h^{\Delta t,+}}_{L^\infty(0,T;L^2)}
    + h \norm{\nabla\sigma_h^{\Delta t,+}}_{L^2(0,T;L^2)} \Big)
    \norm{\zeta}_{L^2(0,T;H^1)}
    \\
    &\leq C \big( h^{1 - \frac{d}{q}} + h \big) \norm{\zeta}_{L^2(0,T;H^1)},
\end{align*}
where we used
\begin{align*}
    h \norm{\sigma_h^{\Delta t,+}}_{L^2(0,T;L^\infty)}
    &\leq C 
    h^{1 - \frac{d}{q}} \norm{\sigma_h^{\Delta t,+}}_{L^2(0,T;L^q)}
    \leq C 
    h^{1 - \frac{d}{q}} \norm{\sigma_h^{\Delta t,+}}_{L^2(0,T;H^1)},
\end{align*}
for any $q\in(1,\infty)$ if $d=1$ and $q \in (d,\frac{2d}{d-2})$ if $d\in\{2,3\}$ on noting \eqref{eq:inverse_estimate} and the Sobolev embedding $H^1(\Omega) \hookrightarrow L^q(\Omega)$ for $d\in\{1,2,3\}$. 

Hence, we finally obtain that
\begin{align*}
    \aaabs{\int_0^T \sskp{\Gamma_\phi(\phi_h^{\Delta t,+},\sigma_h^{\Delta t,+})}{\zeta_h}_h
    - \sskp{\Gamma_\phi(\phi, \sigma)}{\zeta}_{L^2} \dt}
    \to 0,
\end{align*}
as $(h,\Delta t)\to (0,0)$. The other source term $\Gamma_\sigma$ can be treated similarly.

%%%%%%%%%
\item Let us assume that the mobility functions $m(\cdot),n(\cdot)$ are constant and that the source terms $\Gamma_\phi,\Gamma_\sigma$ have the specific form \eqref{eq:Gamma_phi_sigma}.
Then, it follows from \cite[Thm.~2.2]{garcke_lam_2017} that solutions $(\phi,\mu,\sigma)$ of \eqref{eq:phi_weak}--\eqref{eq:sigma_weak} depend continuously on the initial and boundary data. In particular, solutions of \eqref{eq:phi_weak}--\eqref{eq:sigma_weak} are unique.
This result also holds if $\Gamma_\phi,\Gamma_\sigma$ are Lipschitz continuous and the proof is similar to the proof of Theorem \ref{theorem:FE_cont_dep}.

Moreover, this result can also be obtained from the passage of the limit in Theorem \ref{theorem:FE_cont_dep}, as $(h,\Delta t)\to (0,0)$, supposed that one can pass to the limit on the right-hand side of \eqref{eq:FE_cont_dep}, as $(h,\Delta t)\to (0,0)$.

Hence, under the assumptions of Theorem \ref{theorem:convergence} and, in addition, if the mobility functions are constant, every subsequence of $\big( \phi_h^{\Delta t(,\pm)}, \mu_h^{\Delta t(,\pm)}, \sigma_h^{\Delta t(,\pm)} \big)$ has a further subsequence converging to the same limit $(\phi,\mu,\sigma)$. We then already have that the whole sequence $\big( \phi_h^{\Delta t(,\pm)}, \mu_h^{\Delta t(,\pm)}, \sigma_h^{\Delta t(,\pm)} \big)$ converges to $(\phi,\mu,\sigma)$.
\end{enumerate}
\end{remark}

%% file: results/fem_6_numeric.tex
\section{Numerical results}
\label{sec:numeric}
% In this section, we present some numerical results for the system \eqref{eq:all}. The calculations have been programmed in Python using the finite element software tool FEniCS \cite{fenics_book_2012}.

In this section, we present numerical results for the model \eqref{eq:all}. In particular, we want to illustrate the practicability of the fully-discrete scheme \eqref{eq:phi_FE}--\eqref{eq:sigma_FE} in the space dimensions $d\in\{1,2,3\}$. First, let us introduce the following scheme for which in \eqref{eq:phi_FE}--\eqref{eq:sigma_FE} specific choices of the parameters $A,B$ and the nutrient mobility function $n(\cdot)$ have been applied. 
%In particular, for the dimensions $d\in\{1,2,3\}$, we implement the following fully-discrete scheme in Python using the finite element software tool FEniCS \cite{fenics_book_2012}.
%in the space dimensions $d\in\{1,2,3\}$ which illustrate the practicability of the fully-discrete scheme \eqref{eq:phi_FE}--\eqref{eq:sigma_FE}. 

For given discrete initial data $(\phi_h^{0},\sigma_{h}^{0}) \in (\calS_h)^2$ and for $n=1,...,N_T$, find the discrete solution triplet $(\phi_h^{n}, \mu_h^n, \sigma_{h}^{n}) \in (\calS_h)^3 $ which satisfies for any test function triplet $(\zeta_h, \rho_h, \xi_h) \in (\calS_h)^3 $:
\begin{subequations}
\begin{align}
    \label{eq:phi_FE_1D}
    \int_\Omega \calI_h \Big[ \Big(\frac{\phi_h^n-\phi_h^{n-1}}{\Delta t}
    %- \frac{1}{2} \big( \lambda_p \sigma_h^n - \lambda_a \big) (1+\phi_h^n) 
    %\Big) \zeta_h \Big]
    - \Gamma_{\phi}(\phi_h^n,\sigma_h^n) \Big) \zeta_h \Big]
    %+ \calI_h\Big[ \frac{M}{2} (1+\phi_h^{n-1})^2 + m_0 \Big] \nabla\mu_h^n \cdot \nabla \zeta_h \dx
    + \calI_h\big[ m(\phi_h^{n-1}) \big] \nabla\mu_h^n \cdot \nabla \zeta_h \dx
    = 0,
    \\
    \label{eq:mu_FE_1D}
    \int_\Omega \calI_h \Big[ \Big( \mu_h^n  
    - \frac{\beta}{\epsilon} \psi_1'(\phi_h^n) 
    - \frac{\beta}{\epsilon} \psi_2'(\phi_h^{n-1})
    %- \frac{\beta}{\epsilon} (\phi_h^n)^3
    %+ \frac{\beta}{\epsilon} \phi_h^{n-1}
    + \chi_\phi \sigma_h^n \Big) \rho_h \Big] 
    - \beta\epsilon \nabla\phi_h^n \cdot \nabla\rho_h \dx 
    = 0,
    \\
    \label{eq:sigma_FE_1D}
    \nonumber
    \int_\Omega \calI_h \Big[ \Big(\frac{\sigma_h^n-\sigma_h^{n-1}}{\Delta t}
    %+ \frac{1}{2}\lambda_c \sigma_h^n (1+\phi_h^n) \Big) \xi_h \Big]
    + \Gamma_{\sigma}(\phi_h^n,\sigma_h^n) \Big) \xi_h \Big]
    + \Big( \nabla\sigma_h^n - %\frac{\chi_\phi}{\chi_\sigma} 
    \eta \nabla\phi_h^n \Big)  \cdot \nabla \xi_h  \dx
    \quad\quad
    \\
    + \int_{\partial\Omega} \calI_h\Big[ K \big(\sigma_h^n 
    - \sigma_{\infty,h}^n \big) \xi_h \Big] \dH^{d-1}
    = 0.
\end{align}
\end{subequations}
Here we defined the parameters $A,B$ as $A = \frac{\beta}{\epsilon}, B=\beta \epsilon$, where $\epsilon>0$ is proportional to the width of the diffuse interface, and $\beta>0$ denotes the surface tension.
By the simplifying assumption that diffusion processes of the nutrient are not influenced by the type of tissue, we limit the numerical experiments as in \cite{GarckeLSS_2016} to the case of a constant nutrient mobility function, i.e.
$n(\phi_h^{n-1}) = \chi_\sigma^{-1}.$
Further, in the model \eqref{eq:all}, the effects of chemotaxis (movement of the tumour along the nutrient gradient) and active nutrient transport (active movement of nutrients towards the tumour) are both connected via the parameter $\chi_\phi$. The choice of the nutrient mobility function $n(\cdot)$ was introduced and motivated in \cite{GarckeLSS_2016} in order to decouple these two processes. In particular, the ratio between the parameters $\chi_\phi$ and $\chi_\sigma$, i.e. $\eta\coloneqq \frac{\chi_\phi}{\chi_\sigma},$
in \eqref{eq:sigma_FE_1D} accounts for active nutrient transport while $\chi_\phi$ in \eqref{eq:mu_FE_1D} controls the effects of chemotaxis.

%The system \eqref{eq:phi_FE_1D}--\eqref{eq:sigma_FE_1D} is exactly the numerical scheme \eqref{eq:phi_FE}--\eqref{eq:sigma_FE}, which was analyzed in the previous sections, with 
Throughout all numerical experiments, we make the following choices for the source terms $\Gamma_\phi(\cdot,\cdot), \Gamma_\sigma(\cdot,\cdot)$, the potential $\psi(\cdot)=\psi_1(\cdot) + \psi_2(\cdot)$ and the mobility function $m(\cdot)$ which are defined for all $r,s\in[-2,2]$ by
\begin{alignat}{2}
    \label{eq:num_gamma_phi}
    \Gamma_{\phi}(r, s) &= \frac{1}{2} \big( \lambda_p s - \lambda_a \big) (1+r),
    \\
    \label{eq:num_gamma_sigma}
    \Gamma_{\sigma}(r, s) &= \frac{1}{2}\lambda_c s (1+r) ,
    \\
    \label{eq:num_psi_1'}
    \psi_1(r) &= \frac{1}{4} r^4 
    && \quad\text{with} \quad 
    \psi_1'(r) = r^3,
    \\
    \label{eq:num_psi_2'}
    \psi_2(r) &= - \frac{1}{2} r^2
    && \quad \text{with} \quad 
    \psi_2'(r) = - r,
    \\
    \label{eq:num_mobility}
    m(r) &= \frac{M}{2} (1+r)^2 + m_0 . %5\cdot 10^{-6},
    %
    %n(\phi_h^{n-1}) & = \chi_\sigma^{-1} .
    % \label{eq:num_gamma_phi}
    % \Gamma_{\phi}(\phi_h^n, \sigma_h^n) &= \frac{1}{2} \big( \lambda_p \sigma_h^n - \lambda_a \big) (1+\phi_h^n),
    % \\
    % \label{eq:num_gamma_sigma}
    % \Gamma_{\sigma}(\phi_h^n, \sigma_h^n) &= \frac{1}{2}\lambda_c \sigma_h^n (1+\phi_h^n) ,
    % \\
    % \label{eq:num_psi_1'}
    % \psi_1(\phi_h^n) &= \frac{1}{4}(\phi_h^n)^4 
    % \quad \text{with} \quad 
    % \psi_1'(\phi_h^n) = (\phi_h^n)^3,
    % \\
    % \label{eq:num_psi_2'}
    % \psi_2(\phi_h^{n-1}) &= - \frac{1}{2} (\phi_h^{n-1})^2
    % \quad \text{with} \quad 
    % \psi_2'(\phi_h^{n-1}) = - \phi_h^{n-1},
    % \\
    % \label{eq:num_mobility}
    % m(\phi_h^{n-1}) &= \frac{M}{2} (1+\phi_h^{n-1})^2 + m_0 . %5\cdot 10^{-6},
    % %
    % %n(\phi_h^{n-1}) & = \chi_\sigma^{-1} .
\end{alignat}
For $r,s \in \R \backslash [-2,2]$, we truncate the functions $\Gamma_\phi(\cdot,\cdot), \Gamma_\sigma(\cdot,\cdot)$, $m(\cdot)$ and restrict the potential $\psi(\cdot)=\psi_1(\cdot) + \psi_2(\cdot)$ to quadratic growth such that the assumptions ($A2$)--($A4$) hold true. We remind that the approach of the truncated functions is inevitable for the mathematical analysis of the discrete scheme \eqref{eq:phi_FE}--\eqref{eq:sigma_FE}.
In practice, our numerical experiments indicate that the order parameter $\phi$ and the nutrient $\sigma$ stay within the range $[-2,2]$. In particular, we observe $\phi \in [-1-\delta, 1+\delta]$ and $\sigma \in [0, 1+\delta]$ for a relatively small constant $\delta>0$.

%Moreover, we observe that the nutrient concentration $\sigma$ takes non-negative values supposed that the consumption rate $\lambda_c\geq0$ in \eqref{eq:num_gamma_sigma} is not chosen too large. 
% Therefore, in order to apply the theoretical results, one could introduce modified functions which match with \eqref{eq:num_gamma_phi}--\eqref{eq:num_mobility} for $\phi\in[-1-\delta, 1+\delta]$ and $\sigma\in\R$, and which are truncated for $\abs{\phi} > 1+\delta$ such that the assumptions ($A2$)--($A4$) are fulfilled.

In the biological context, the source terms in \eqref{eq:num_gamma_phi}--\eqref{eq:num_gamma_sigma} model the processes of proliferation, apoptosis and nutrient consumption with the corresponding rates $\lambda_p, \lambda_a, \lambda_c \geq 0$.
We assume these effects to occur only in presence of tumour cells and to vanish in the pure healthy phase where $\phi=-1$.
Also, the potential $\psi$ is chosen of polynomial type such that $\psi=\psi_1 + \psi_2$ with $\psi_1$ convex and $\psi_2$ concave.
Moreover, the choice of $m(\cdot)$ in \eqref{eq:num_mobility} allows for a constant mobility ($M=0$, $m_0>0$) and for a (nearly) one-sided degenerate mobility function ($M>0$, $m_0\approx 0$), where both choices were suggested in \cite{ebenbeck_garcke_nurnberg_2020}. 
%By the simplifying assumption that diffusion processes of the nutrient are not influenced by the type of tissue, we limit the numerical experiments as in \cite{GarckeLSS_2016} to the case of a constant nutrient mobility $n(\cdot)$. Further, in the system \eqref{eq:all}, the effects of chemotaxis and active nutrient transport are both connected via the parameter $\chi_\phi$. The choice of the nutrient mobility function $n(\cdot)$ in \eqref{eq:num_mobility} was introduced and motivated in \cite{GarckeLSS_2016} in order to decouple these two processes. In particular, the ratio between the parameters $\chi_\phi$ and $\chi_\sigma$ (i.e.~$\eta\coloneqq \frac{\chi_\phi}{\chi_\sigma}$) in \eqref{eq:sigma_FE_1D} accounts for active nutrient transport while $\chi_\phi$ in \eqref{eq:mu_FE_1D} controls the effects of chemotaxis.

Let us now explain some implementation aspects for the system \eqref{eq:phi_FE_1D}--\eqref{eq:sigma_FE_1D}. All calculations have been performed in Python using the finite element software tool FEniCS \cite{fenics_book_2012}. In order to solve the nonlinear system \eqref{eq:phi_FE_1D}--\eqref{eq:sigma_FE_1D}, it is linearized with the Newton method and the resulting linear systems are solved with the PETSc-built in sparse LU solver which is provided by FEniCS.

%\subsection{\colorbox{Gainsboro}{Numerical example and convergence rates in 1D}}
\bigskip

At first, we want to investigate the numerical errors of discrete solutions in one spatial dimension on the fixed interval $\Omega = (0,1)\subset \R$. %in order to verify the convergence property of discrete solutions of \eqref{eq:phi_FE_1D}--\eqref{eq:sigma_FE_1D} numerically. 
The strategy is similar to %the work of Blowey and Elliott
\cite{blowey_elliott_1992}, where the authors verified numerical convergence rates for the classical Cahn--Hilliard equation in one spatial dimension. 
We proceed as follows. We calculate the numerical solutions $\phi_h^{\Delta t(,\pm)}$, $\mu_h^{\Delta t(,\pm)}$ and $\sigma_h^{\Delta t(,\pm)}$ on a fixed time interval $(0,T)$ for some different values of $h$ and $\Delta t$, where $T=0.1$, $\Delta t = h^2$ and $h \in \{ \frac{1}{32}, \frac{1}{64}, \frac{1}{128}, \frac{1}{256} \}$. 
Then, a comparison is made with a reference solution $\big(\phi^*, \mu^*, \sigma^*\big)$. Due to the lack of knowledge of exact solutions of \eqref{eq:phi}-\eqref{eq:sigma}, the reference solution is approximated by a discrete solution obtained on a fine mesh with mesh and time sizes $h_* = \frac{1}{1024}$ and $\Delta t_* = h_*^2$, respectively.

Moreover, the initial conditions are constructed as follows. Some given functions $\tilde\phi_0$ and $\tilde\sigma_0$ are interpolated as initial data for some numerical solutions $\tilde\phi$ and $\tilde\sigma$, where the discrete setup is the same as for the reference solutions. Then the nodal interpolations of $\tilde\phi$ and $\tilde\sigma$ at time $\tilde T=0.01$ are taken as the initial data for the error tests, i.e.~$\phi_h^0 = \calI_h \tilde\phi(\tilde T)$, $\sigma_h^0 = \calI_h \tilde\sigma(\tilde T)$, and $\phi^*(0) = \tilde\phi(\tilde T)$, $\sigma^*(0) = \tilde\sigma(\tilde T)$.

We choose the functions $\tilde\phi_0$ and $\tilde\sigma_0$ as
\begin{align*}
    \tilde\phi_0(x) &= - \tanh \Big(\frac{r(x)-0.2}{\sqrt{2}\epsilon}  \Big), \quad r(x)=\abs{x-0.5}, \quad x\in\Omega,
    \\
    \tilde\sigma_0(x) &= 1, \quad x\in\Omega,
\end{align*}
and the model parameters 
\begin{equation}
\label{eq:1d_parameter1}
\begin{aligned}
    &\beta=0.1,
    \quad\quad  &&\epsilon=0.02, 
    \quad\quad  &&\chi_\phi=1,
    \quad\quad  &&\eta=0.02, 
    %\quad\quad  &&\chi_\sigma = 50 ,
    \quad\quad  &&\lambda_p=0,
    \quad\quad  &&\lambda_a=5,
    \\
    &\lambda_c=2,
    \quad\quad  && \sigma_{\infty,h}^n =1 , 
    \quad\quad  && K= 1,
    %\quad\quad  && n_0=1 ,
    \quad\quad  && M = 0, 
    \quad\quad  && m_0 = 1 .
\end{aligned}
\end{equation}
%Further, we replace $\psi_2'(\phi_h^{n-1})$ by $\psi_2'(\phi_h^n)$, and choose a constant mobility function, i.e.~$m(\phi_h^{n-1}) = 1$.

The reference solution at times $t\in\{0, \ 0.04768, \ 0.1\}$ is visualized in Figure \ref{fig:1d_reference} and 
the numerical errors in several norms and the associated experimental orders of convergence (EOC) are displayed in Table \ref{tab:error1_Robin}.
Our results agree with the results which have been obtained for other phase-field systems in the literature, see, e.g., \cite{blowey_elliott_1992, elliott_second_order}.
%and for some variants of the system \eqref{eq:phi_FE_1D}--\eqref{eq:sigma_FE_1D}, see \cite{trautwein_MA} for details.

\begin{figure}[H]
    \centering
	\includegraphics[width=0.32\textwidth]{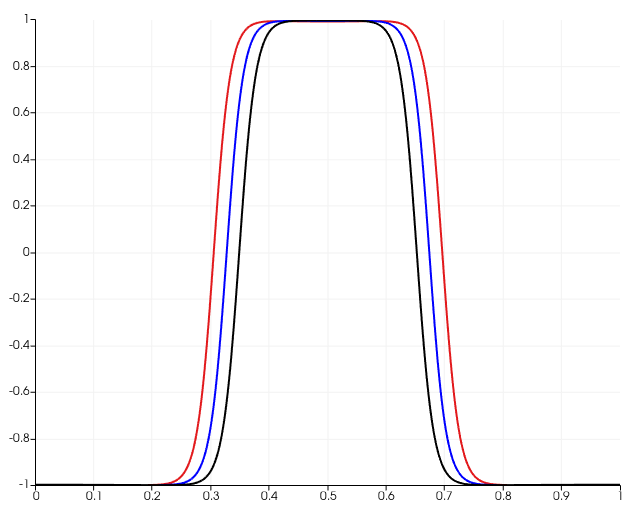}
	\includegraphics[width=0.32\textwidth]{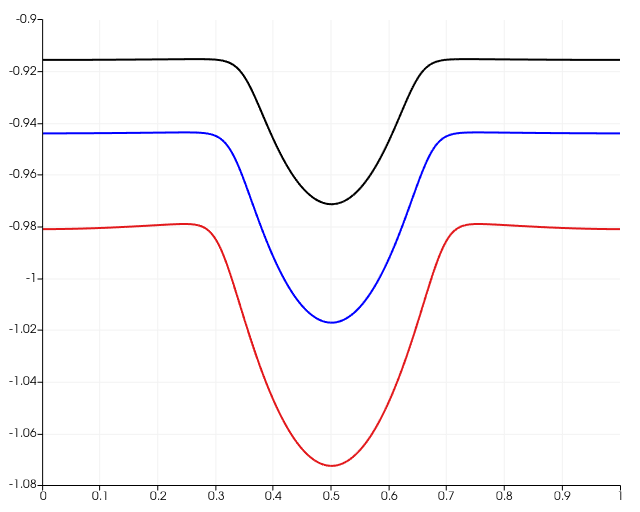}
	\includegraphics[width=0.32\textwidth]{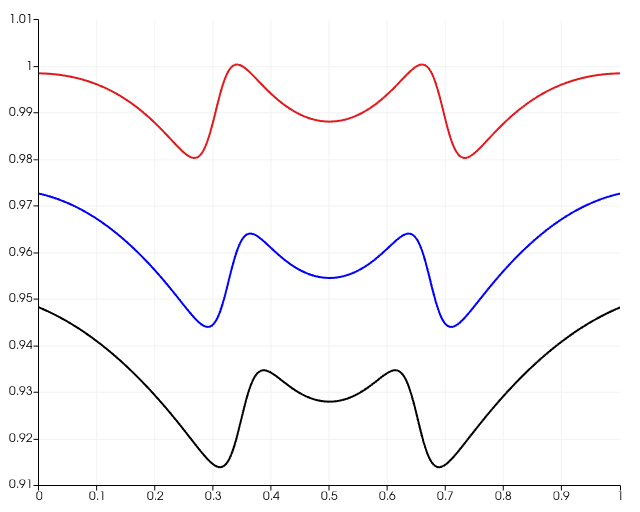}
    \caption{Reference solution $\phi^*$ (left), $\mu^*$ (center) and $\sigma^*$ (right) at times $t=0$ (red), $t=0.04768$ (blue) and $t=0.1$ (black).} 
    \label{fig:1d_reference}
\end{figure}

\begin{table}[H]
\centering
\begin{subtable}{1.0\textwidth}
    \centering
    \begin{tabular}{l|l|l|l|}
    $h$ & $\norm{\phi_h^{\Delta t,+} - \phi^*}_{L^\infty(0,T; L^2)}$ 
    & $\norm{\phi_h^{\Delta t,+} - \phi^*}_{L^2(0,T; L^2)} $
    & $\norm{\nabla\phi_h^{\Delta t,+} - \nabla\phi^*}_{L^2(0,T; L^2)}$
    \\ \hline
    1/32 & 0.08163562460405772 & 0.01800710854612 & 1.2066061901421978 \\
    1/64 & 0.006693917025388757 & 0.0015120246761391146 & 0.4484730959466349 \\
    1/128 & 0.0016810945887728536 & 0.0003596862687547078 & 0.218637109333781 \\
    1/256 & 0.0004050560991618219 & 8.65161998871061e-05 & 0.10622248307490406 \\ \hline
    EOC & 2.053207265533753 & 2.0556968814724033 & 1.0414491311766247 \\
    \end{tabular}
    \caption{Numerical errors and EOC for the phase field variable.}
\end{subtable}% <---- don't forget this %

\begin{subtable}{1.0\textwidth}
    \centering
    \begin{tabular}{l|l|l|l|}
    $h$ & $\norm{\mu_h^{\Delta t,+} - \mu^*}_{L^\infty(0,T; L^2)}$ 
    & $\norm{\mu_h^{\Delta t,+} - \mu^*}_{L^2(0,T; L^2)} $
    & $\norm{\nabla\mu_h^{\Delta t,+} - \nabla\mu^*}_{L^2(0,T; L^2)}$
    \\ \hline
    1/32  & 0.3289157935738554 & 0.07020245096950607 & 0.0712371886033935 \\
    1/64  & 0.015083096865408247 & 0.0041559773789519834 & 0.007677948652658751 \\
    1/128  & 0.003485687033855135 & 0.0009995787982063913 & 0.0033683163873372223 \\
    1/256  & 0.0013677932457678628 & 0.00023695351944984542 & 0.0015255678722707706 \\ \hline
    EOC & 1.3495928709653409 & 2.076716211785033 & 1.1426812910887603 \\
    \end{tabular}
    \caption{Numerical errors and EOC for the chemical potential.}
\end{subtable}% <---- don't forget this %

\begin{subtable}{1.0\textwidth}
    \centering
    \begin{tabular}{l|l|l|l|}
    $h$ & $\norm{\sigma_h^{\Delta t,+} - \sigma^*}_{L^\infty(0,T; L^2)}$ 
    & $\norm{\sigma_h^{\Delta t,+} - \sigma^*}_{L^2(0,T; L^2)} $
    & $\norm{\nabla\sigma_h^{\Delta t,+} - \nabla\sigma^*}_{L^2(0,T; L^2)}$
    \\ \hline
    1/32 &  0.0014306201003786636 & 0.0003185502409487455 & 0.023247581874250266 \\
    1/64  & 0.00012800173351707603 & 2.9889592098977667e-05 & 0.008618697011299468 \\
    1/128  & 3.2174915725953576e-05 & 7.129324028437376e-06 & 0.004199016945370087 \\
    1/256  & 7.754159898532114e-06 & 1.7138863454008837e-06 & 0.0020397507710109924 \\ \hline
    EOC & 2.0528939793534287 & 2.056493851277305 & 1.0416587244340068 \\
    \end{tabular}
    \caption{Numerical errors and EOC for the nutrient.}
\end{subtable}
\caption{Error investigation in one spatial dimension.} \label{tab:error1_Robin}
\end{table}

%%%%%%%%%%%%%%%%%%
%%%%%%%%%%%%%%%%%%

\bigskip

%\subsection{\colorbox{Gainsboro}{Long-time simulation in 2D}}
In the following, we present the results of a long-time simulation in two space dimensions which is motivated by the numerical examples from \cite{GarckeLSS_2016}. 
%First, we recall the setting in two dimensions.
% Let us consider the system \eqref{eq:phi_FE_1D}-\eqref{eq:sigma_FE_1D} with $\Gamma_{\phi}(\phi_h^n,\sigma_h^n)$ and $\Gamma_{\sigma}(\phi_h^n,\sigma_h^n)$ replaced by $\Gamma_{\phi}(\phi_h^{n-1},\sigma_h^n)$ and $\Gamma_{\sigma}(\phi_h^{n-1},\sigma_h^n)$, respectively.\footnote{Rerun 2D simulation with implicit source terms.}
The following set of parameters is used.
\begin{equation}
\label{eq:2d_parameter}
\begin{aligned}
    &\Omega=(-12.5,12.5)^2,  
    \quad\quad  &&\tau=10^{-3},
    \quad\quad  &&\epsilon=0.01, 
    \quad\quad  &&\beta=0.1, 
    \quad\quad  &&\chi_\phi=5,
    \\
    % &\chi_\sigma = \frac{\chi_\phi}{0.04} , % =125
    &\eta = 0.04, 
    \quad\quad  &&\lambda_p=0.5, 
    \quad\quad  &&\lambda_a=0, 
    \quad\quad  &&\lambda_c=1,
    \quad\quad  && \sigma_{\infty,h}^n =1, 
    \\
    &K=1000,
    %\quad\quad  && n_0=1 ,
    \quad\quad  && M = 1 , 
    \quad\quad  && m_0 = 5\cdot 10^{-6} .
\end{aligned}
\end{equation}

In practice, the computations on the domain $\Omega = (-12.5, 12.5)^2 \subset\R^2$ have been performed only on the upper right square $\Omega=(0, 12.5)^2$ due to symmetry reasons. For $\sigma$, homogenous Neumann boundary conditions are used on
\begin{align*}
    \Big(\{0\}\times[0,12.5]\Big) 
    \cup \Big([0,12.5]\times\{0\}\Big),
\end{align*}
and Robin boundary conditions on 
\begin{align*}
    \Big(\{12.5\}\times[0,12.5]\Big)
    \cup \Big([0,12.5]\times\{12.5\}\Big).
\end{align*}

%%%%%%%%%%%%%%

As initial data we start with a slightly perturbed sphere for the tumour, see Figure \ref{fig:2d_phi0}. Besides, the nutrient is assumed to be unconsumed in the beginning. In particular, we set
\begin{align}
    &\phi_0(x) = 
    - \tanh \Big(\frac{r(x)}{\sqrt{2}\epsilon} \Big) ,
    \\
    & \sigma_0(x) = 1,
\end{align}
where
\begin{align*}
    r(x) = \abs{x} - (2 + 0.1\cos(2\theta)), %\colorbox{Gainsboro}{0.1}\cos(2\theta)),
\end{align*}
where $x = \abs{x} ( \cos(\theta),\sin(\theta) )^T$.

\begin{figure}[ht]
    \centering
	\includegraphics[width=0.35\textwidth,trim={10cm 1cm 10cm 5cm},clip]{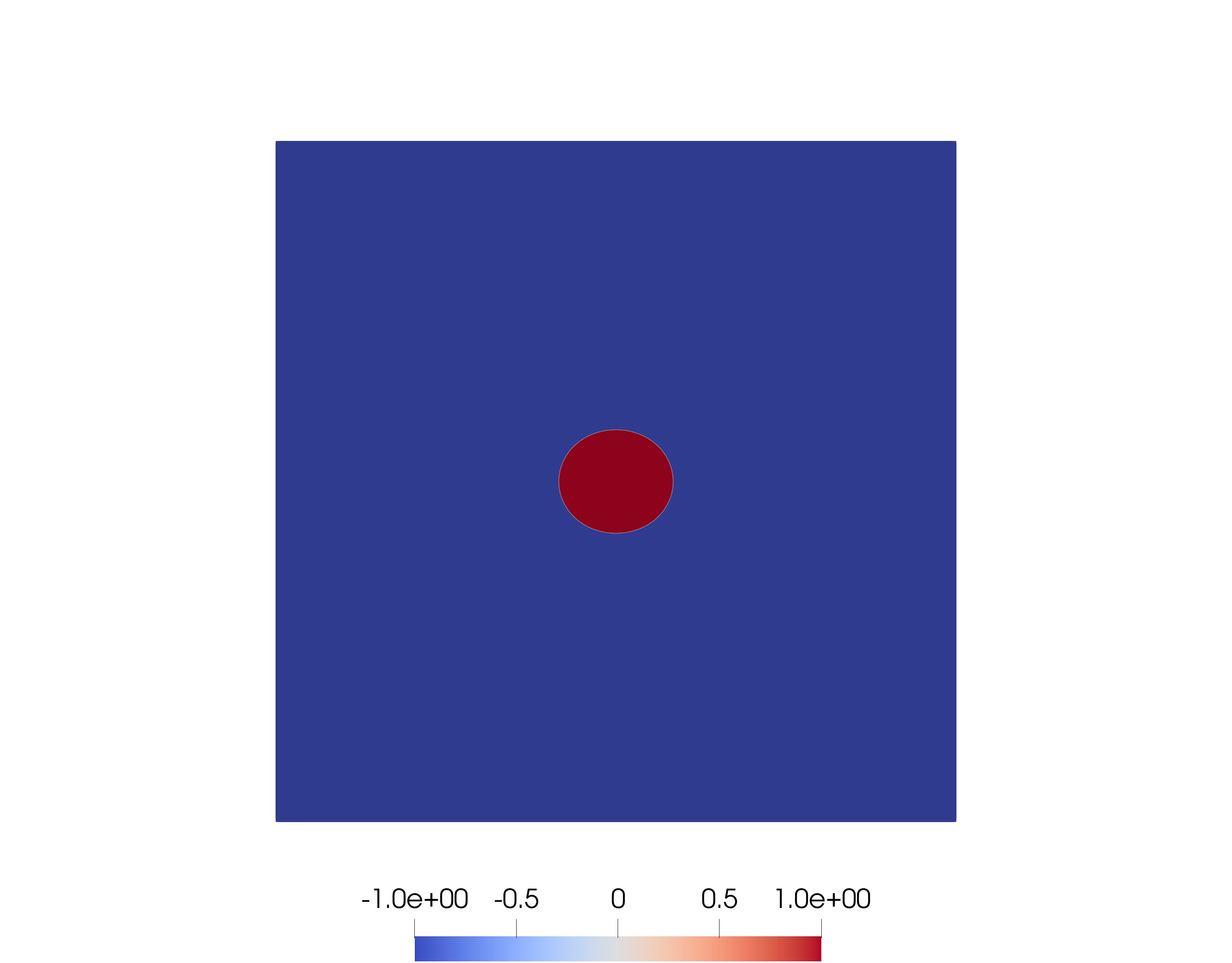}
    \caption{Initial tumour size in two dimensions: A slightly perturbed sphere.} 
    \label{fig:2d_phi0}
\end{figure}

%%%%%%%%%%%

We use a mesh refinement strategy which is similar to the one in \cite{GarckeLSS_2016}. Since the interfacial thickness is assumed to be proportional to $\epsilon$, in order to resolve the interfacial layer we need to choose $h$ such that there lie enough spatial mesh points on the interface. Far away from the interface, the local mesh size can be chosen larger and hence adaptivity in space can heavily speed up computations. 
For the simulations, a mesh with maximal diameter $h_{\max} = 12.5\cdot 2^{-6} \approx 0.1953$ and minimal diameter $h_{\min} = 12.5 \cdot 2^{-10} \approx 0.0122$ are used.
For more details regarding the mesh refinement strategy, see \cite{trautwein_MA}.
%\footnote{For details on the concrete implementation of the mesh refinement. Also explain how the nonlinear scheme is solved -> Newton for linearization, sparse LU for linear system (direct solver works quite nice as we always have between 5'000 and 20'000 spatial DOFS throughout all simulations in 2D and 3D).}

%%%%%%%%%%%%%%%%%%%%%%%%%%%%%

In Figure \ref{fig:2d_proliferation_2}, we display $\phi$ (top row) and $\sigma$ (bottom row) at times $t=8, 15, 22$. 
One can clearly see that after some time, the tumour develops fingers towards regions with higher concentration of the nutrient which allows for better access to the nutrient. 
This effect can be seen as the chemotactic response of the tumour to the lack of nutrients. % and agrees with the numerical examples of \cite{GarckeLSS_2016}.

\begin{figure}[H]
    \centering
    \subfloat
	{\includegraphics[width=0.32\textwidth,trim={10cm 2cm 10cm 0},clip]{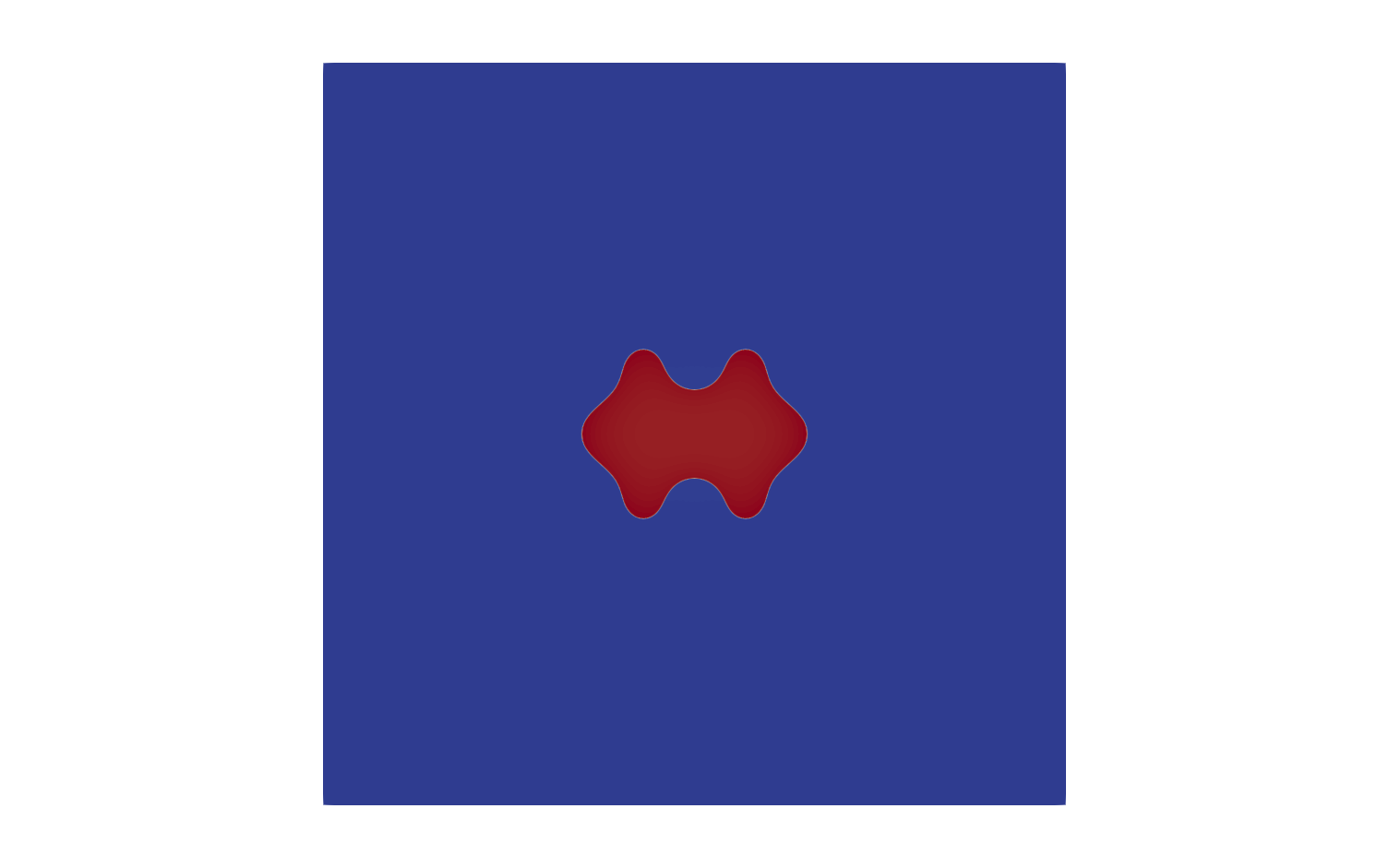}}
	%\hspace{-1em}
	\subfloat
	{\includegraphics[width=0.32\textwidth,trim={10cm 2cm 10cm 0},clip]{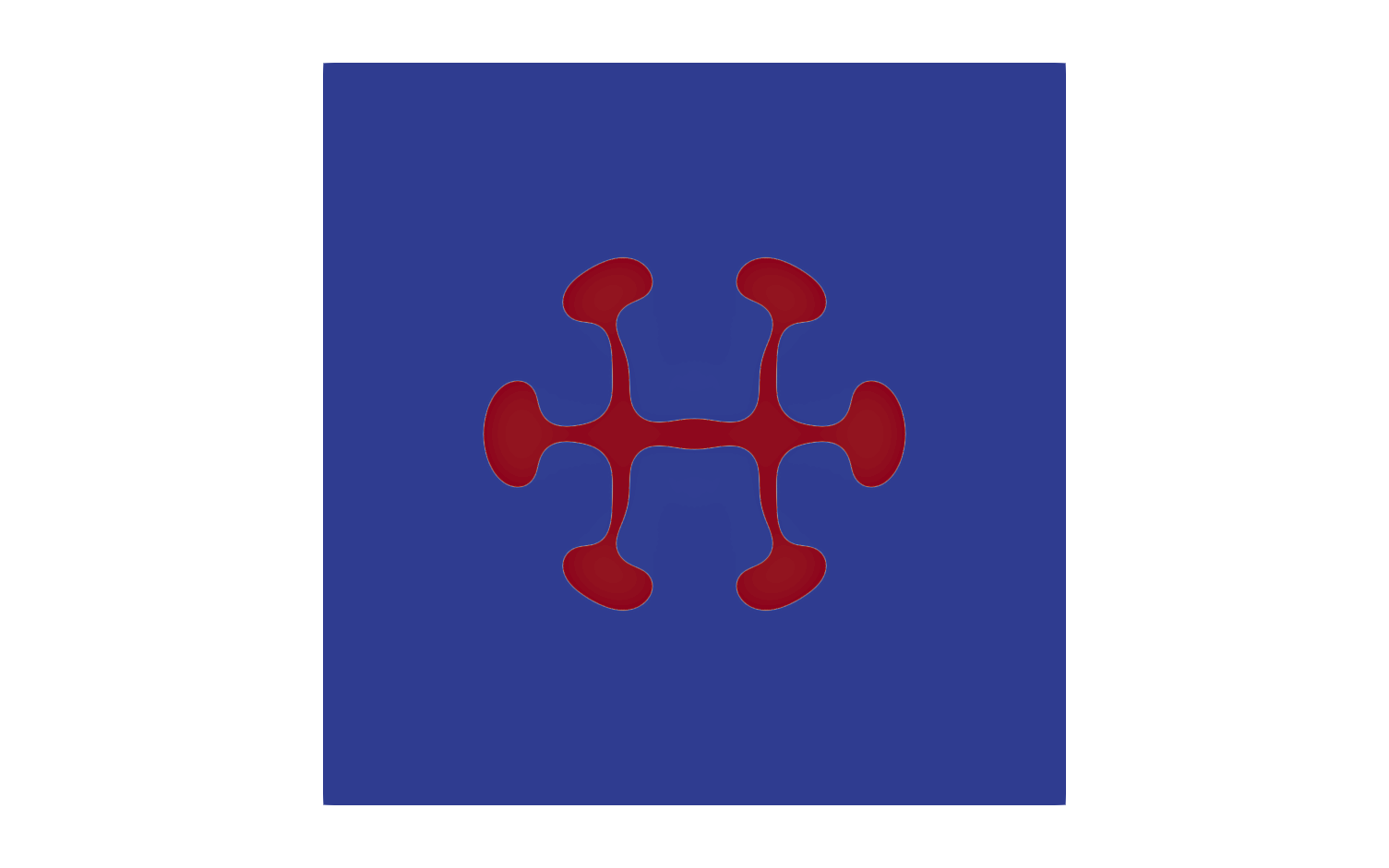}}
	%\hspace{-1em}
	\subfloat
	{\includegraphics[width=0.32\textwidth,trim={10cm 2cm 10cm 0},clip]{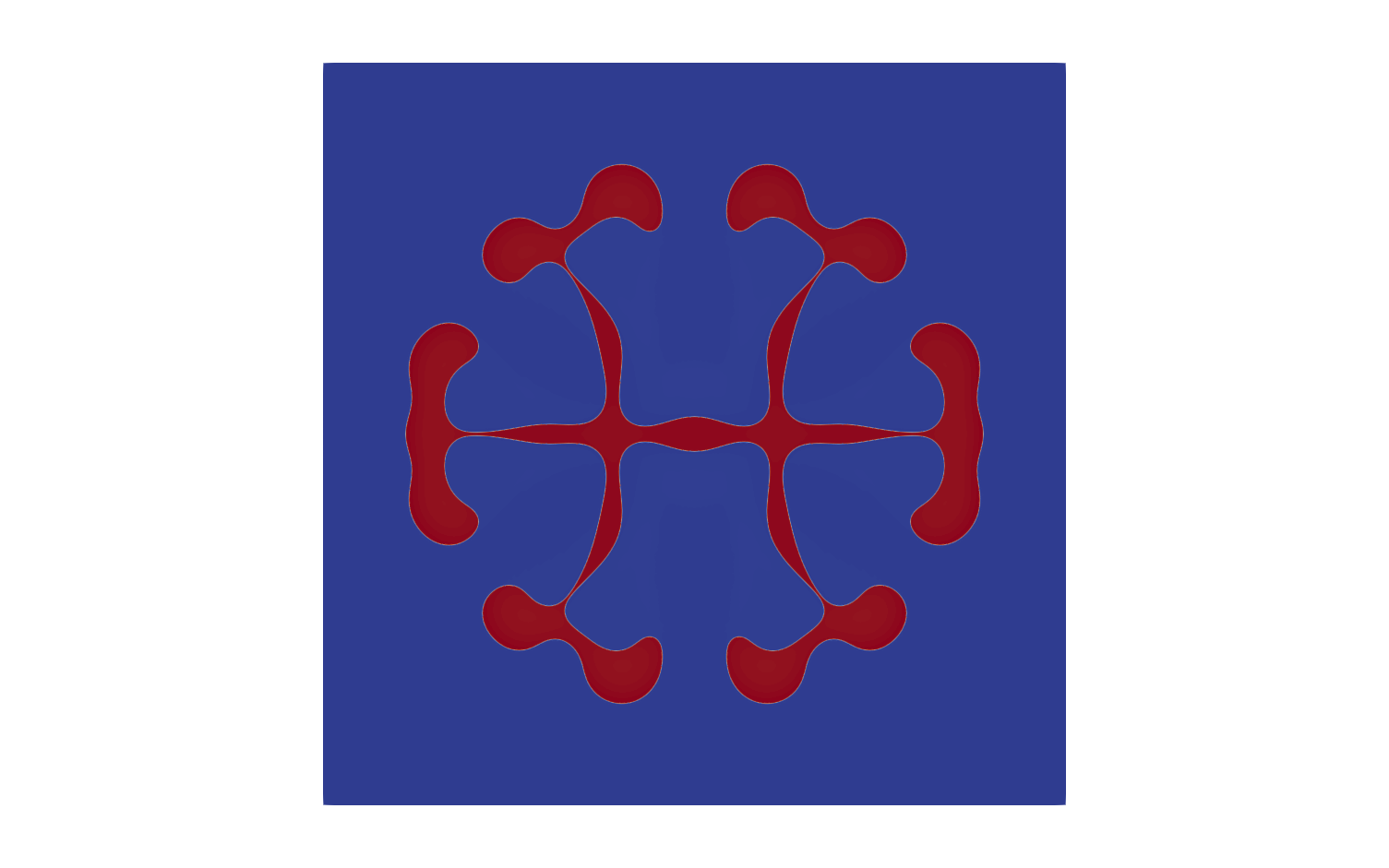}}
	\\%[-2.7ex]
	\subfloat
	{\includegraphics[width=0.32\textwidth,trim={10cm 0cm 10cm 0},clip]{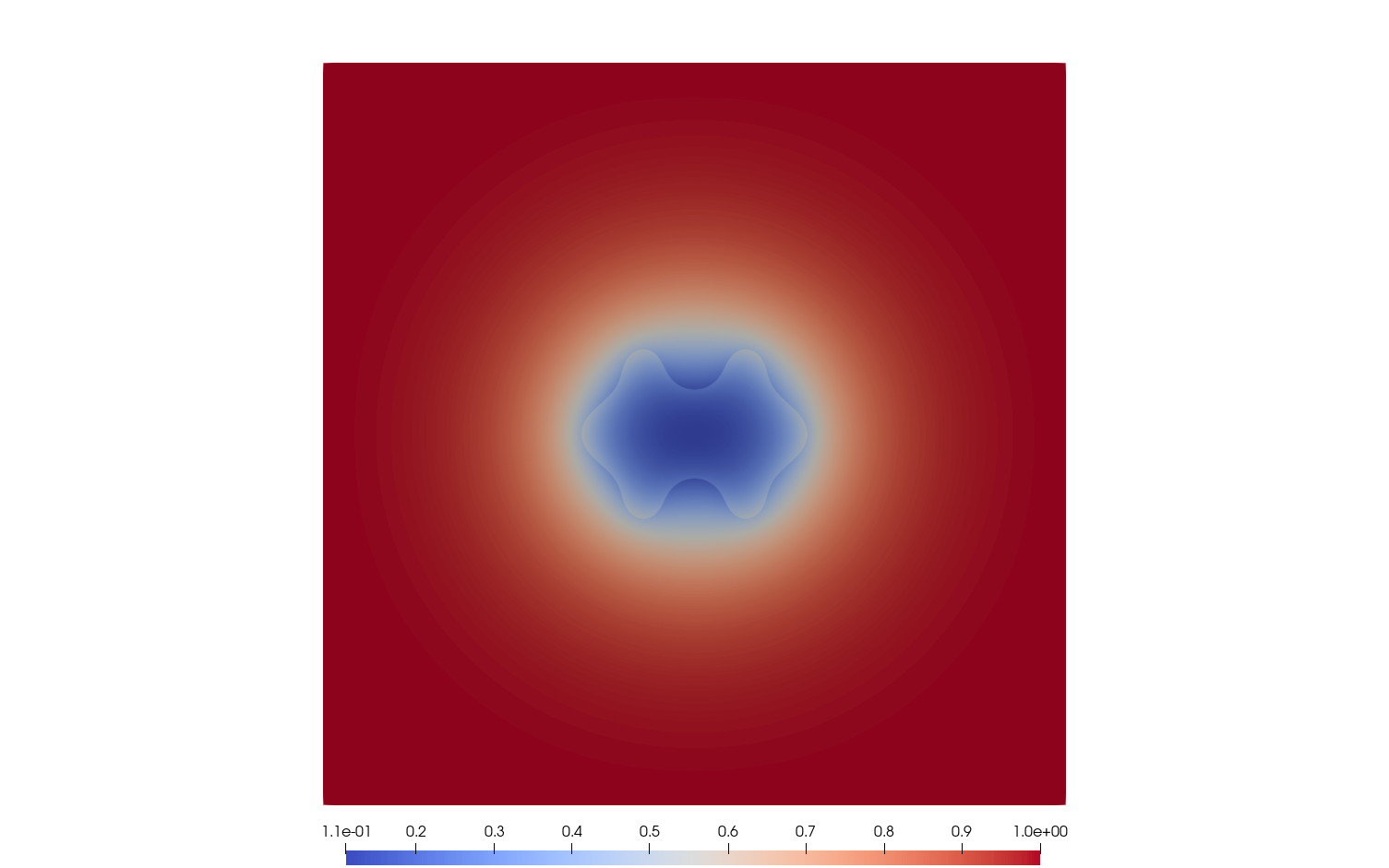}}
	%\hspace{-1em}
	\subfloat
	{\includegraphics[width=0.32\textwidth,trim={10cm 0cm 10cm 0},clip]{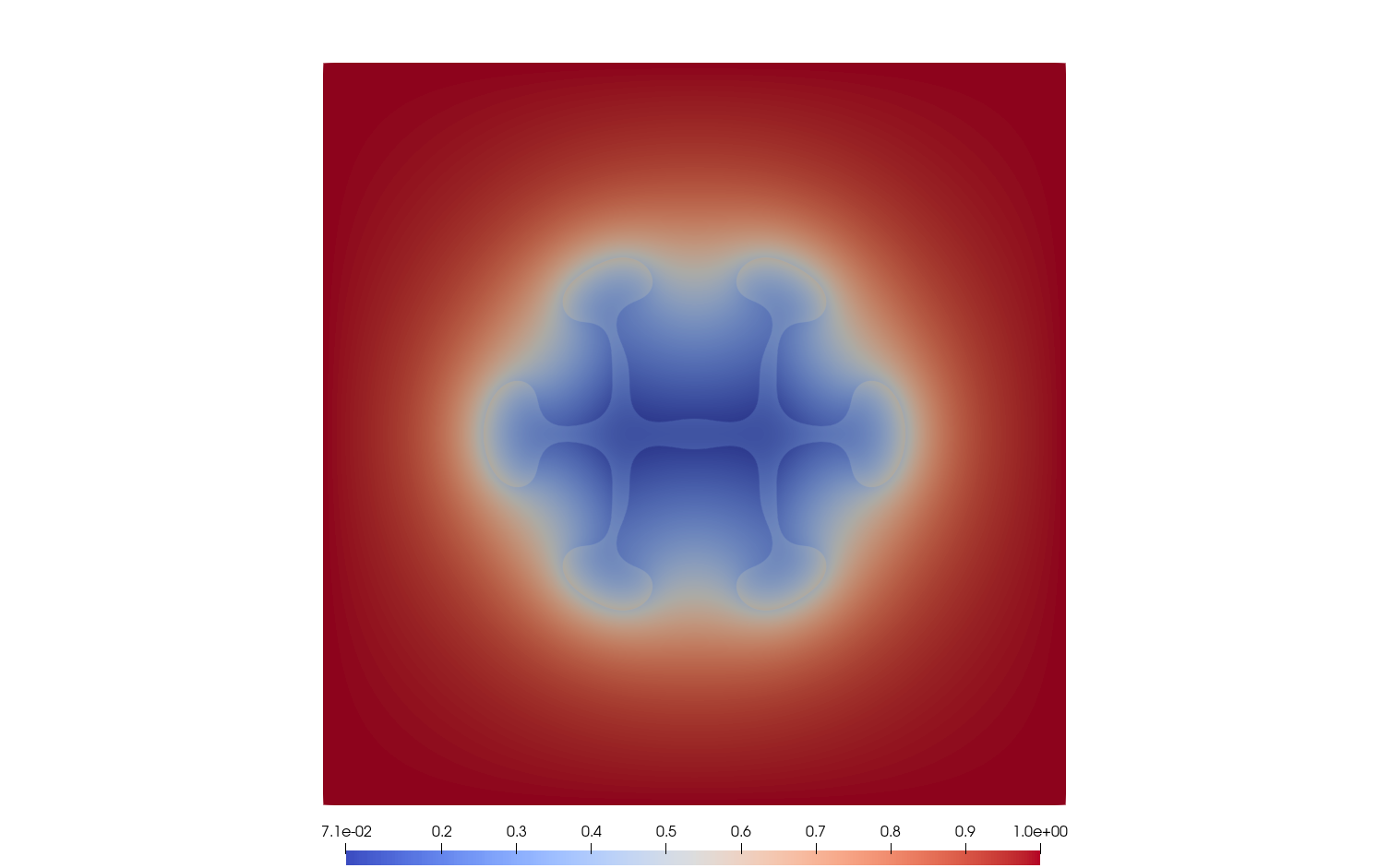}}
	%\hspace{-1em}
	\subfloat
	{\includegraphics[width=0.32\textwidth,trim={10cm 0cm 10cm 0},clip]{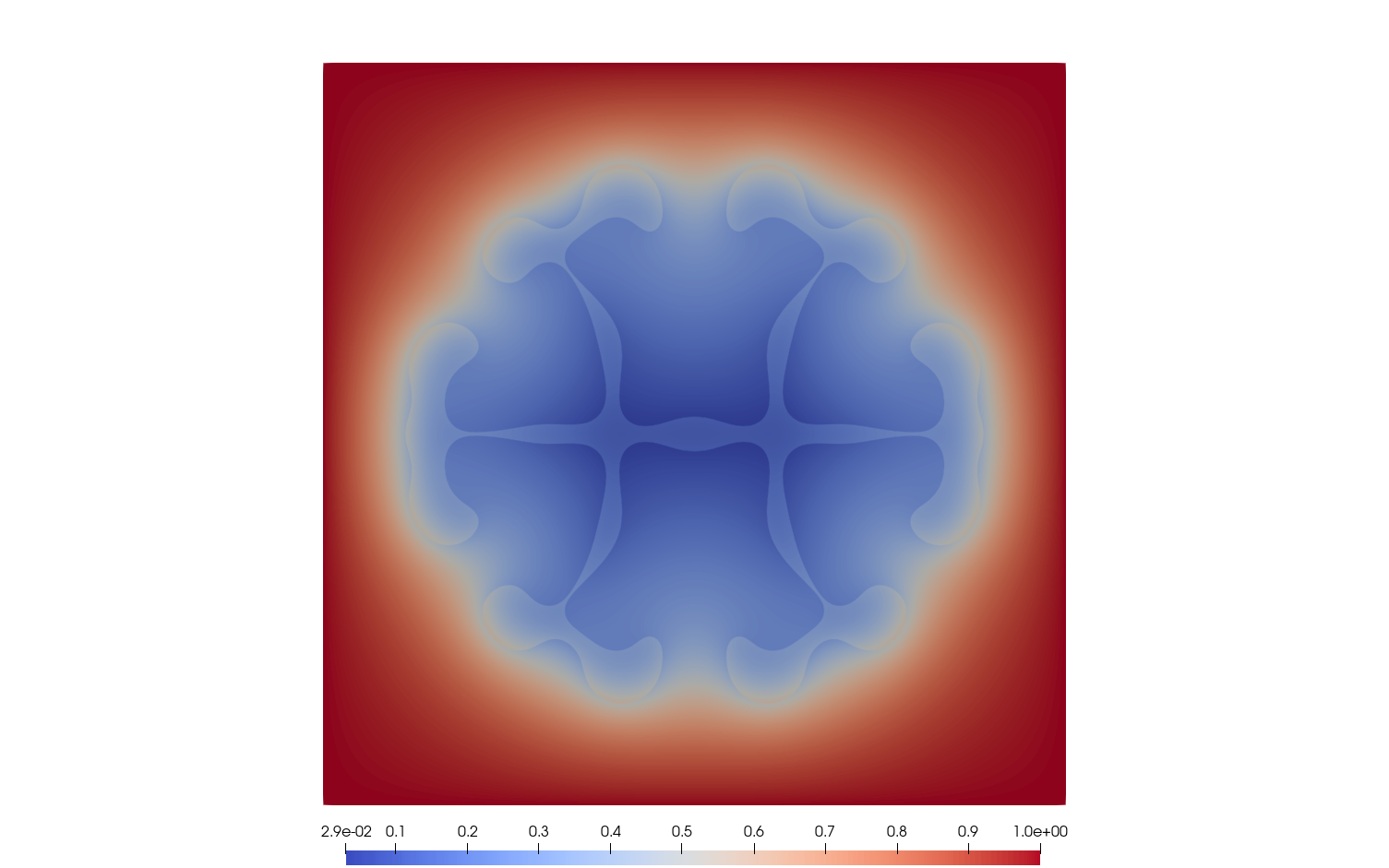}}
    \caption{Numerical solution in two dimensions at times $t=8, 15, 22$.} \label{fig:2d_proliferation_2}
\end{figure}

%%%%%%%%%%%%%%%%%%%%%%%%%%%%%
%%%%%%%%%%%%%%%%%%%%%%

%\subsection{\colorbox{Gainsboro}{Numerical tests in 3D}}

% For the numerical results in three dimensions, we consider the system \eqref{eq:phi_FE_1D}-\eqref{eq:sigma_FE_1D} with $\Gamma_\phi(\phi_h^n,\sigma_h^n)$ and $\Gamma_\sigma(\phi_h^n,\sigma_h^n)$ replaced by $\Gamma_\phi(\phi_h^{n-1},\sigma_h^n)$ and $\Gamma_\sigma(\phi_h^{n-1},\sigma_h^n)$, respectively.\footnote{Rerun 3D simulation with implicit source terms, dynamic nutrient and Robin boundary conditions.} 
% Moreover, we prescribe the Dirichlet boundary conditions $\sigma|_\Gamma=\sigma_\infty$ instead of Robin boundary conditions. Motivated by a non-dimensionalisation argument, see, e.g., \cite{ebenbeck_garcke_nurnberg_2020, garcke_lam_2017, GarckeLSS_2016}, we consider a \textit{quasistatic} nutrient equation, i.e.~the difference quotient $\frac{\sigma_h^n - \sigma_h^{n-1}}{\Delta t}$ in \eqref{eq:sigma_FE_1D} is neglected.

\bigskip

For the numerical results in three dimensions, we use the following set of parameters:
\begin{equation}
\label{eq:3d_parameter}
\begin{aligned}
    &\Omega=(-3,3)^3,  
    \quad\quad  &&\tau=10^{-3},
    \quad\quad  &&\epsilon=0.02, 
    \quad\quad  &&\beta=0.1, 
    \quad\quad  &&\chi_\phi \in\{15,30\},
    \\
    % &\chi_\sigma = \frac{\chi_\phi}{0.02} ,
    &\eta = 0.02,
    \quad\quad  &&\lambda_p=0.5,  %0.1?
    \quad\quad  &&\lambda_a=0, 
    \quad\quad  &&\lambda_c=2,
    \quad\quad  && \sigma_{\infty,h}^n =1, 
    \\
    &K=1000,
    %\quad\quad  && n_0=1 ,
    \quad\quad  && M = 1, 
    \quad\quad  && m_0 = 5\cdot 10^{-6}.
\end{aligned}
\end{equation}

Similarly to the two dimensional case, the computations on the domain $\Omega = (-3, 3)^3 \subset\R^3$ have been performed only on the subdomain $\Omega=(0, 3)^3$ because of symmetry reasons. For $\sigma$, we use homogenous Neumann boundary conditions on
\begin{align*}
    \Big( \{0\}\times[0,3]\times[0,3]\Big) \cup \Big([0,3]\times\{0\}\times[0,3]\Big) \cup \Big([0,3]\times[0,3]\times\{0\}\Big),
\end{align*} 
and Robin boundary conditions on
\begin{align*}
    \Big(\{3\}\times[0,3]\times[0,3]\Big) \cup \Big([0,3]\times\{3\}\times[0,3]\Big) \cup \Big([0,3]\times[0,3]\times\{3\}\Big).
\end{align*}
For the simulations, a mesh with maximal diameter $h_{\max} = \sqrt{3}\cdot 0.3 \approx 0.517$ and minimal diameter $h_{\min} = 0.3 \cdot 2^{-4} = 0.01875$ is chosen.

For the first example, the initial data are given by
\begin{align}
\label{eq:initial_3D_1}
\begin{split}
    \phi_0(x) &= - \tanh\Big( \frac{r(x)-0.1}{\sqrt{2}\epsilon} \Big),
    \\
    r(x) &= \frac{1}{3} \big( 0.5 x_1^4 +  1.5 x_2^4 + 1.5 x_3^4 \big)^{1/4},
    \\
    \sigma_0(x) &= 0.9, 
    \\
    x &= (x_1, x_2, x_3)^T.
\end{split}
\end{align}
In the following, we visualize $\phi$ (top row) and $\sigma$ (middle row) on the subdomain $(0,3)^3$ of $(-3,3)^3$. In the bottom row, we show the surface of the tumour tissue within the whole domain $(-3,3)^3$, where a different perspective is chosen for visualization reasons.
In Figure \ref{fig:3D_4a} we visualize the solution at times $t=0.5, 1, 1.5$ with chemotaxis parameter $\chi_\phi = 30$. %\footnote{15'000 DOF at time step 30, 14'000 DOF at time step 40, 7'000 DOF at time step 170. Up to 2 Newton-iterations per time step.}
Further, the solution with chemotaxis parameter $\chi_\phi=15$ is shown in Figure \ref{fig:3D_4c} at times $t=1, 3, 5$. In both cases, the tumour undergoes morphological instabilities and the shape resembles a dumbbell. For larger value of $\chi_\phi$, the evolution of the tumour is quicker.

%%%%%%%%%%%%%%%%%%%%%%%%%%
\begin{figure}[ht]
    \centering
	\subfloat
	{\includegraphics[width=0.32\textwidth,trim={10cm 2cm 10cm 0},clip]{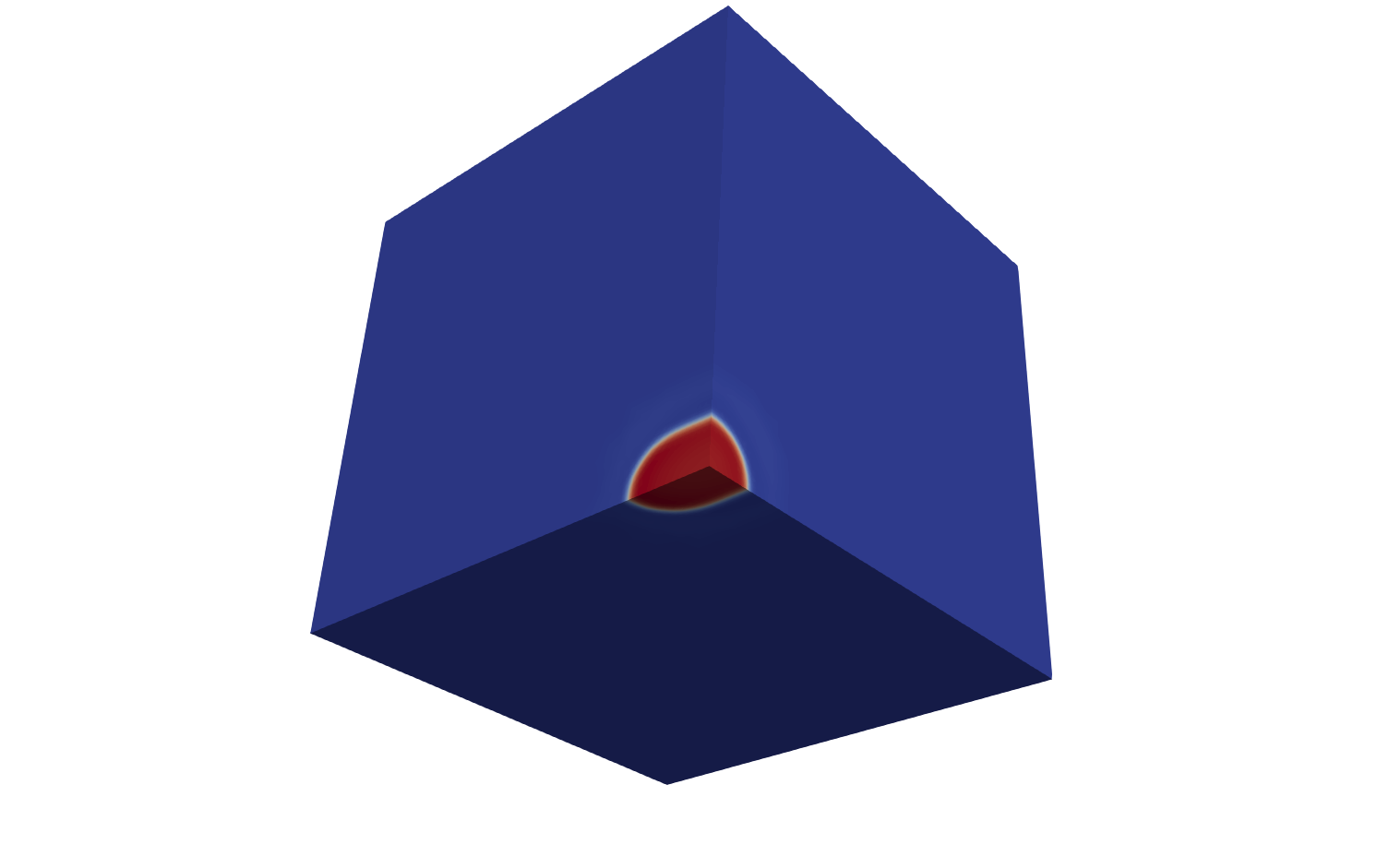}}
	%\hspace{-1em}
	\subfloat
	{\includegraphics[width=0.32\textwidth,trim={10cm 2cm 10cm 0},clip]{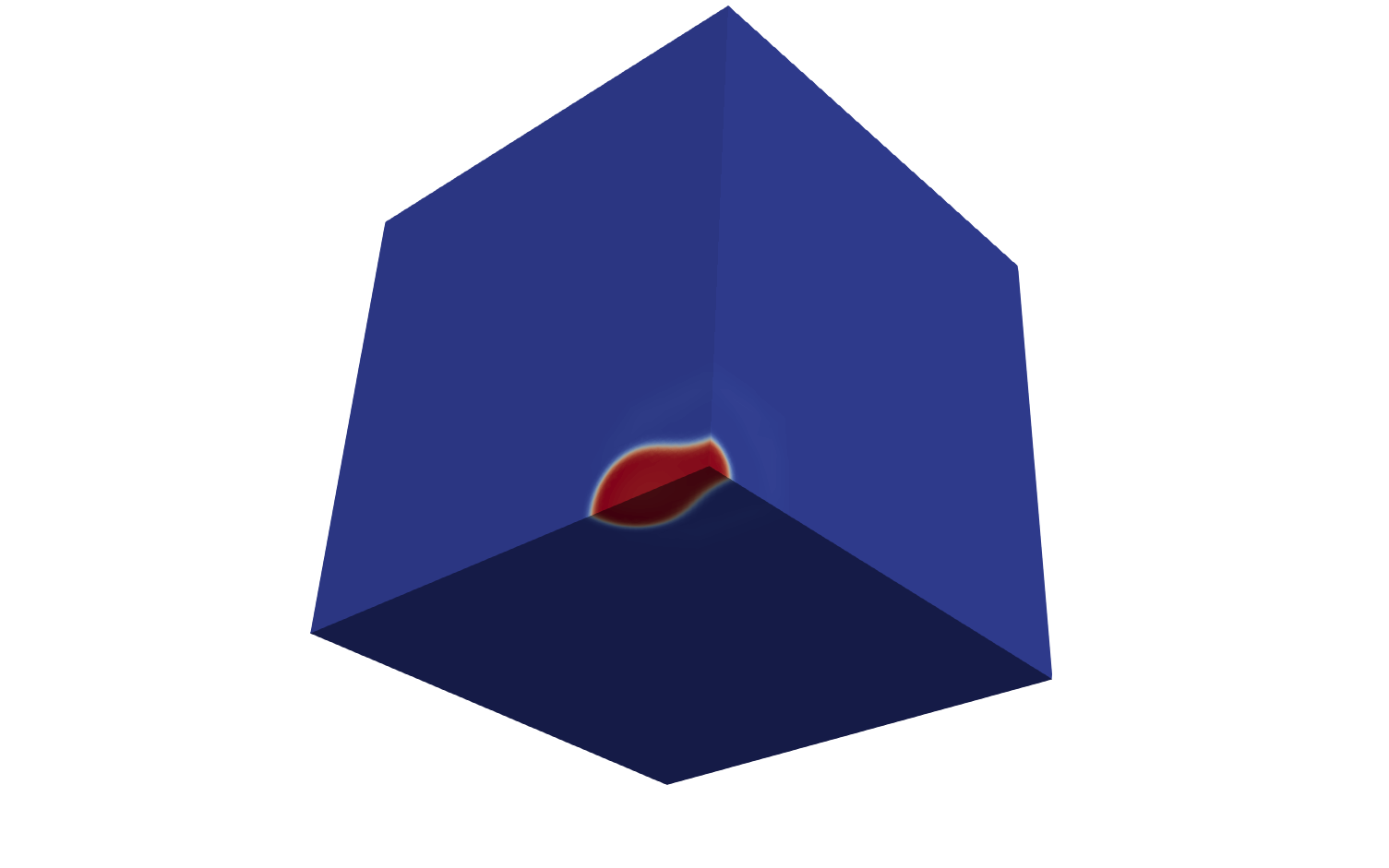}}
	%\hspace{-1em}
	\subfloat
	{\includegraphics[width=0.32\textwidth,trim={10cm 2cm 10cm 0},clip]{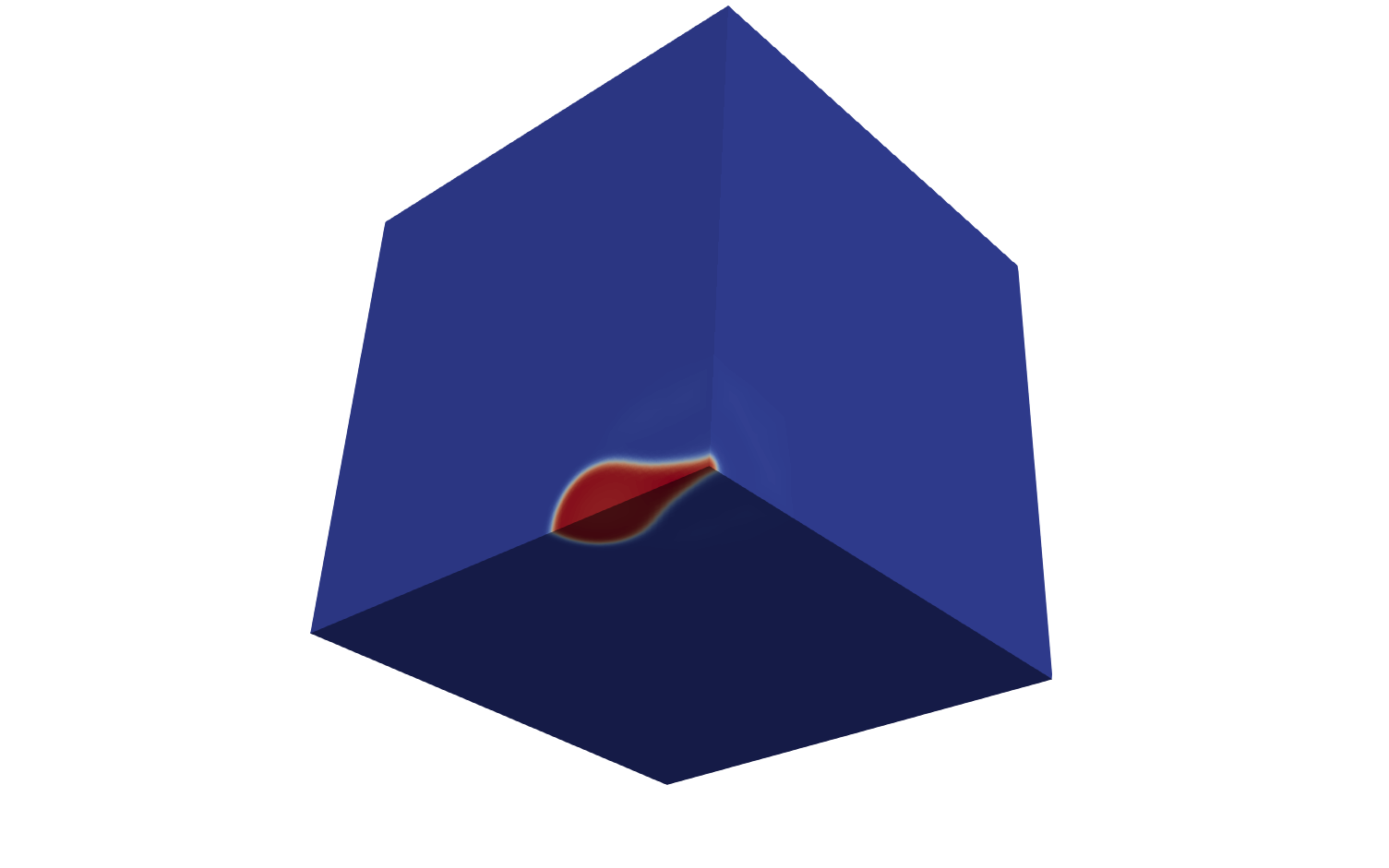}}
	\\%[-2.7ex]
	\subfloat
	{\includegraphics[width=0.32\textwidth,trim={10cm 0 10cm 0},clip]{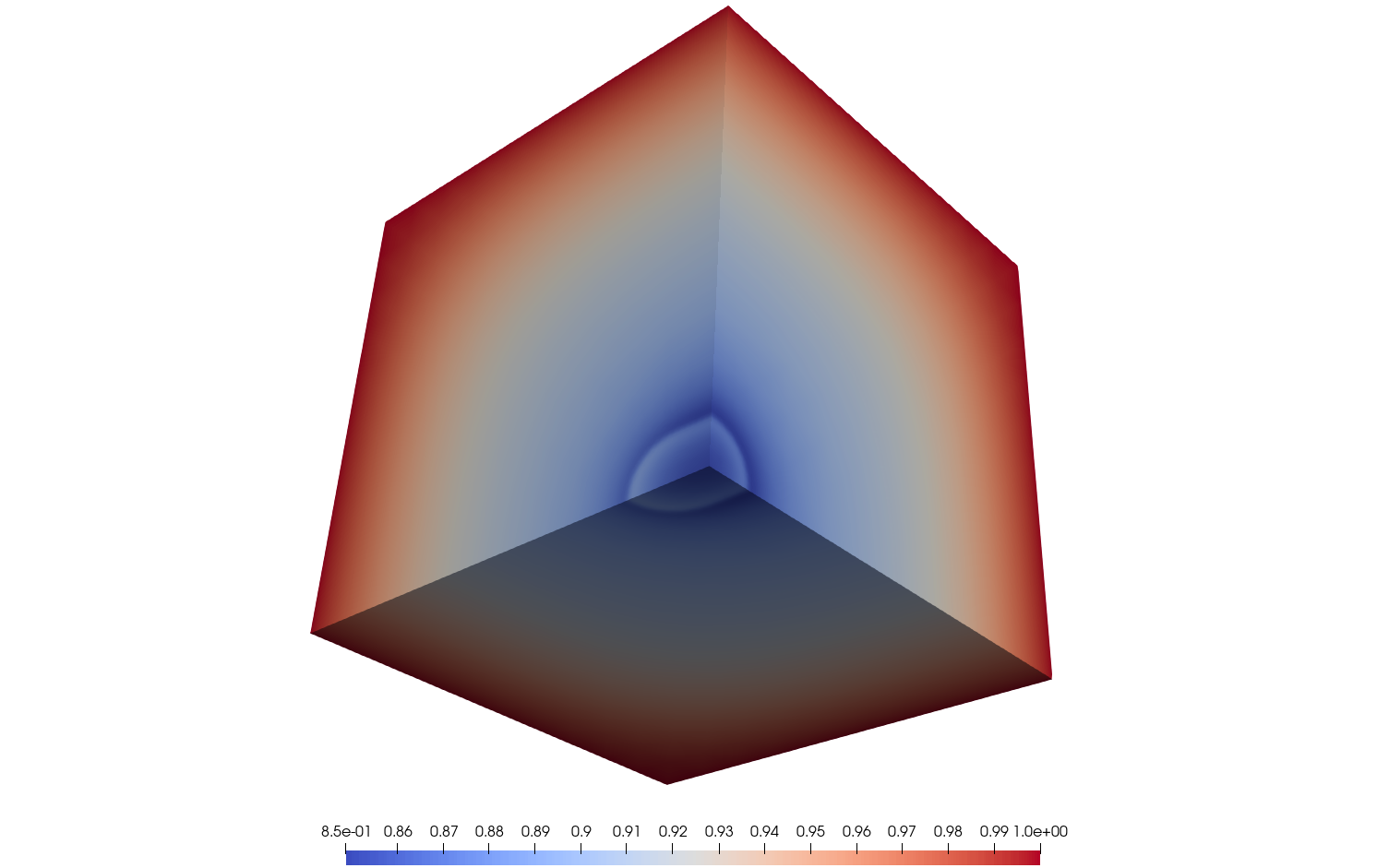}}
	%\hspace{-1em}
	\subfloat
	{\includegraphics[width=0.32\textwidth,trim={10cm 0 10cm 0},clip]{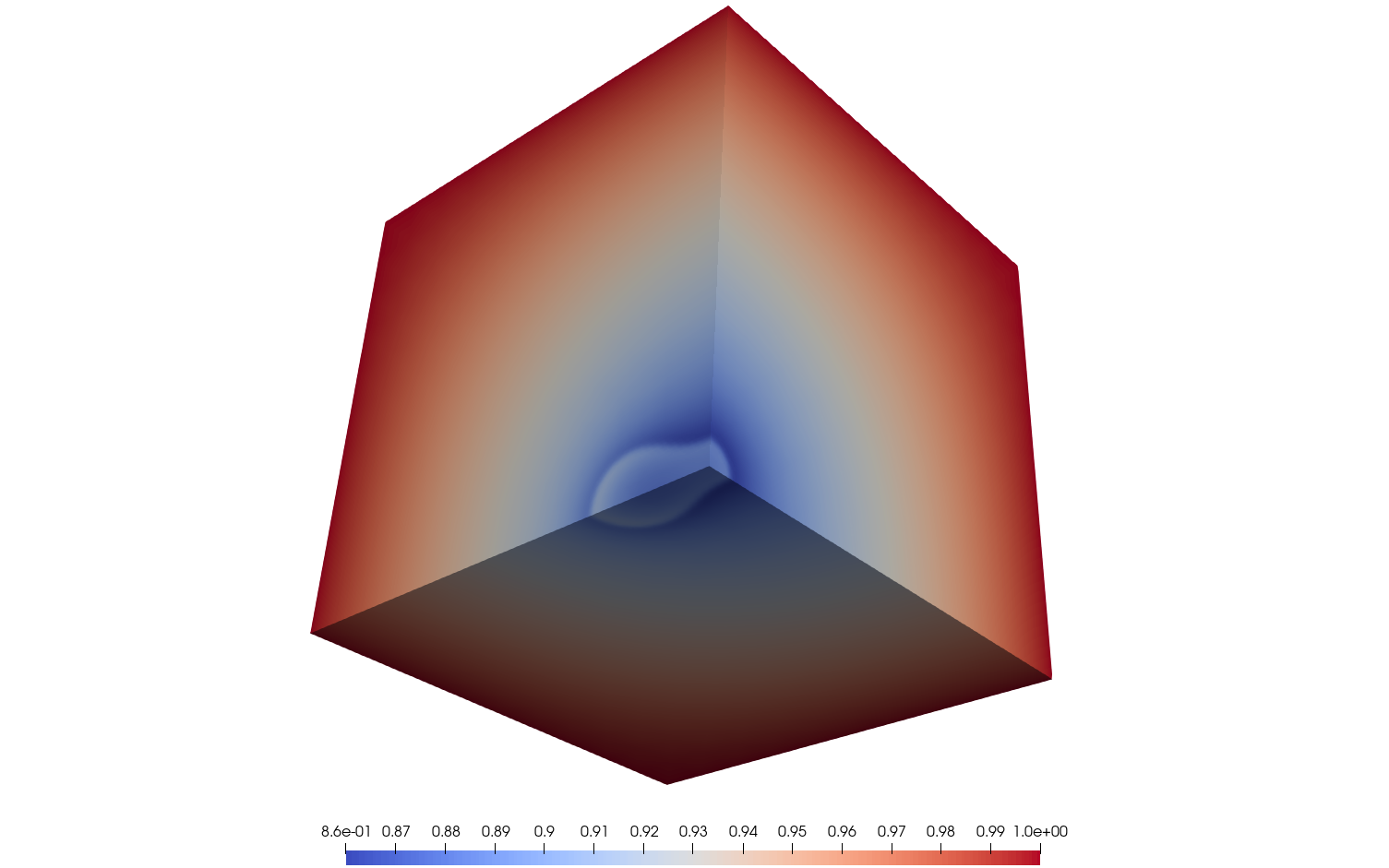}}
	%\hspace{-1em}
	\subfloat
	{\includegraphics[width=0.32\textwidth,trim={10cm 0 10cm 0},clip]{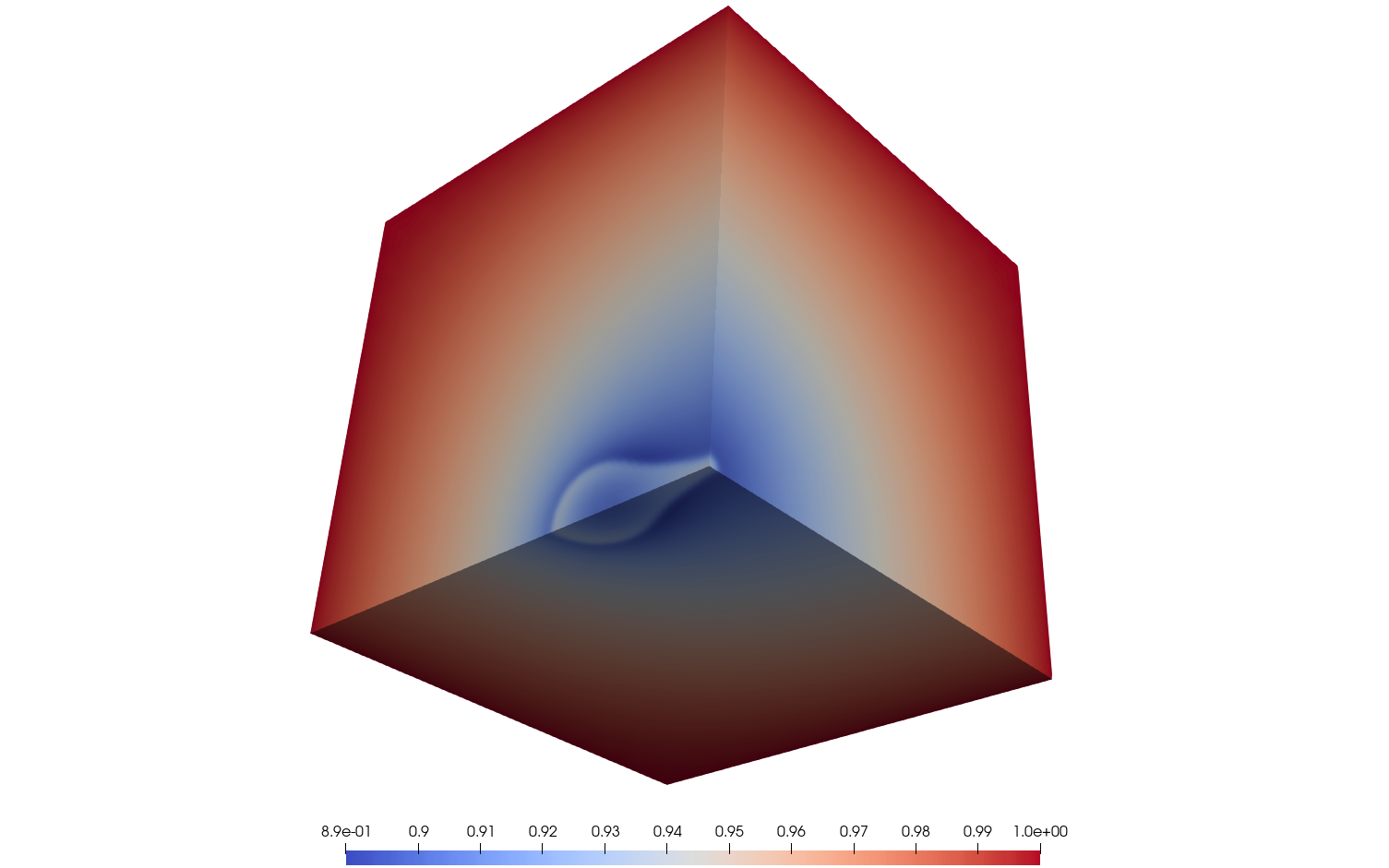}}
	\\%[-2.7ex]
	\subfloat
	{\includegraphics[width=0.32\textwidth,trim={10cm 10cm 10cm 7cm},clip]{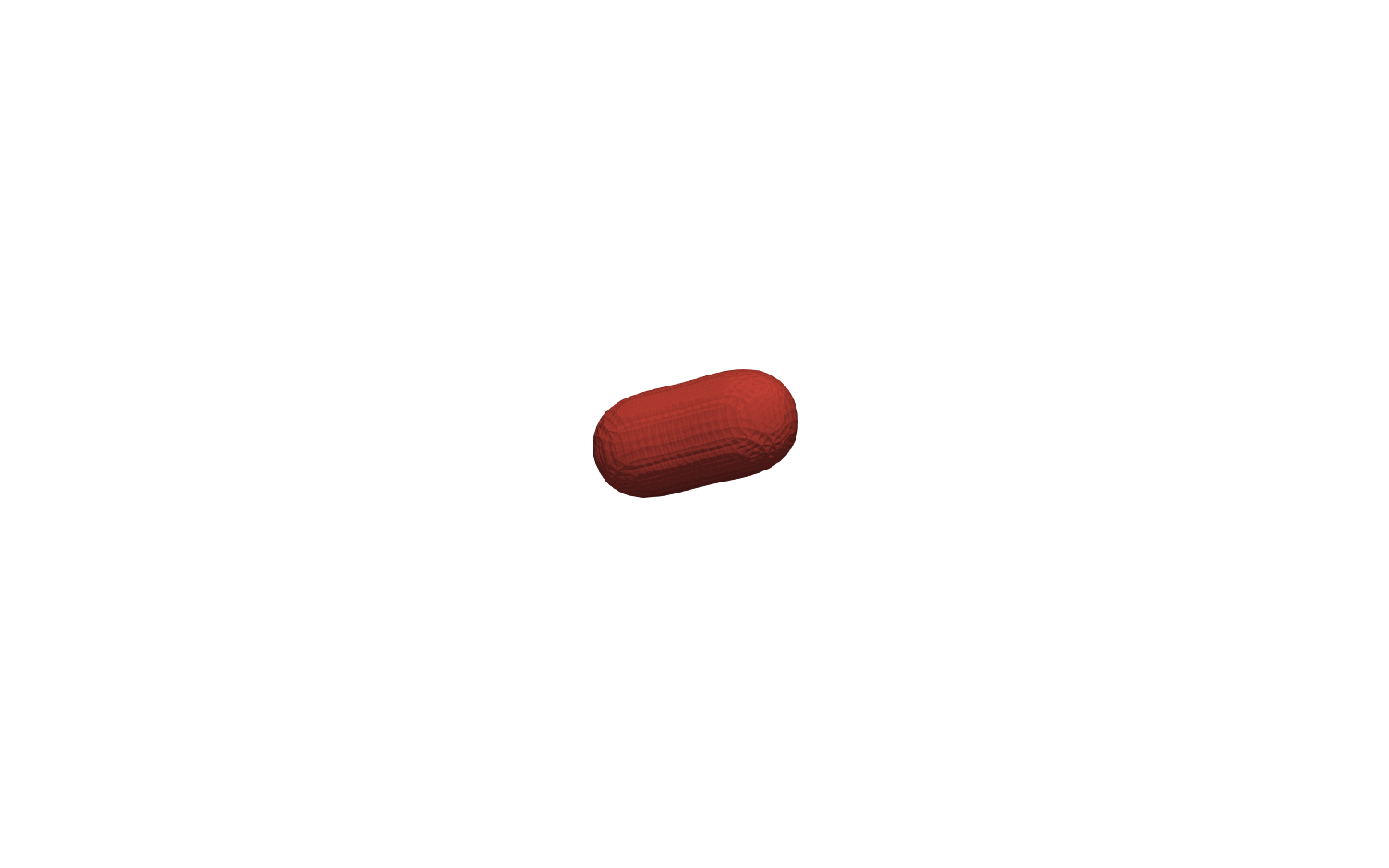}}
	%\hspace{-1em}
	\subfloat
	{\includegraphics[width=0.32\textwidth,trim={10cm 10cm 10cm 7cm},clip]{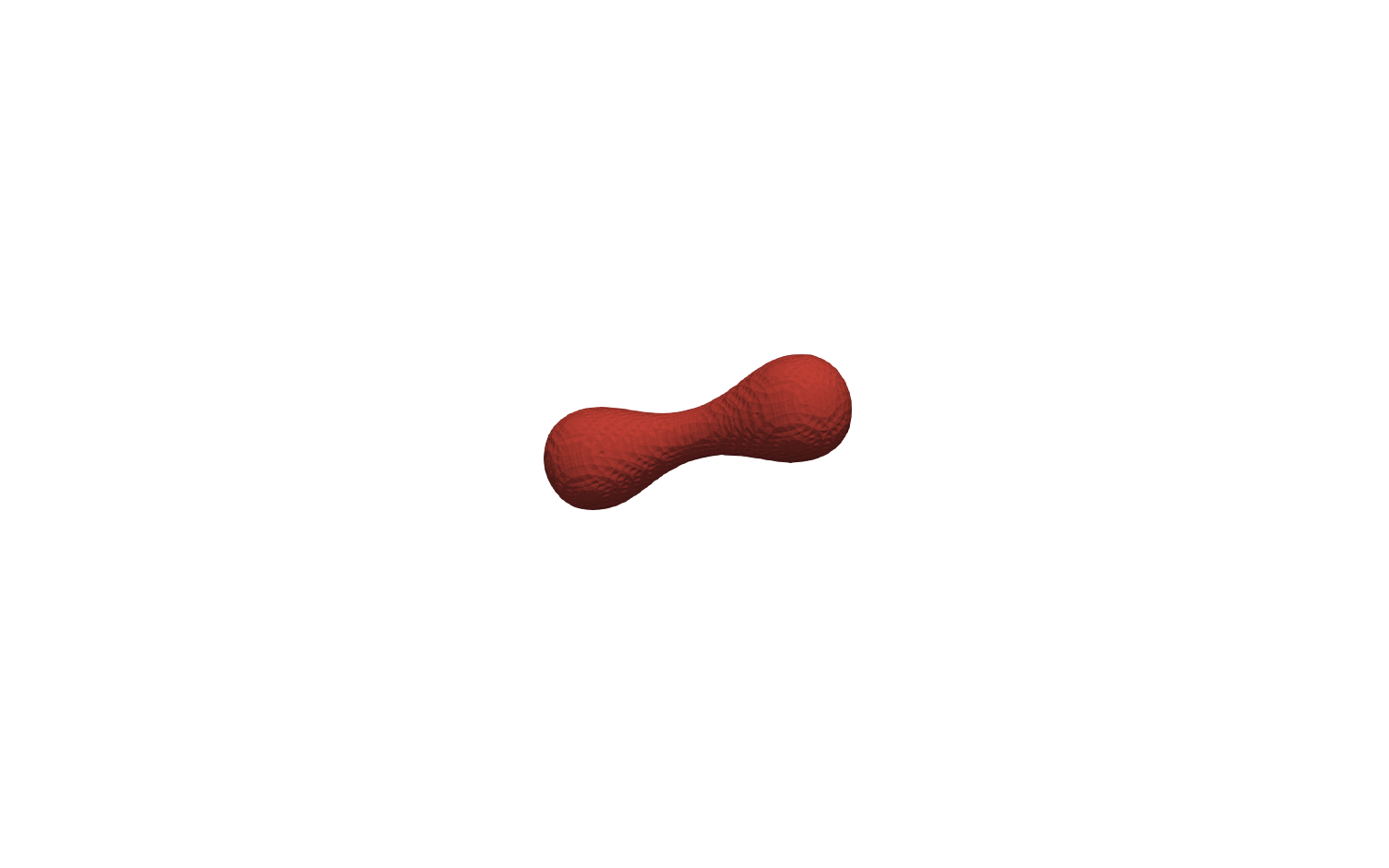}}
	%\hspace{-1em}
	\subfloat
	{\includegraphics[width=0.32\textwidth,trim={10cm 10cm 10cm 7cm},clip]{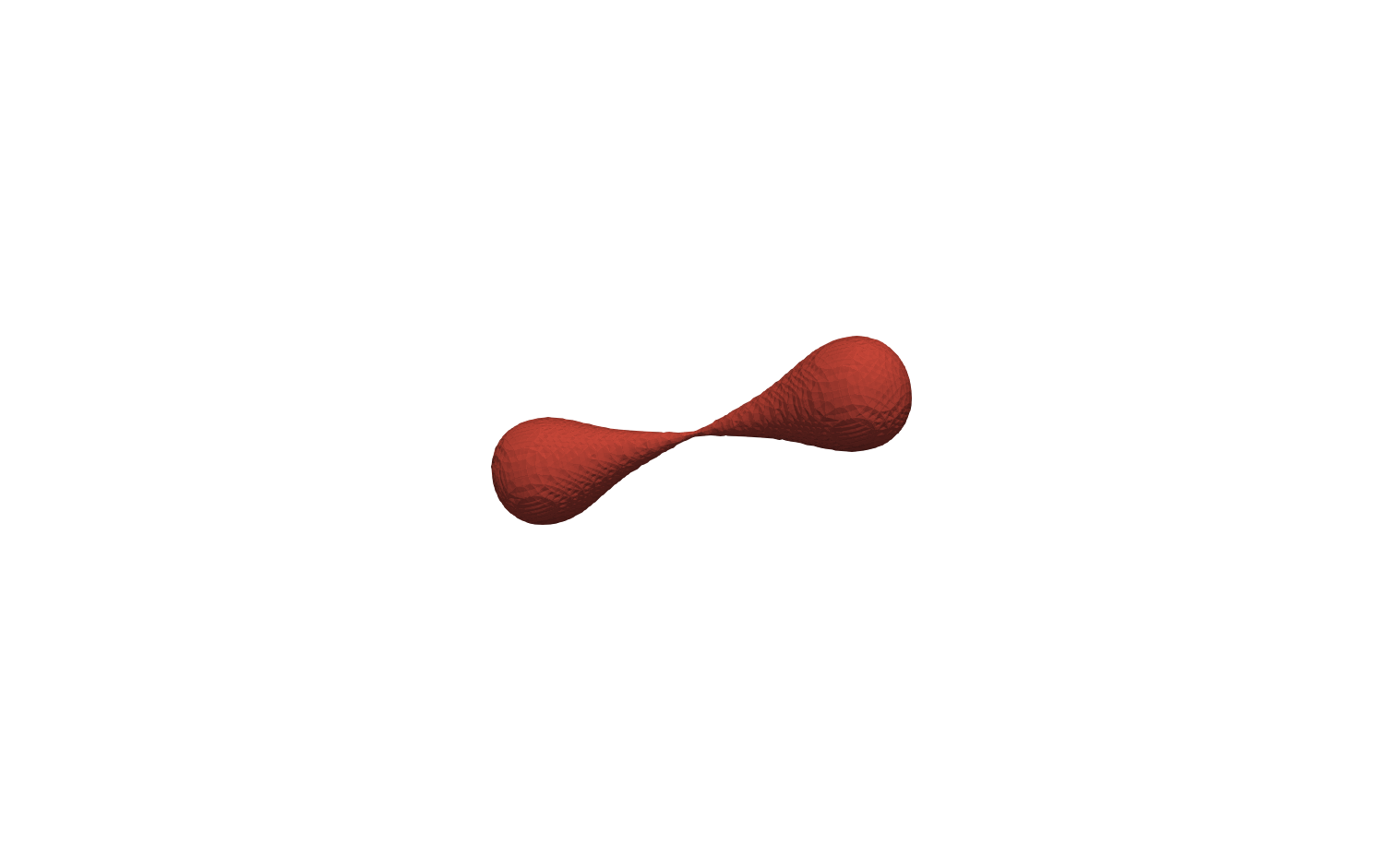}}
	\caption{Numerical solution in three dimensions with initial profile \eqref{eq:initial_3D_1} and chemotaxis parameter $\chi_\phi=30$ at times $t=0.5$ (left), $t=1$ (center) and $t=1.5$ (right).} 
	\label{fig:3D_4a}
\end{figure}

%%%%%%%%%%

\begin{figure}[ht]
    \centering
	\subfloat
	{\includegraphics[width=0.32\textwidth,trim={10cm 2cm 10cm 0},clip]{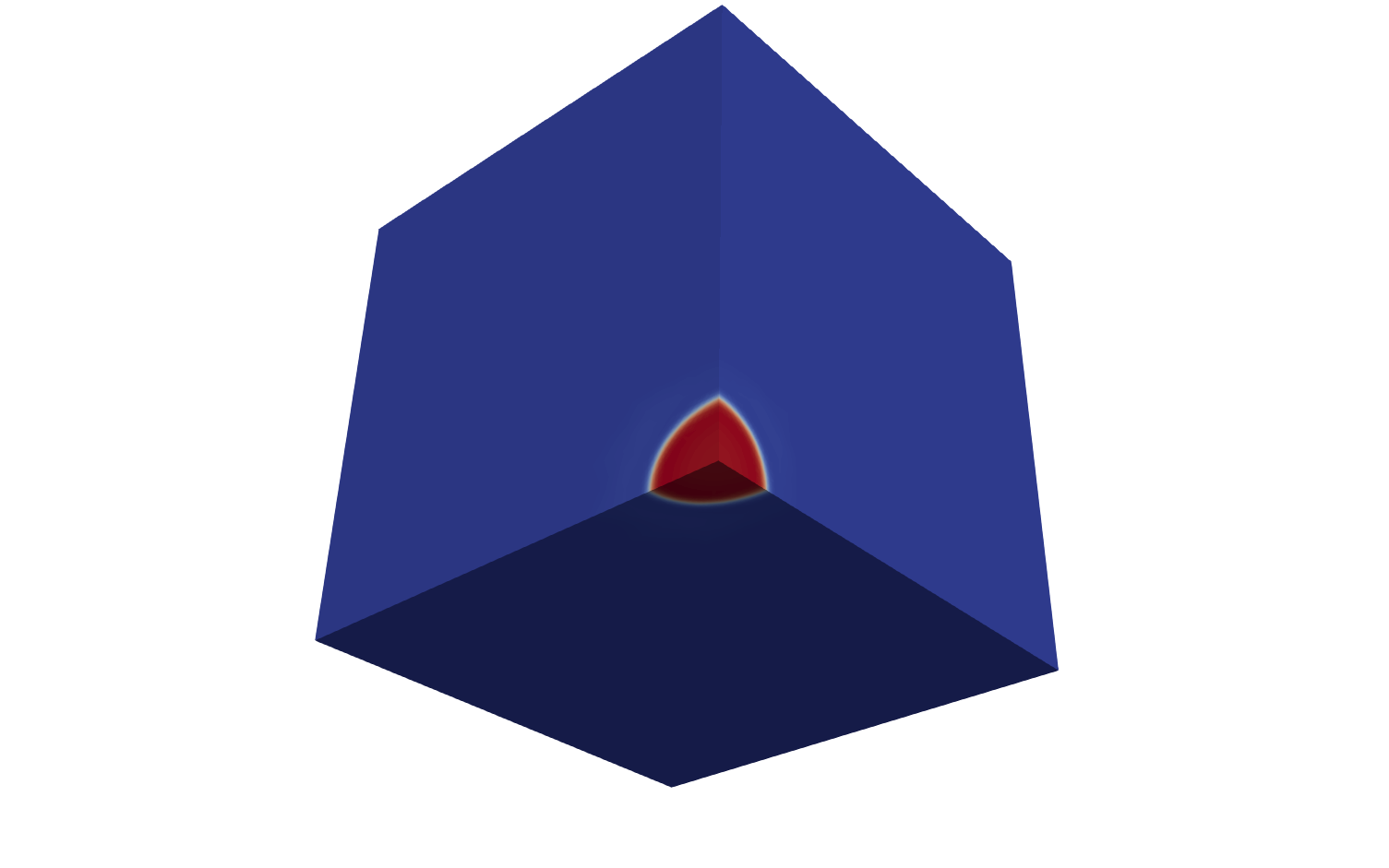}}
% 	\hspace{-1em}
	\subfloat
	{\includegraphics[width=0.32\textwidth,trim={10cm 2cm 10cm 0},clip]{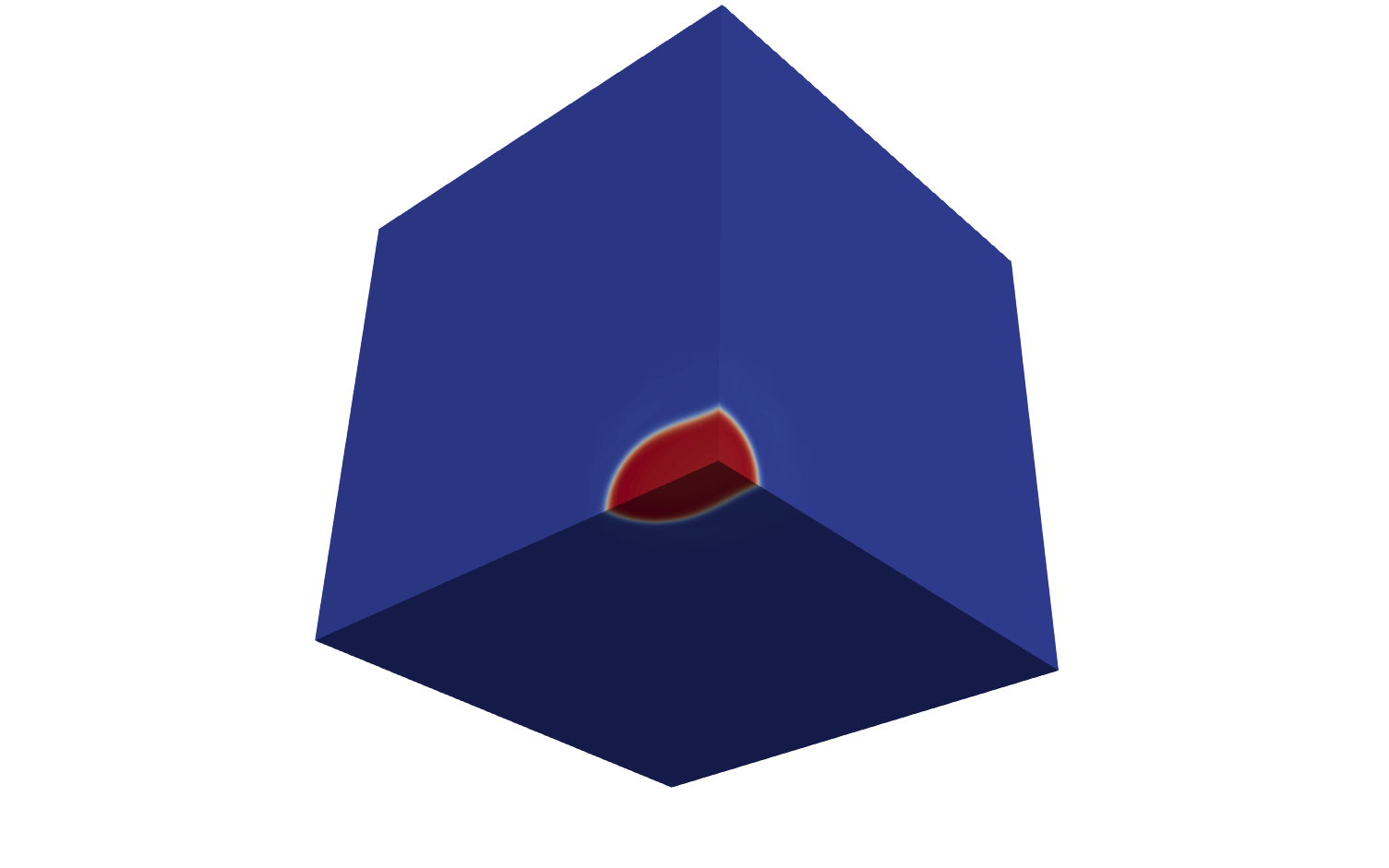}}
% 	\hspace{-1em}
	\subfloat
	{\includegraphics[width=0.32\textwidth,trim={10cm 2cm 10cm 0},clip]{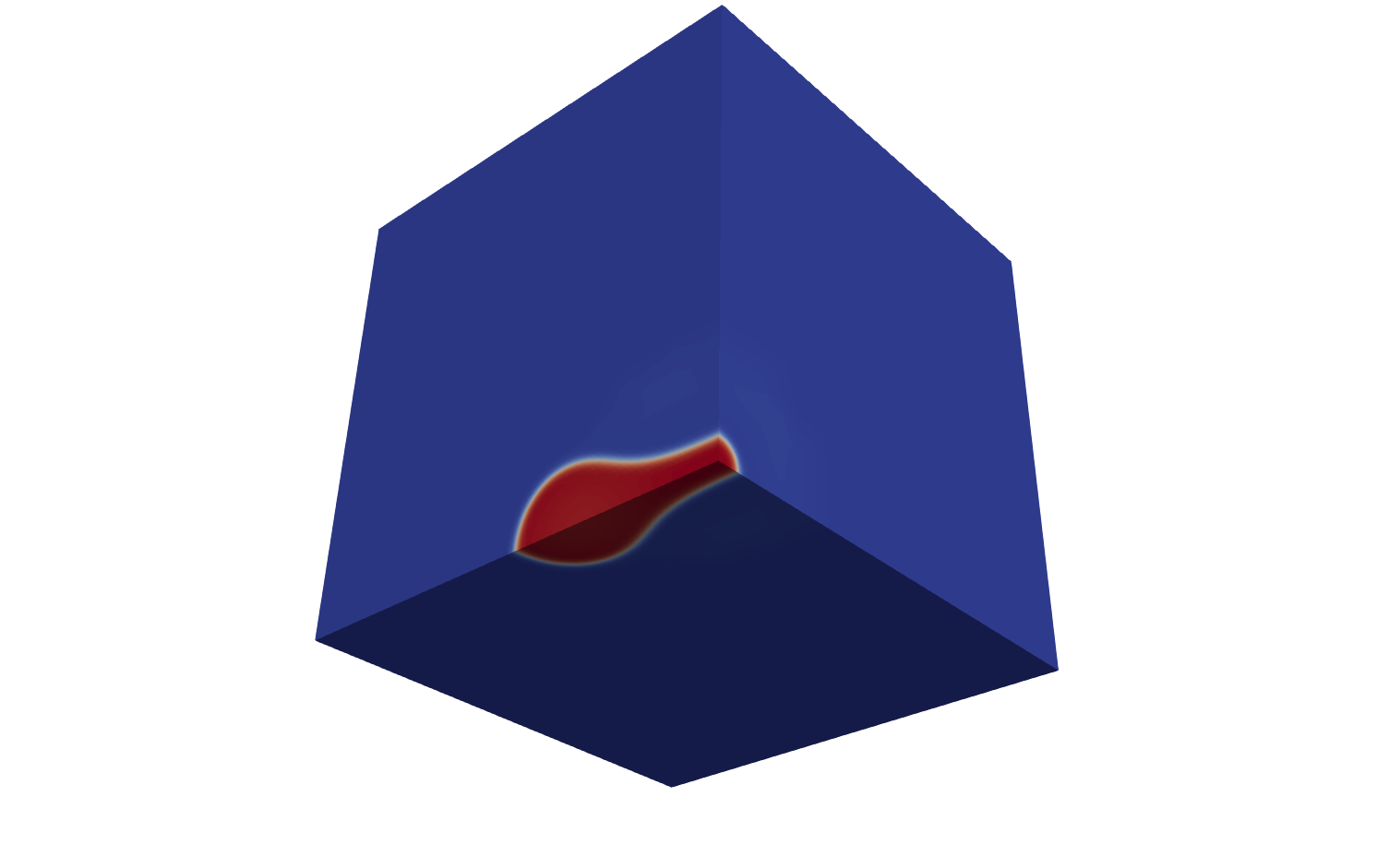}}
	\\%[-2.7ex]
	\subfloat
	{\includegraphics[width=0.32\textwidth,trim={10cm 0 10cm 0},clip]{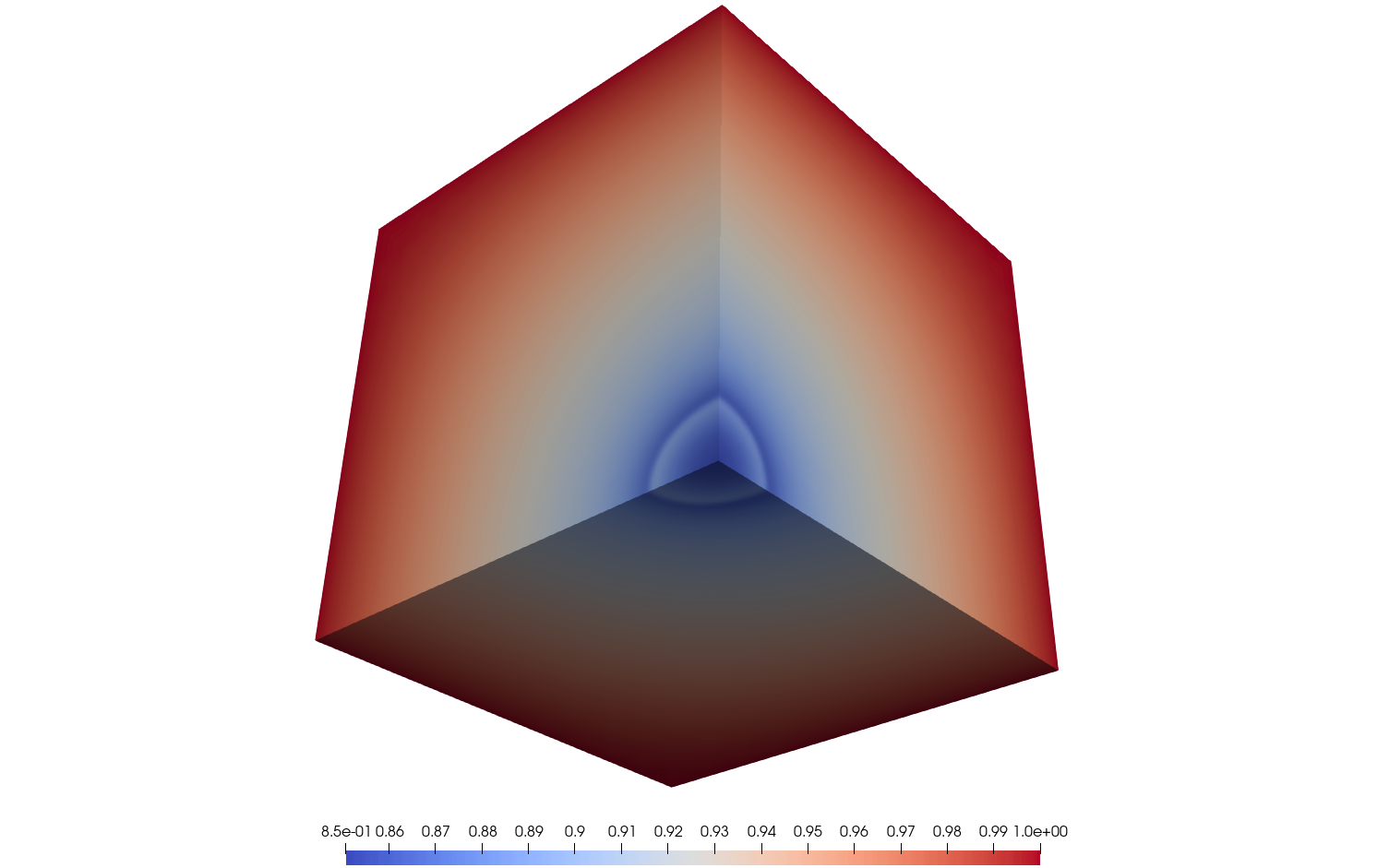}}
% 	\hspace{-1em}
	\subfloat
	{\includegraphics[width=0.32\textwidth,trim={10cm 0 10cm 0},clip]{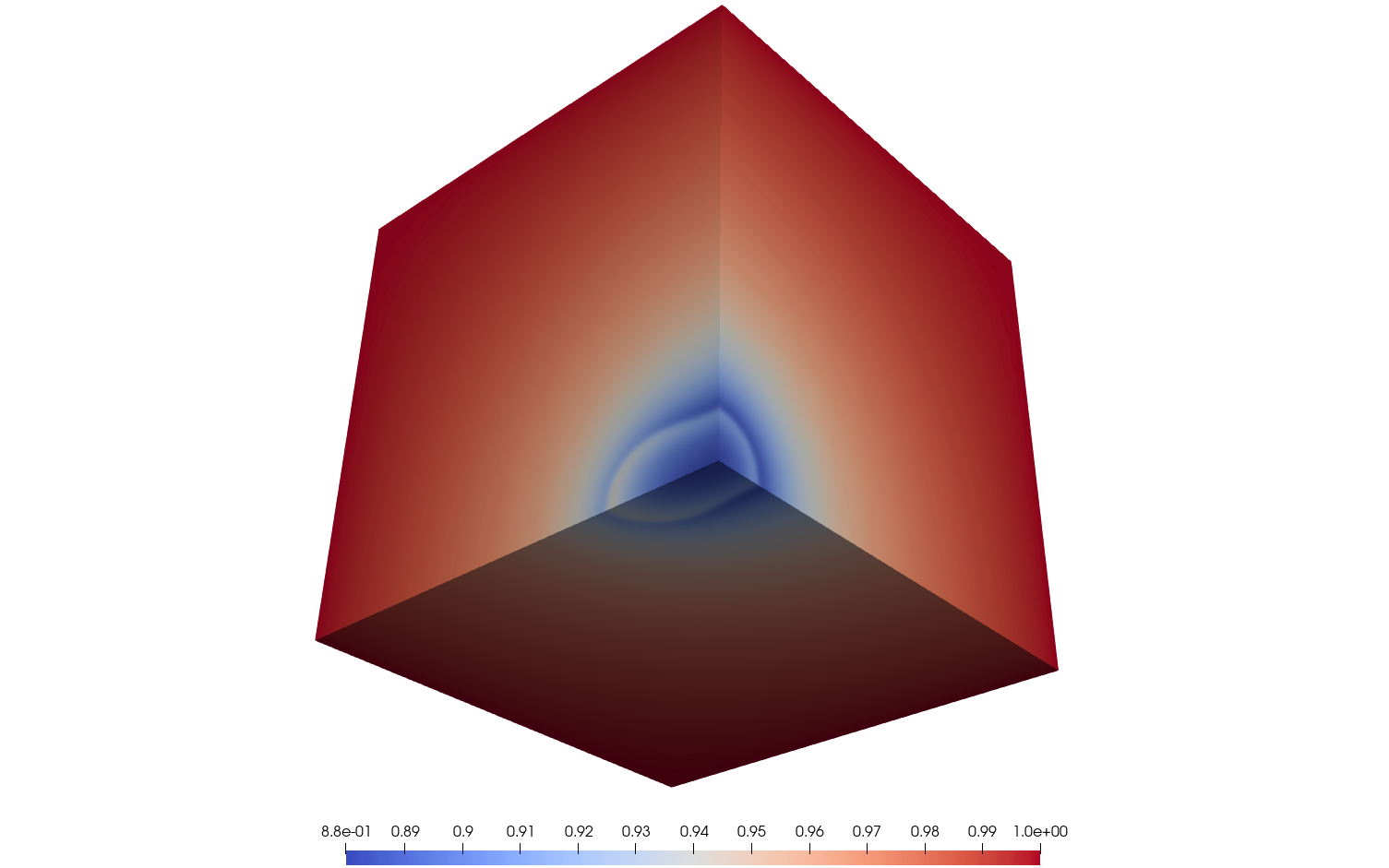}}
% 	\hspace{-1em}
	\subfloat
	{\includegraphics[width=0.32\textwidth,trim={10cm 0 10cm 0},clip]{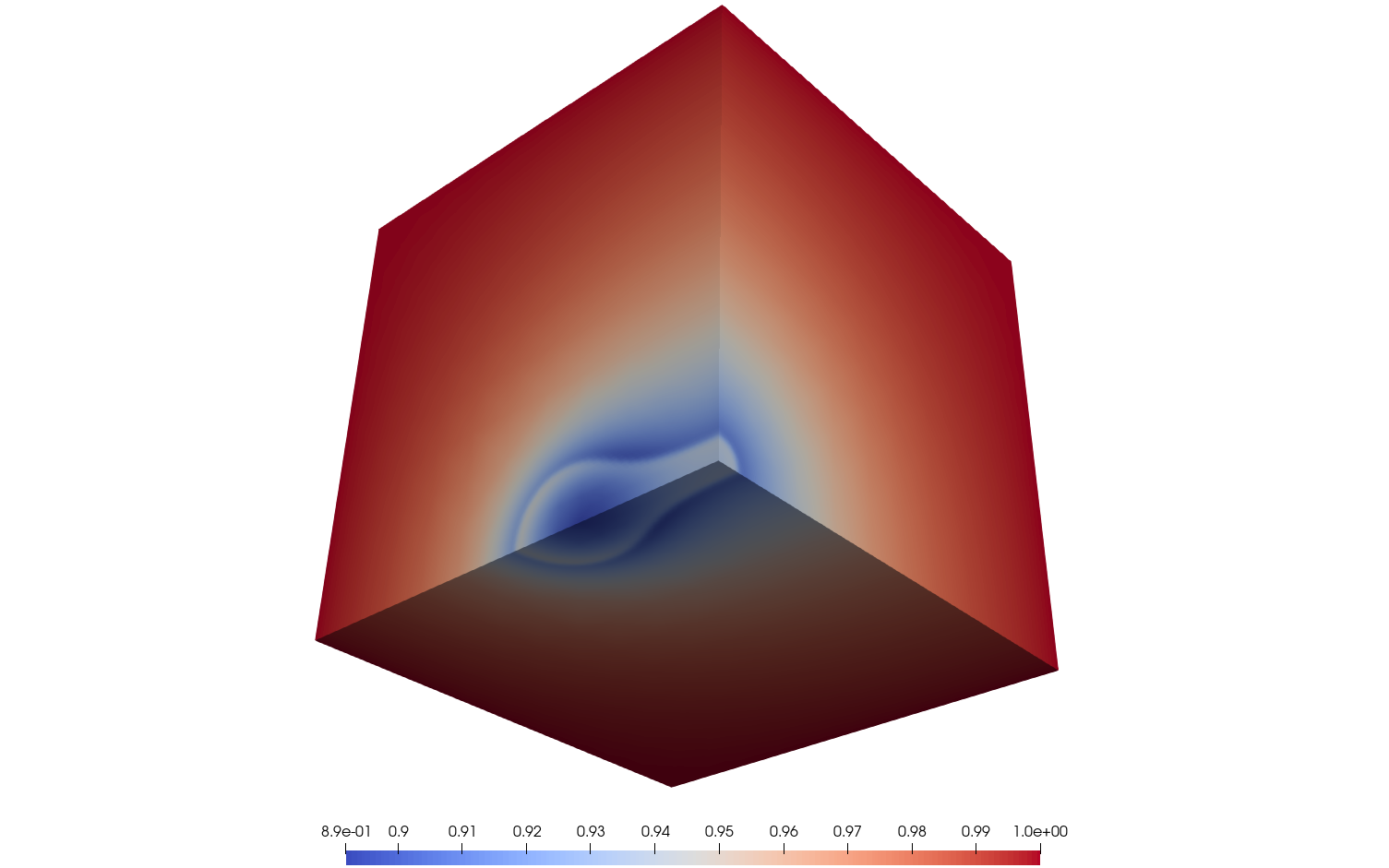}}
	\\%[-2.7ex]
	\subfloat
	{\includegraphics[width=0.32\textwidth,trim={10cm 10cm 10cm 7cm},clip]{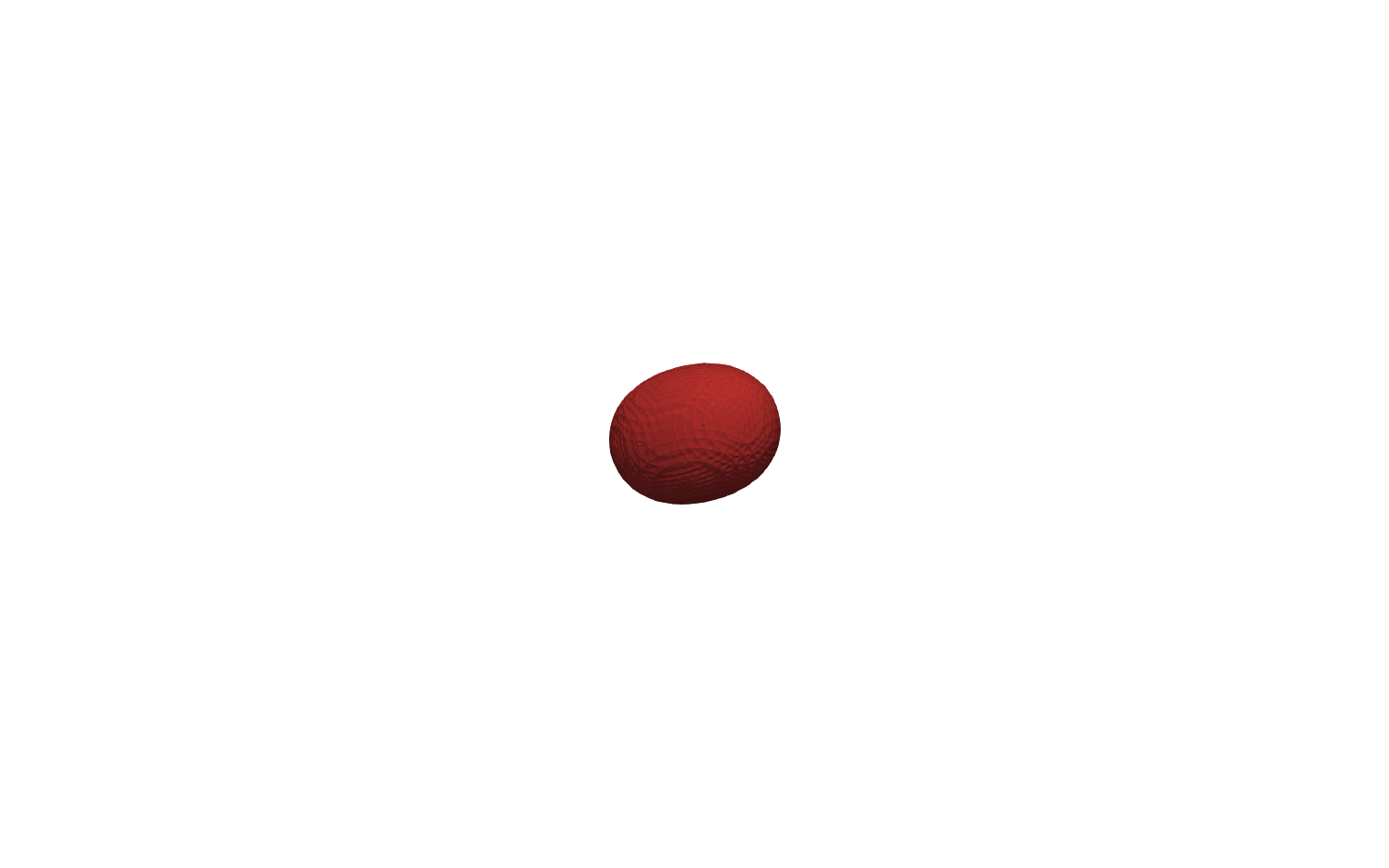}}
% 	\hspace{-1em}
	\subfloat
	{\includegraphics[width=0.32\textwidth,trim={10cm 10cm 10cm 7cm},clip]{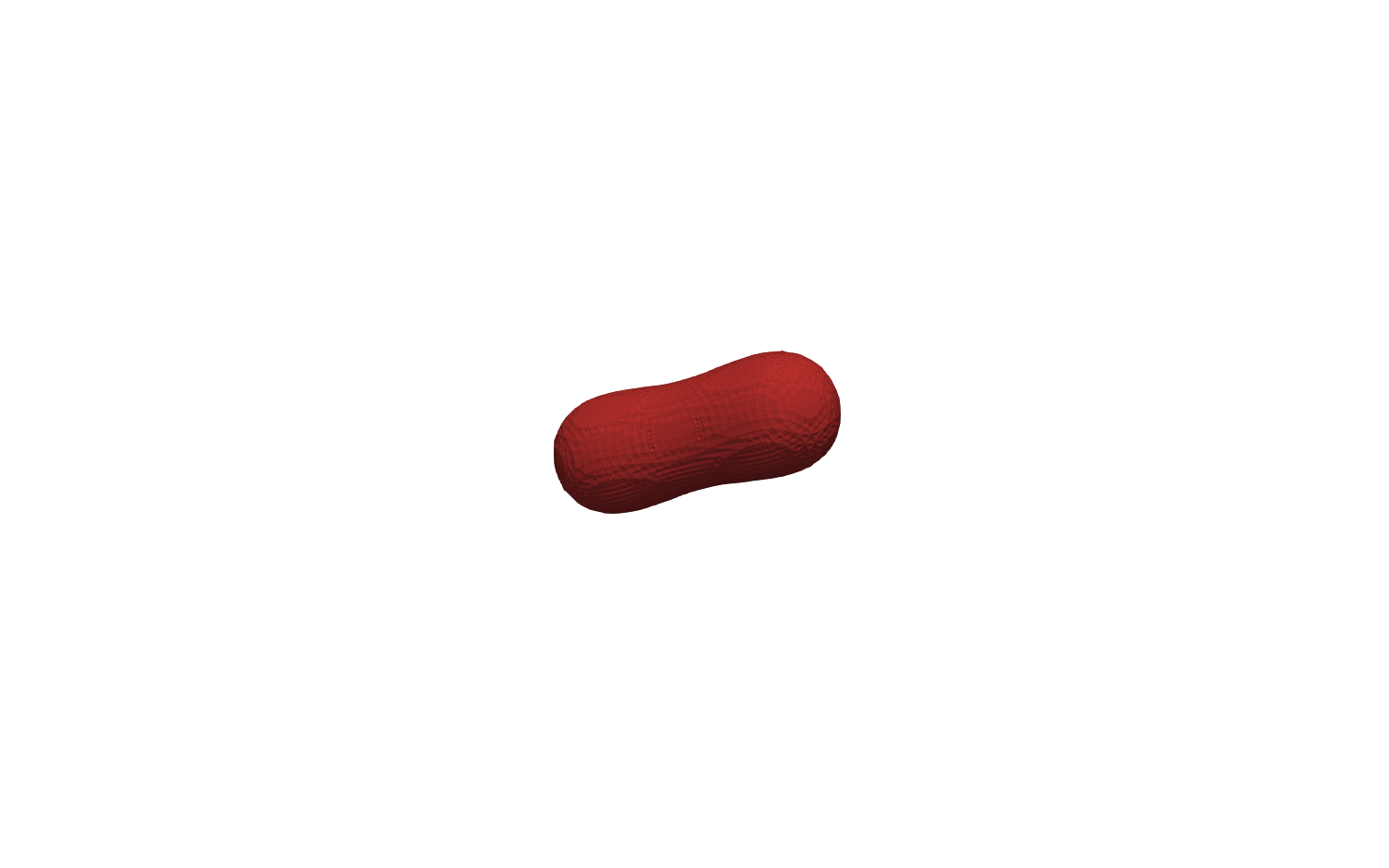}}
% 	\hspace{-1em}
	\subfloat
	{\includegraphics[width=0.32\textwidth,trim={10cm 10cm 10cm 7cm},clip]{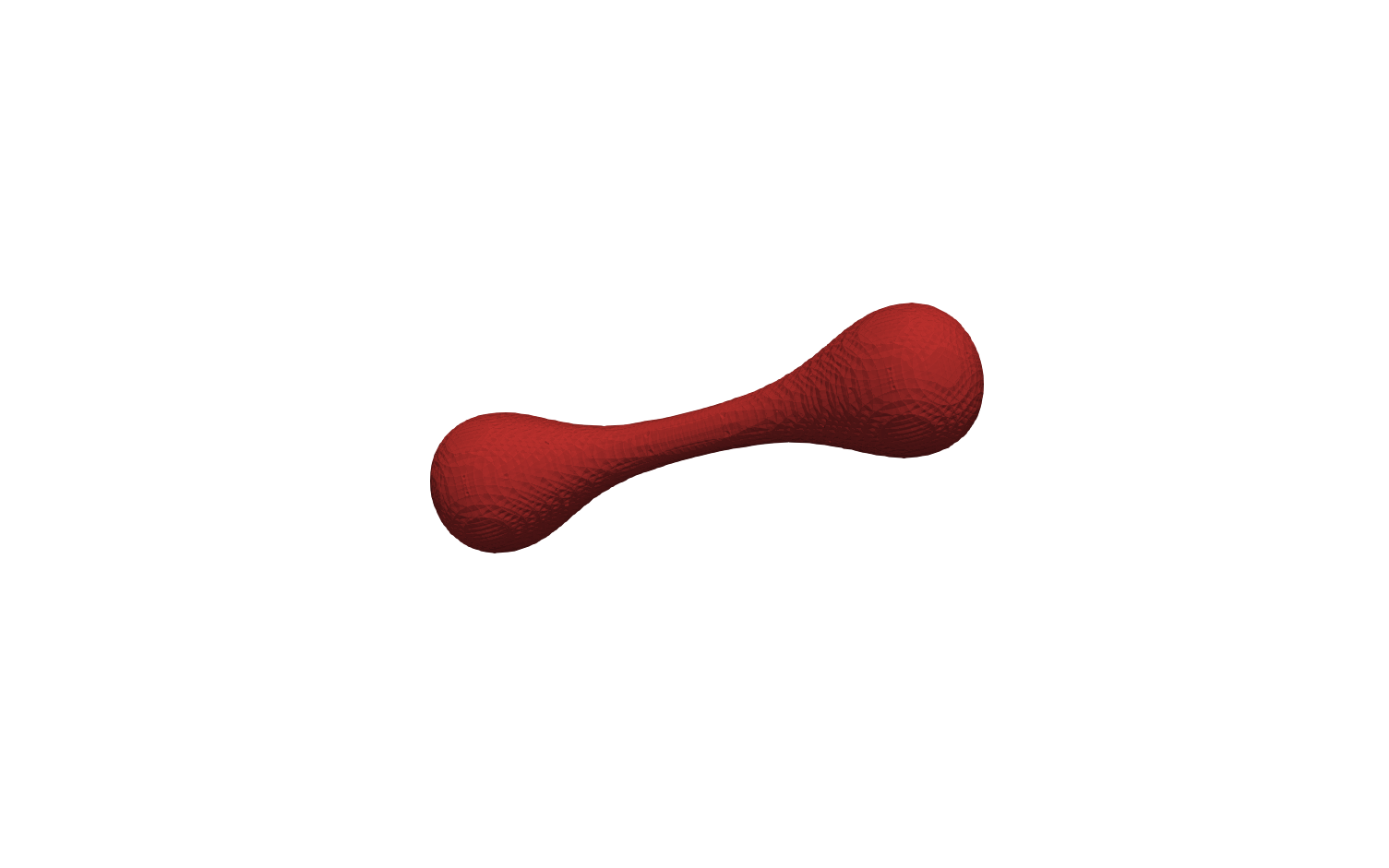}}
	\caption{Numerical solution in three dimensions with initial profile \eqref{eq:initial_3D_1} and chemotaxis parameter $\chi_\phi=15$ at times $t=1$ (left), $t=3$ (center) and $t=5$ (right).} 
	\label{fig:3D_4c}
\end{figure}
%%%%%%%%%%%%%%%%%%%%%%%%%%%%%%%%%

In the second example, %\footnote{throughout the simulation, around 10'000 DOF and up to 2 Newton iterations per time step.}, 
we use the following initial data:
\begin{align}
\label{eq:initial_3D_2}
\begin{split}
    \phi_0(x) &= - \tanh\Big( \frac{r(x)-0.2}{\sqrt{2}\epsilon} \Big),
    \\
    r(x) &= \frac{1}{3} \big( \abs{x_1} + \abs{x_2} + \abs{x_3} \big),
    \\
    \sigma_0(x) &= 0.9,
    \\
    x &= (x_1, x_2, x_3)^T,
\end{split}
\end{align}
and we use $\lambda_p = 0.1$ and $\chi_\phi = 50$. The other parameters are chosen like in \eqref{eq:3d_parameter}. We visualize the solution at times $t=0.1, 0.2, 0.3$ in Figure \ref{fig:3D_2}. As before, we visualize $\phi$ (top row) and $\sigma$ (middle row) on the subdomain $(0,3)^3$ of $(-3,3)^3$. In the bottom row, the shape of the tumour within the whole domain $(-3,3)^3$ is presented.
One can clearly see that an instability with six enhanced fingers arises.

\begin{figure}[ht]
    \centering
	\subfloat
	{\includegraphics[width=0.32\textwidth,trim={10cm 2cm 10cm 0},clip]{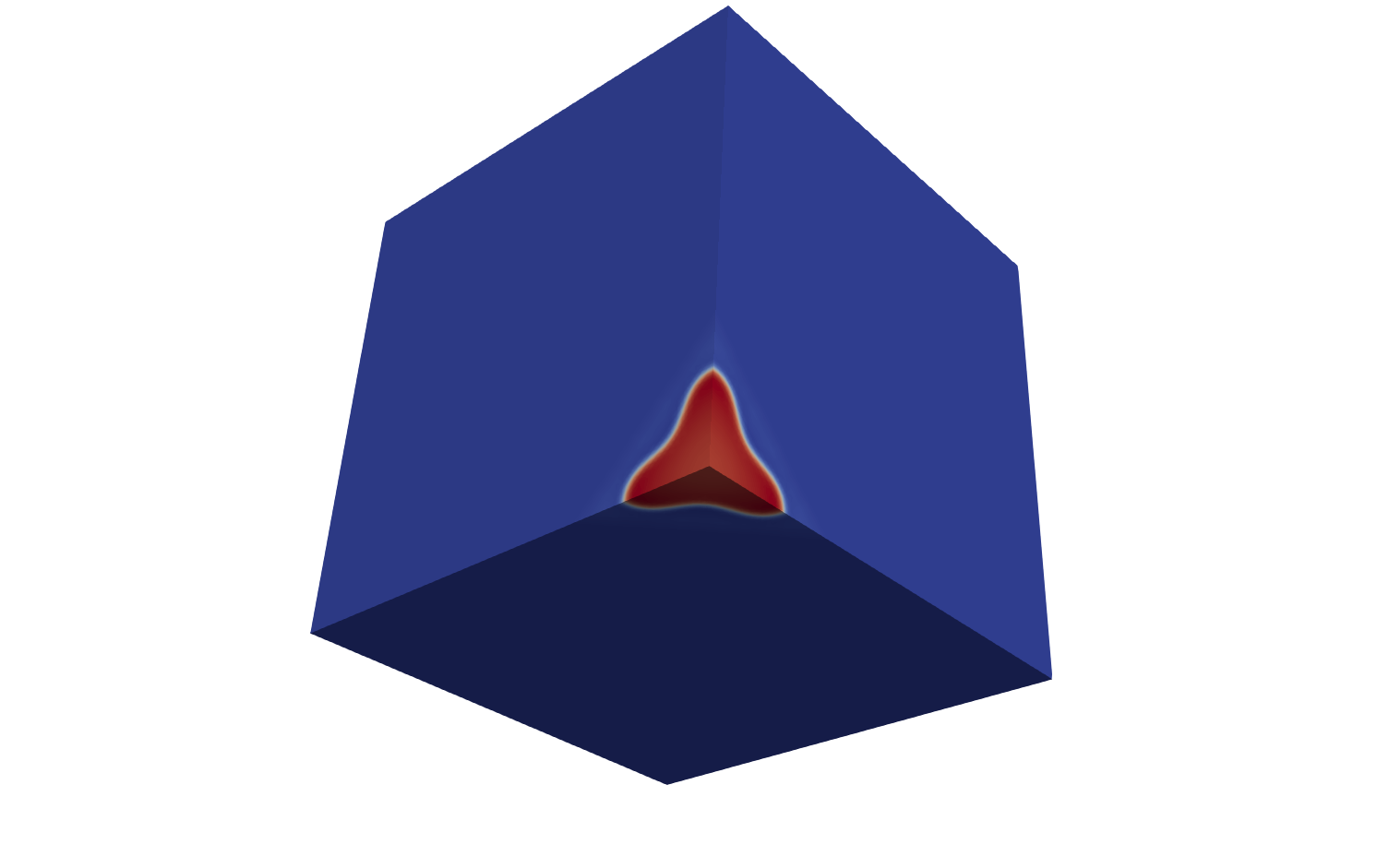}}
% 	\hspace{-5em}
	\subfloat
	{\includegraphics[width=0.32\textwidth,trim={10cm 2cm 10cm 0},clip]{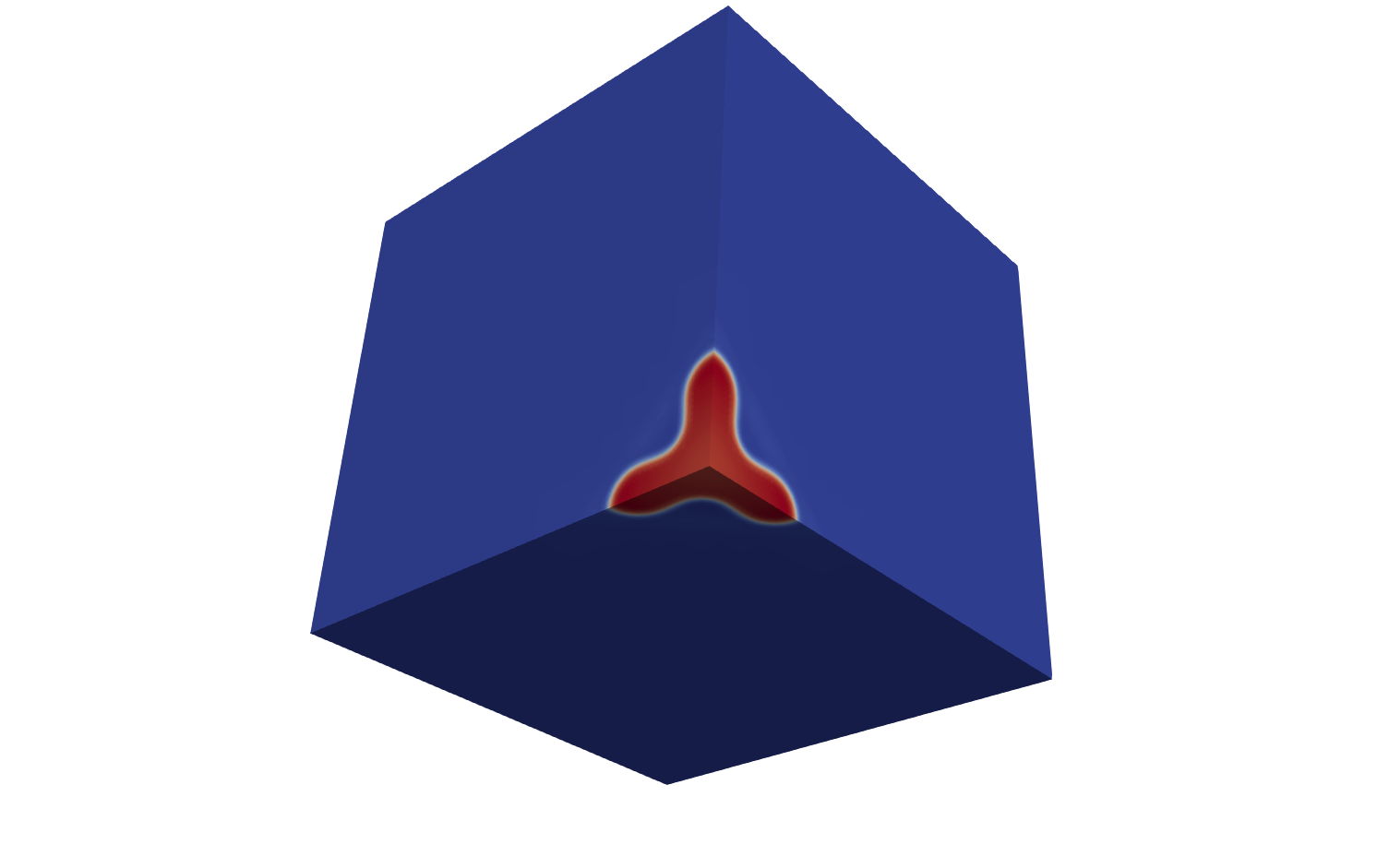}}
% 	\hspace{-5em}
	\subfloat
	{\includegraphics[width=0.32\textwidth,trim={10cm 2cm 10cm 0},clip]{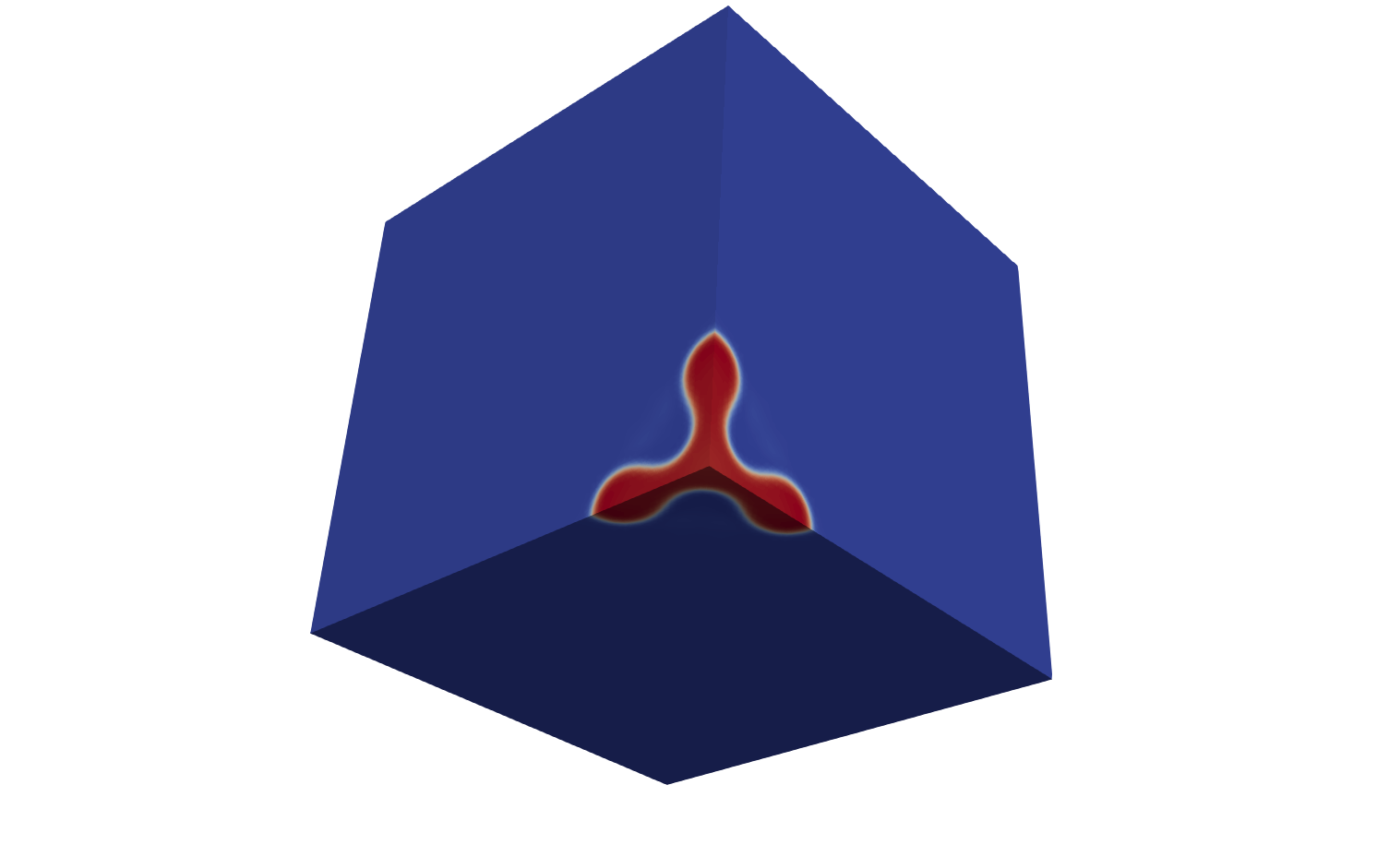}}
	\\%[-2.7ex]
	\subfloat
	{\includegraphics[width=0.32\textwidth,trim={10cm 0 10cm 0},clip]{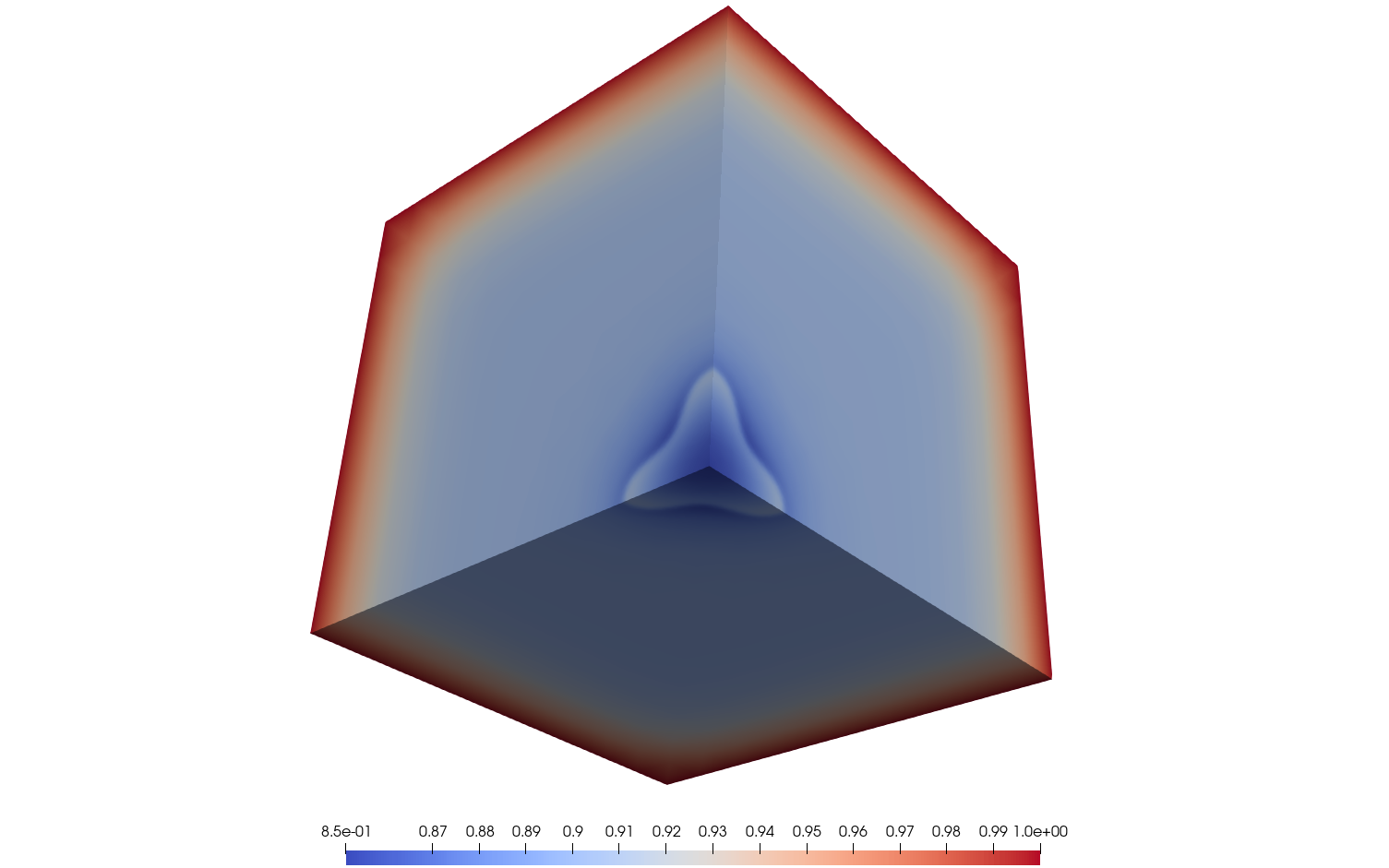}}
% 	\hspace{-1em}
	\subfloat
	{\includegraphics[width=0.32\textwidth,trim={10cm 0 10cm 0},clip]{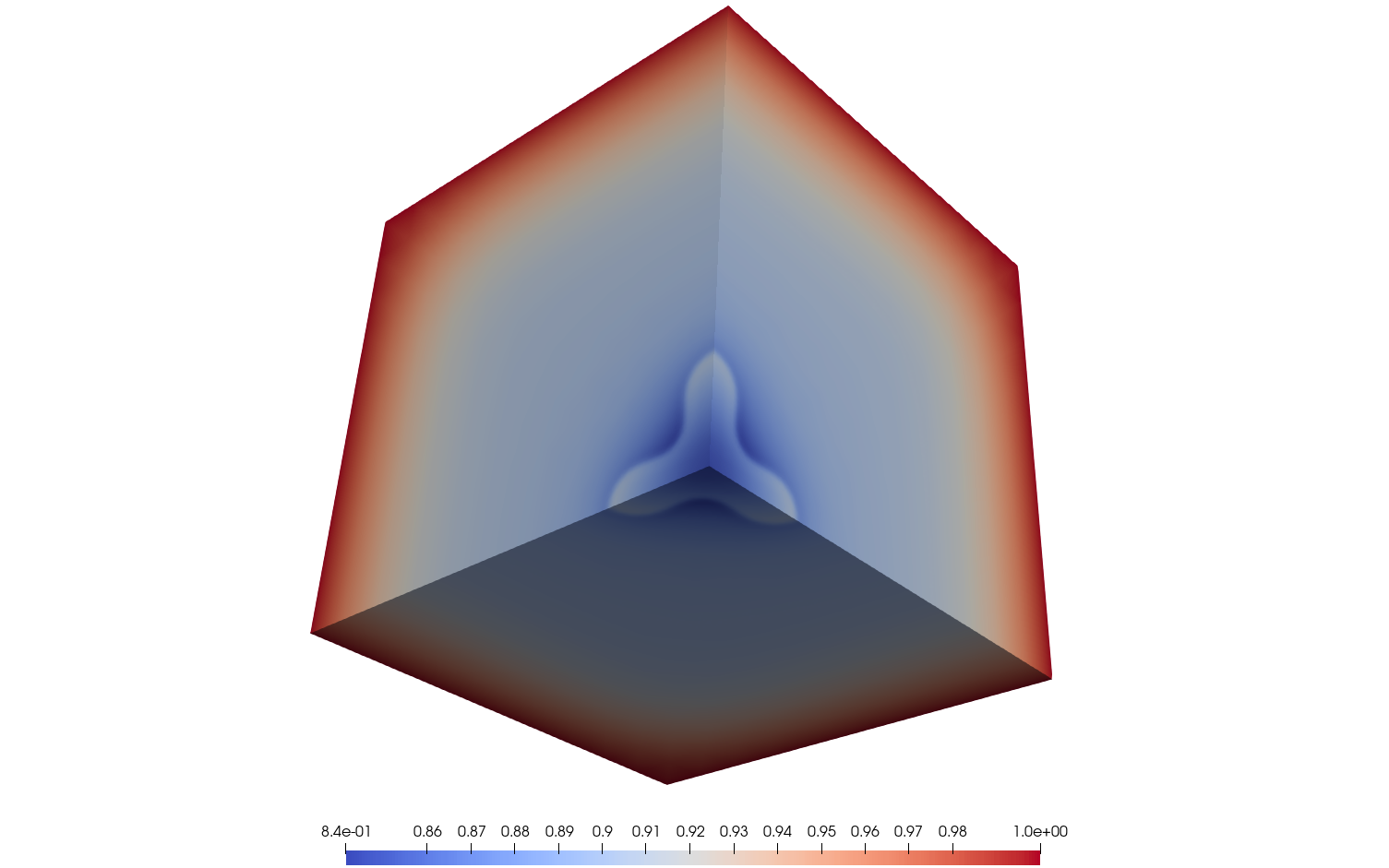}}
% 	\hspace{-1em}
	\subfloat
	{\includegraphics[width=0.32\textwidth,trim={10cm 0 10cm 0},clip]{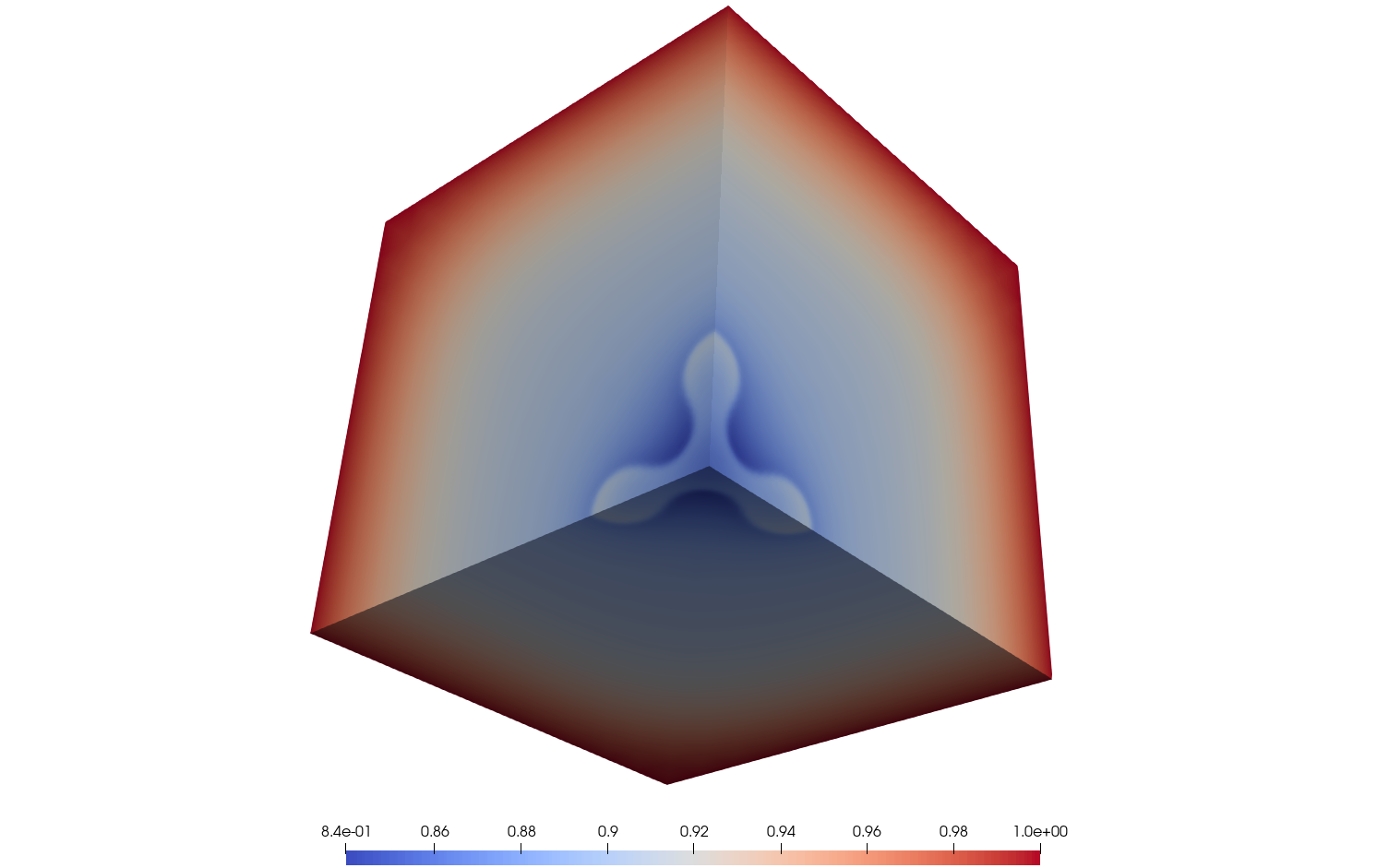}}
	\\%[-2.7ex]
	\subfloat
	{\includegraphics[width=0.32\textwidth,trim={10cm 10cm 10cm 7cm},clip]{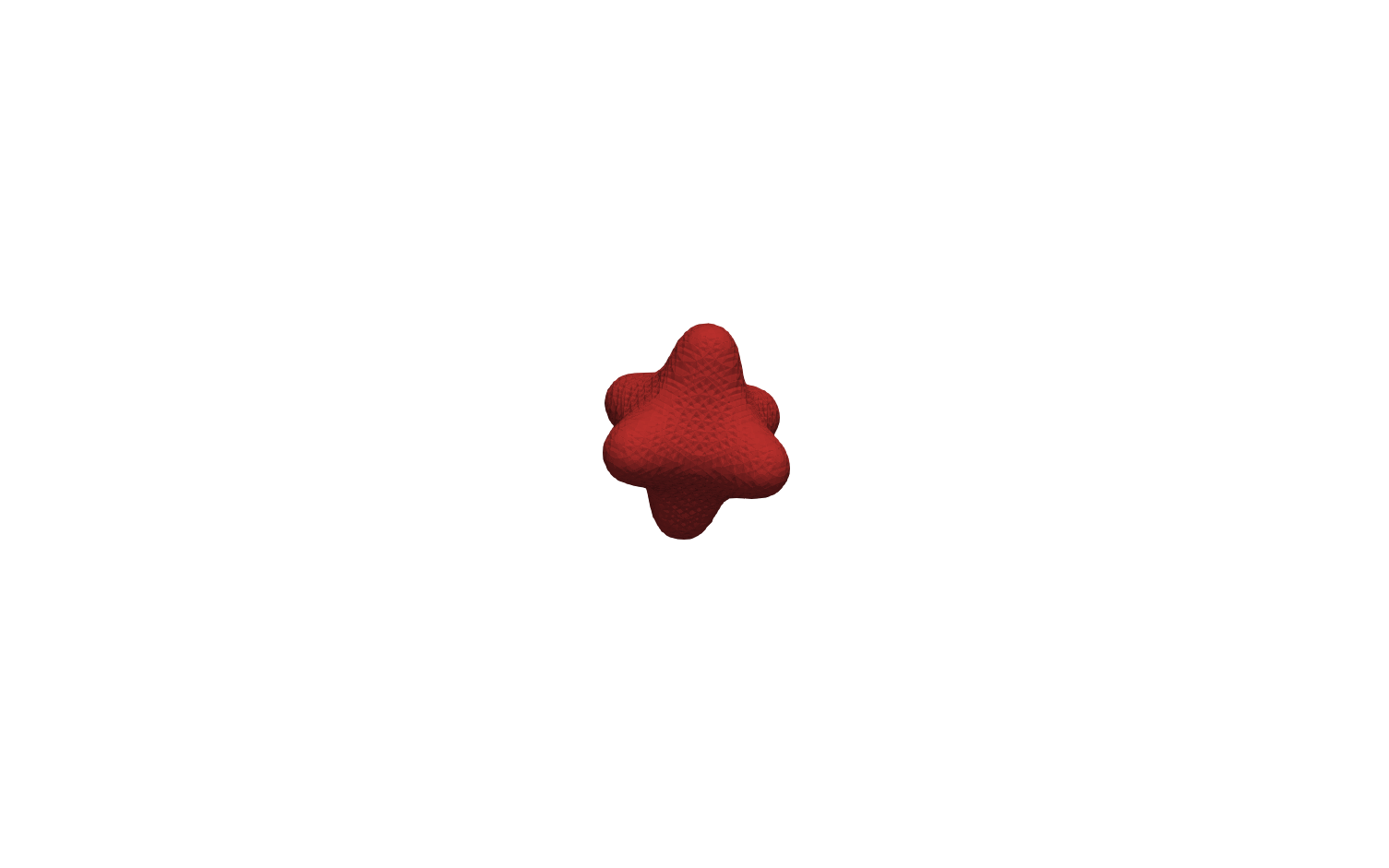}}
% 	\hspace{-1em}
	\subfloat
	{\includegraphics[width=0.32\textwidth,trim={10cm 10cm 10cm 7cm},clip]{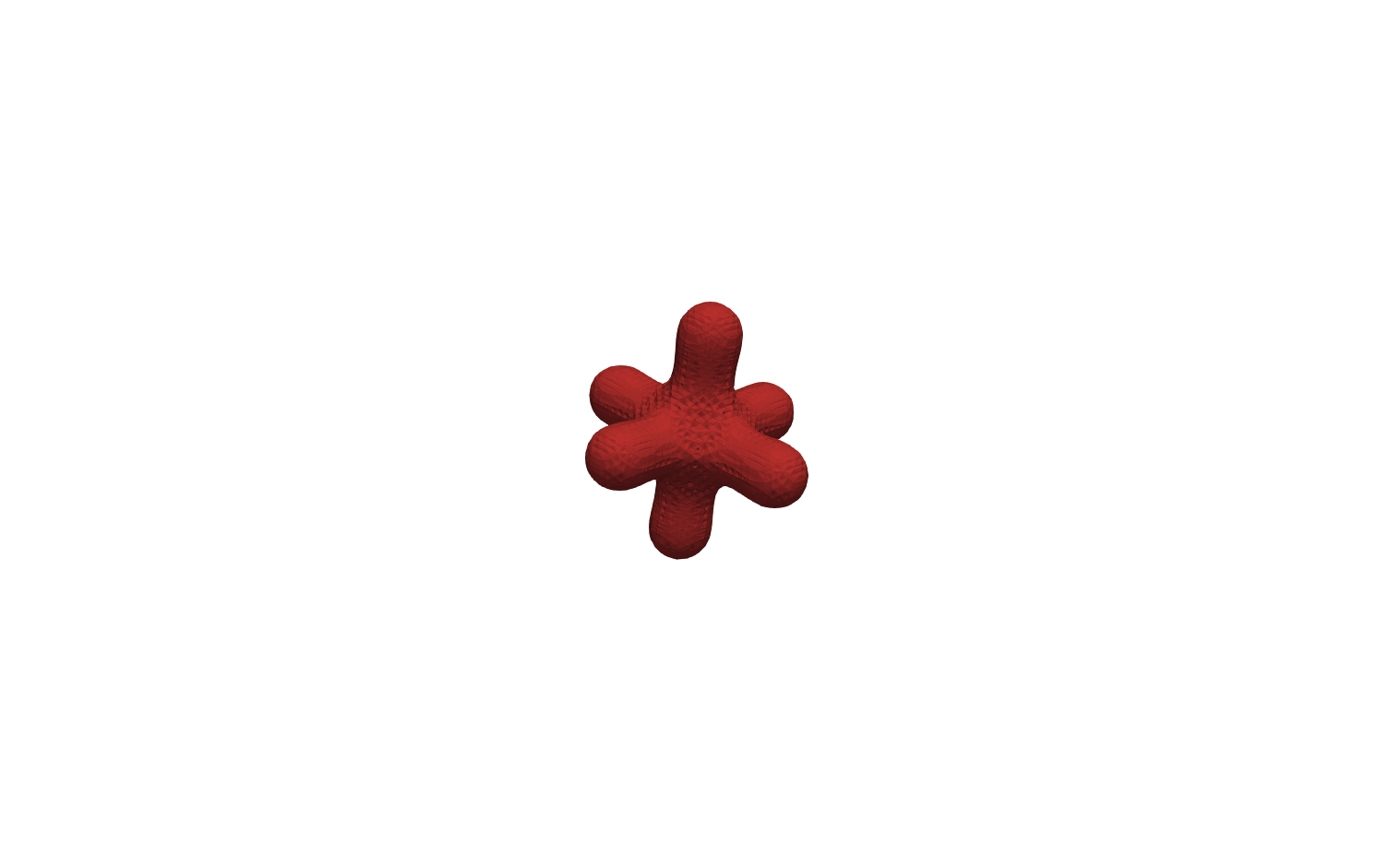}}
% 	\hspace{-1em}
	\subfloat
	{\includegraphics[width=0.32\textwidth,trim={10cm 10cm 10cm 7cm},clip]{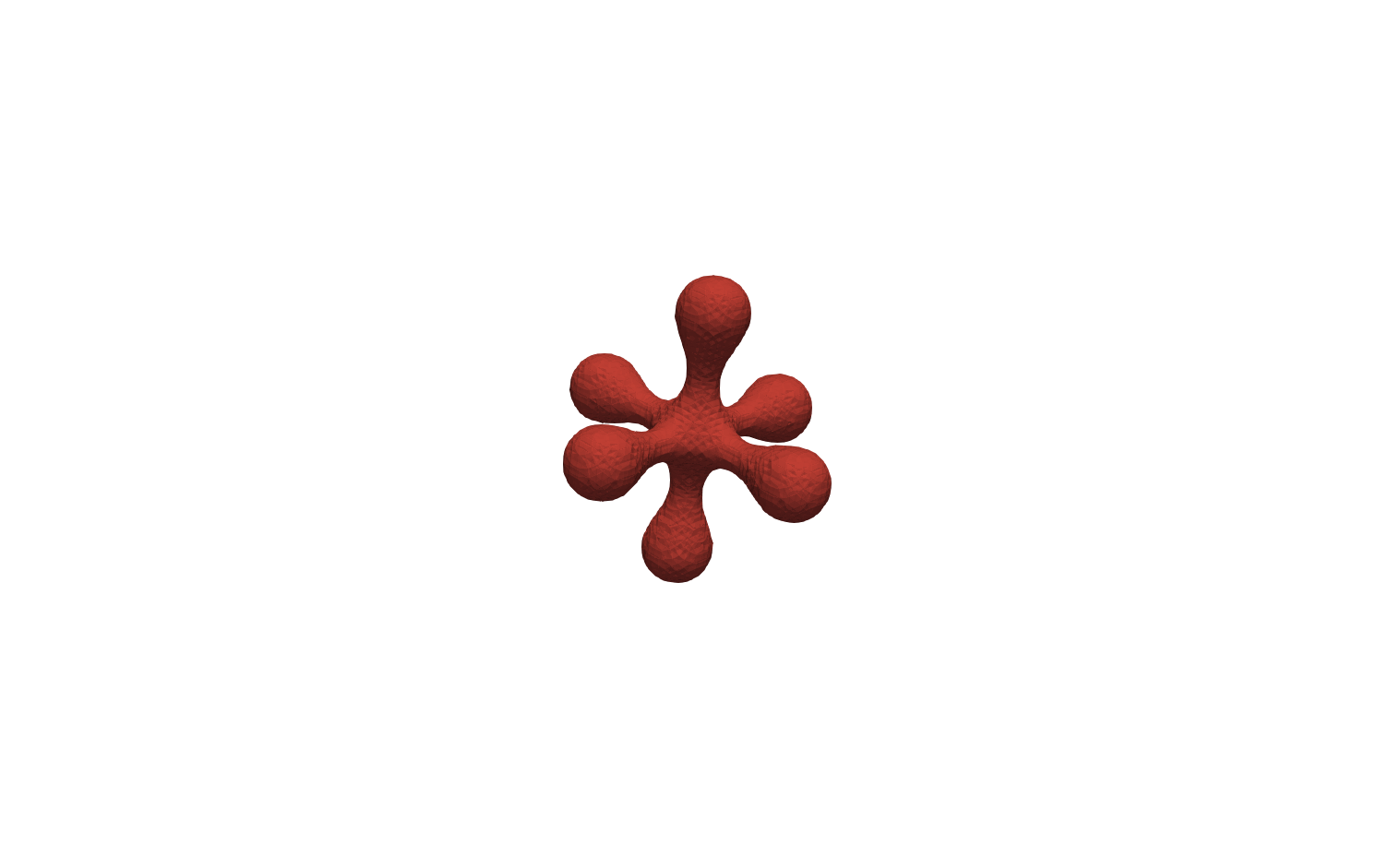}}
	\caption{Numerical solution in three dimensions with initial profile \eqref{eq:initial_3D_2} at times $t=0.1$ (left), $t=0.2$ (center) and $t=0.3$ (right).} 
	\label{fig:3D_2}
\end{figure}